\documentclass[a4paper,reqno]{amsart}

\usepackage{chemarrow}
\usepackage{comment}
\usepackage{subfiles}
\usepackage{upgreek}
\usepackage{amsfonts}
\usepackage{amsmath,amsthm,amssymb,colonequals}
\usepackage{indentfirst}
\usepackage{amscd}
\usepackage{setspace}
\usepackage{mathrsfs}
\usepackage{float}
\usepackage{graphicx, subcaption}
\usepackage{wrapfig}
\usepackage{enumitem}
\usepackage{euscript}
\usepackage{stackrel}

\usepackage[T1]{fontenc}
\usepackage[utf8]{inputenc}
\usepackage{pifont}

\setlist[enumerate]{format=\normalfont}

\usepackage{array,multirow,booktabs,longtable}
\usepackage{centernot}
\usepackage{mathtools}
\usepackage{stmaryrd}

\makeatletter
\newcommand{\xMapsto}[2][]{\ext@arrow 0599{\Mapstofill@}{#1}{#2}}
\def\Mapstofill@{\arrowfill@{\Mapstochar\Relbar}\Relbar\Rightarrow}
\makeatother
\usepackage{tikz,mathrsfs}
\usetikzlibrary{arrows,decorations.pathmorphing,decorations.pathreplacing,positioning,shapes.geometric,shapes.misc,decorations.markings,decorations.fractals,calc,patterns}

\tikzset{>=stealth',
        cvertex/.style={circle,draw=black,inner sep=1pt,outer sep=3pt},
        vertex/.style={circle,fill=black,inner sep=1pt,outer sep=3pt},
        star/.style={circle,fill=yellow,inner sep=0.75pt,outer sep=0.75pt},
        tvertex/.style={inner sep=1pt,font=\scriptsize},
        gap/.style={inner sep=0.5pt,fill=white}}

\newtheorem{theorem}{Theorem}[section]
\newtheorem{prop}[theorem]{Proposition}
\newtheorem{lemma}[theorem]{Lemma}
\newtheorem{definition}[theorem]{Definition}
\newtheorem{cor}[theorem]{Corollary}

\theoremstyle{definition}

\newtheorem{example}[theorem]{Example}

\newtheorem{remark}[theorem]{Remark}

\newtheorem{notation}[theorem]{Notation}

\numberwithin{equation}{section}

\usepackage[colorlinks]{hyperref}

\newcommand{\Jac}{\scrJ\mathrm{ac}}

\newcommand\N{\mathbb{N}}
\newcommand\Z{\mathbb{Z}}
\newcommand\Q{\mathbb{Q}}

\newcommand\C{\mathbb{C}}
\newcommand{\GV}{\textnormal{GV}}

\DeclareMathOperator{\Supp}{\mathrm{Supp}}
\renewcommand{\mod}{\mathop{\mathrm{mod}}}
\newcommand{\Spec}{\mathop{\mathrm{Spec}}}
\DeclareMathOperator{\Ext}{\mathrm{Ext}}

\DeclareMathOperator{\Hom}{\mathrm{Hom}}
\DeclareMathOperator{\End}{\mathrm{End}}
\DeclareMathOperator{\Sing}{\mathrm{Sing}}
\DeclareMathOperator{\add}{\mathrm{add}}

\renewcommand{\dim}{\mathop{\mathrm{dim}}}

\renewcommand{\deg}{\mathrm{deg}}
\renewcommand{\le}{\mathrm{len}}
\renewcommand{\l}{\textnormal{left}}
\renewcommand{\r}{\textnormal{right}}

\newcommand{\Curve}{\mathrm{C}}

\newcommand{\scrA}{\EuScript{A}}
\newcommand{\scrB}{\EuScript{B}}

\newcommand{\scrJ}{\EuScript{J}}

\newcommand{\scrL}{\EuScript{L}}

\newcommand{\scrO}{\EuScript{O}}

\newcommand{\scrR}{\EuScript{R}}
\newcommand{\scrS}{\EuScript{S}}
\newcommand{\scrT}{\EuScript{T}}
\newcommand{\scrU}{\EuScript{U}}

\newcommand{\scrX}{\EuScript{X}}
\newcommand{\scrY}{\EuScript{Y}}

\newcommand{\lcl}{(\kern -2.5pt(}
\newcommand{\rcl}{)\kern -2.5pt)}
\newcommand{\lal}{[\kern -1pt[}
\newcommand{\ral}{ ] \kern -1pt ]}
\newcommand{\lbl}{\langle \kern -2.5pt \langle}
\newcommand{\rbl}{\rangle \kern -2.5pt \rangle}
\renewcommand*\partial{\textnormal{\reflectbox{6}}}

\newcommand{\x}{\mathsf{x}}
\newcommand{\y}{\mathsf{y}}

\newcommand{\CM}{\mathrm{CM}}
\newcommand{\red}{\textnormal{red}}
\newcommand{\m}{\mathfrak{m}}
\newcommand{\llcurve}{\{\kern -3pt\{}
\newcommand{\rrcurve}{\}\kern -3pt\}}

\begin{document}

\title{Local forms for the double $A_n$ quiver}
\author{Hao Zhang}
\address{Hao Zhang, The Mathematics and Statistics Building, University of Glasgow, University Place, Glasgow, G12 8QQ, UK.}
\email{h.zhang.4@research.gla.ac.uk}
\begin{abstract}
This paper studies the noncommutative singularity theory of the double $A_n$ quiver $Q_n$ (with a single loop at each vertex), with applications to algebraic geometry and representation theory. We give various intrinsic definitions of a Type $A$ potential on $Q_n$, then via coordinate changes we (1) prove a monomialization result that expresses these potentials in a particularly nice form, (2) prove that Type $A$ potentials precisely correspond to crepant resolutions of $cA_n$ singularities, (3) solve the Realisation Conjecture of Brown--Wemyss in this setting.

For $n \leq 3$, we furthermore give a full classification of Type $A$ potentials (without loops) up to isomorphism, and those with finite-dimensional Jacobi algebras up to derived equivalence. There are various algebraic corollaries, including to certain tame algebras of quaternion type due to Erdmann, where we describe all basic algebras in the derived equivalence class.
\end{abstract}

%\subjclass[2010]{Primary 14F05; Secondary 14D15, 14E30, 16E45, 16S38, 18E30}
\maketitle

\maketitle
\setcounter{tocdepth}{1}
\tableofcontents

\parindent 0pt
\parskip 5pt
\section{Introduction}
\subsection{Motivation}
This paper is motivated by classification problems arising in the Minimal Model Program (MMP). Given a reasonable variety, the MMP seeks a nice representatives within its birational class. Such representatives are in general not unique; rather, they are related by codimension--two surgeries known as \emph{flops} \cite{BCHM}.

Let $\uppi \colon \scrX \to \Spec \scrR$ be a crepant resolution, where $\scrR$ has compound Du Val (cDV) singularities.
Associated to $\uppi$ is a noncommutative algebra $\Lambda_{\mathrm{con}}(\uppi)$, called the \emph{contraction algebra}, which represents the noncommutative deformation theory of the exceptional curves \cite{DW2}.
The contraction algebra $\Lambda_{\mathrm{con}}(\uppi)$ is isomorphic to the Jacobi algebra of a quiver with some potential \cite{VM}, and it classifies $\Spec \scrR$ complete locally if $\scrR$ is furthermore isolated \cite{JKM}.

This motivates classifying Jacobi algebras (equivalently, their potentials) on various quivers, as this immediately then classifies certain crepant resolutions. In this paper, we give various intrinsic algebraic definitions of a Type $A$ potential on the double $A_n$ quiver $Q_n$ (with a single loop at each vertex). Then via coordinate changes, we prove a monomialization result that expresses these potentials in a particularly nice form, and prove that these potentials precisely correspond to $cA_n$ crepant resolutions, which solves the Realisation Conjecture of Brown--Wemyss in Type A cases \cite{BW2}. There are further applications to representation theory. 

Taken together, our results may be viewed as a noncommutative analogue of the classical classification of simple singularities by commutative polynomials \cite{A1}, as well as a generalisation of the fact that the germ of a complex analytic hypersurface with an isolated singularity is determined by its Tjurina algebra \cite{MY}.

\subsection{Main Results}\label{mainresult}
We now summarise the main results of the paper.
For any fixed $n \geq 1$, consider the following quiver $Q_{n}$, which is the double of the usual $A_n$ quiver, with a single loop at each vertex. Label the arrows of $Q_n$ left to right, as illustrated below.
 \[
\begin{array}{c}
\begin{tikzpicture}[bend angle=15, looseness=1.2]
\node (a) at (-1.5,0) [vertex] {};
\node (b) at (0,0) [vertex] {};
\node (c) at (1.5,0) [vertex] {};
\node (c2) at (2,0) {$\hdots$};
\node (d) at (2.5,0) [vertex] {};
\node (e) at (4,0) [vertex] {};
\node (a1) at (-1.5,-0.2) {$\scriptstyle 1$};
\node (a2) at (0,-0.2) {$\scriptstyle 2$};
\node (a3) at (1.5,-0.2) {$\scriptstyle 3$};
\node (a4) at (2.5,-0.25) {$\scriptstyle n-1$};
\node (a5) at (4,-0.25) {$\scriptstyle n$};
\draw[->,bend left] (a) to node[above] {$\scriptstyle a_{2}$} (b);
\draw[<-,bend right] (a) to node[below] {$\scriptstyle b_{2}$} (b);
\draw[->,bend left] (b) to node[above] {$\scriptstyle a_{4}$} (c);
\draw[<-,bend right] (b) to node[below] {$\scriptstyle b_{4}$} (c);
\draw[->,bend left] (d) to node[above] {$\scriptstyle a_{2n-2}$} (e);
\draw[<-,bend right] (d) to node[below] {$\scriptstyle b_{2n-2}$} (e);
\draw[<-]  (a) edge [in=120,out=55,loop,looseness=10] node[above] {$\scriptstyle a_{1}$} (a);
\draw[<-]  (b) edge [in=120,out=55,loop,looseness=11] node[above] {$\scriptstyle a_{3}$} (b);
\draw[<-]  (c) edge [in=120,out=55,loop,looseness=11] node[above] {$\scriptstyle a_{5}$} (c);
\draw[<-]  (d) edge [in=120,out=55,loop,looseness=11] node[above] {$\scriptstyle a_{2n-3}$} (d);
\draw[<-]  (e) edge [in=120,out=55,loop,looseness=11] node[above] {$\scriptstyle a_{2n-1}$} (e);
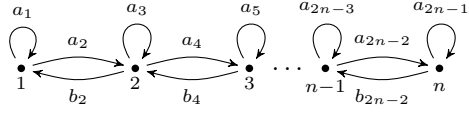
\node (z) at (1.5,-1) {Quiver $Q_{n}$ which has loop $a_{2i-1}$ at each vertex $i$.};
\end{tikzpicture}
\end{array}
\]

From this, define elements $\x_i$ and $\x'_i$ as follows: first, set $b_{2i-1}$ to be the trivial path at vertex $i$, for any $1 \leq i \leq n$. Then for any $1 \leq i \leq 2n-1$, set $\x_i \colonequals a_ib_i$ and $\x_i' \colonequals b_ia_i$.

For example, in the case $n=3$, 
 \[
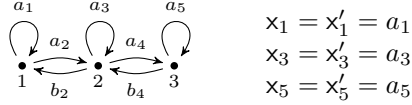
 \begin{array}{cl}
\begin{array}{c}
\begin{tikzpicture}[bend angle=15, looseness=1.2]
\node (a) at (-1,0) [vertex] {};
\node (b) at (0,0) [vertex] {};
\node (c) at (1,0) [vertex] {};

\node (a1) at (-1,-0.2) {$\scriptstyle 1$};
\node (a2) at (0,-0.2) {$\scriptstyle 2$};
\node (a3) at (1,-0.2) {$\scriptstyle 3$};

\draw[->,bend left] (a) to node[above] {$\scriptstyle a_{2}$} (b);
\draw[<-,bend right] (a) to node[below] {$\scriptstyle b_{2}$} (b);
\draw[->,bend left] (b) to node[above] {$\scriptstyle a_{4}$} (c);
\draw[<-,bend right] (b) to node[below] {$\scriptstyle b_{4}$} (c);
\draw[<-]  (a) edge [in=120,out=55,loop,looseness=11] node[above] {$\scriptstyle a_{1}$} (a);
\draw[<-]  (b) edge [in=120,out=55,loop,looseness=11] node[above] {$\scriptstyle a_{3}$} (b);
\draw[<-]  (c) edge [in=120,out=55,loop,looseness=11] node[above] {$\scriptstyle a_{5}$} (c);
\end{tikzpicture}
\end{array}
&
\begin{array}{l}
\x_1=\x_1'=a_1\\
\x_3=\x_3'=a_3\\
\x_5=\x_5'=a_5
\end{array}
\end{array}
\]
whereas $\x_2=a_2b_2$, $\x_2'=b_2a_2$, and $\x_4=a_4b_4$, $\x_4'=b_4a_4$. 

Given the above $\x_i$ and $\x_i'$, we first define a \emph{reduced Type $A$ potential} on $Q_{n}$ to be any reduced potential $f$ that contains the terms $\x_i^{\prime}\x_{i+1}$ for all $1\leq i \leq 2n-2$. A \emph{Type $A$ potential} on $Q_{n}$ is then defined in \ref{TypeA}. But for the purposes of this introduction, it suffices to work with the notion of a \emph{monomialized Type $A$ potential} on $Q_n$, defined to be any potential of the form
\begin{equation*}
    \sum_{i=1}^{2n-2}\x_i^{\prime}\x_{i+1} + \sum_{i=1}^{2n-1}\sum_{j=2}^{\infty} \upkappa_{ij} \x_i^{j} 
\end{equation*}
for some $\upkappa_{ij} \in \C$. We prove in \ref{047} and \ref{410} that any Type $A$ potential is isomorphic to a monomialized Type $A$ potential, and so the above monomialized version suffices. 
%We follow the usual definition of right-equivalence in \cite{DWZ1} (see \ref{211}). the Jacobi algebra of two right-equivalent potentials is isomorphic to each other, 

The first main result is that the complete Jacobi algebra (denoted $\Jac$) of any Type $A$ potential on $Q_n$ can be realized as the contraction algebra of a crepant resolution of some $cA_n$ singularity.  %This proves the Realisation Conjecture of Brown--Wemyss \cite{BW2} in the setting of noncommutative Type $A$ potentials. 

\begin{theorem}[\ref{511}]\label{11}
For any Type $A$ potential $f$ on $Q_n$ where $n \geq 1$, there exists a crepant resolution $\uppi \colon \scrX \rightarrow \Spec\scrR$ where $\scrR$ is $cA_n$, such that $\Jac(f) \cong \Lambda_{\mathrm{con}}(\uppi)$.
\end{theorem}

The Brown--Wemyss Realisation Conjecture \cite{BW2} predicts that if $f$ is a potential whose Jacobi algebra $\Jac(f)$ is either finite-dimensional, or infinite-dimensional with at most linear growth in the successive quotients by powers of its Jacobi ideal, then $\Jac(f)$ arises as the contraction algebra of some crepant resolution $\scrX \to \Spec \scrR$, where $\scrR$ is cDV.
Theorem~\ref{11} verifies this conjecture for all Type $A$ potentials on $Q_n$, for any $n \geq 1$.

We then obtain the converse to \ref{11} (see \ref{514}), which shows that our definition of Type A potential is intrinsic. The definition of the quiver $Q_{n,I}$
and Type $A_{n,I}$ crepant resolutions are given in  \S\ref{mono} and \ref{def: Type An}.

\begin{cor}[\ref{cor: TypeA_intrin}]
Let $f$ be a reduced potential on $Q_{n, I}$. The following are equivalent.
\begin{enumerate}
   \item $f$ is Type A.
   \item There exists a Type $A_{n, I}$ crepant resolution $\uppi$ such that $\Jac(f) \cong \Lambda_{\mathrm{con}}(\uppi)$.
   \item $e_i \Jac(f) e_i$ is commutative for $1 \leq i \leq n$.
\end{enumerate}
\end{cor}

Moreover, there is a correspondence between crepant resolutions of $cA_n$ singularities and our intrinsic noncommutative monomialized Type $A$ potentials, as follows.

\begin{cor}[\ref{516}]\label{12}
For any $n$, the set of isomorphism classes of contraction algebras associated to crepant resolutions of $cA_n$ singularities is equal to the set of isomorphism classes of Jacobi algebras of monomialized Type $A$ potentials on $Q_n$.
\end{cor}

Then, after restricting to those $cA_n$ singularities which are isolated, we obtain the following consequence.
\begin{theorem}[\ref{cor}]
For any $n$, there exists a one-to-one correspondence  
\[
\begin{tikzpicture}[bend angle=0, looseness=0]
\node[align=left] (a) at (0,-1) {derived equivalence classes of monomialized Type $A$ potentials on $Q_n$\\ with finite-dimensional Jacobi algebra.};
\node[align=left] (b) at (0,0.5) {isomorphism classes of isolated $cA_n$ singularities\\ which admit a crepant resolution};
\draw[->,bend left] (b) to node[below] {} (a);
\draw[->,bend left] (a) to node[above] {} (b);
\end{tikzpicture}
\]
\end{theorem}

\subsection{Special cases: $A_3$}
In the case of the double $A_3$ quiver without loops, it is possible to describe the full isomorphism classes of Type $A$ potentials, and the derived equivalence classes of those with finite-dimensional Jacobi algebras. This generalises \cite{DWZ1, E1, H}. 

To ease notation, now consider the following labelling.
\[
 \begin{array}{cl}
\begin{array}{c}
\begin{tikzpicture}[bend angle=30, looseness=1]
\node (a) at (-2,0) [vertex] {};
\node (b) at (0,0) [vertex] {};
\node (c) at (2,0) [vertex] {};
\node (a1) at (-2,-0.2) {$\scriptstyle 1$};
\node (a1) at (0,-0.2) {$\scriptstyle 2$};
\node (a1) at (2,-0.2) {$\scriptstyle 3$};
\draw[<-,bend right] (a) to node[below] {$\scriptstyle b_1$} (b);
\draw[->,bend left] (a) to node[above] {$\scriptstyle a_1$} (b);
\draw[<-,bend right] (b) to node[below] {$\scriptstyle b_2 $} (c);
\draw[->,bend left] (b) to node[above] {$\scriptstyle a_2$} (c);
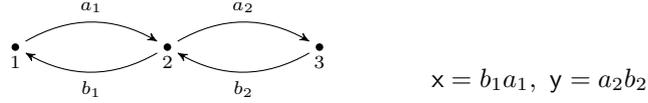
\node (z) at (0,-1.2) {Double $A_3$ quiver without loops $Q$};
\end{tikzpicture}
\end{array}
&
\begin{array}{l}
\x=b_1a_1, \ \y=a_2b_2
\end{array}
\end{array}
\]

Given two potentials $f$ and $g$ on $Q$, we say that $f$ is \emph{isomorphic} to $g$, written $f \cong g$, if the corresponding Jacobi algebras are isomorphic (see \ref{211}).  Similarly, we say that $f$ is \emph{derived equivalent} to $g$, written $f \simeq g$, if the corresponding Jacobi algebras are derived equivalent (see \ref{notation:derived}).

\begin{theorem}[\ref{610}]\label{117}
Any Type A potential on $Q$ must be isomorphic to an algebra in one of the following isomorphism classes of potentials:
\begin{enumerate}
    \item $\x^2+\x\y+\uplambda\y^2$ for any $\uplambda \in \mathbb{C} \setminus \{0,\tfrac14\}$.
        \item $\x^2+\x\y+\frac{1}{4}\y^2+ \x^r$ for any $r \geq 3$.
        \item $\x^p+\x\y+\y^q \cong \x^q+\x\y+\y^p$ for any $(p,q) \neq (2,2)$.
        \item $\x^2+\x\y+\frac{1}{4}\y^2$.
        \item $\x^p+\x\y \cong \x\y+\y^p$ for any $p\geq 2$.
        \item $\x\y$.
\end{enumerate}
The Jacobi algebras of these potentials are all mutually non-isomorphic (except those isomorphisms stated), and in particular the Jacobi algebras with different parameters in the same item are non-isomorphic. 

The Jacobi algebras in \textnormal{(1)}, \textnormal{(2)}, \textnormal{(3)} are realized by crepant resolutions of isolated $cA_3$ singularities, and those in \textnormal{(4)}, \textnormal{(5)}, \textnormal{(6)} are realized by crepant resolutions of non-isolated $cA_3$ singularities.
\end{theorem}

\begin{theorem}[\ref{611}]\label{116}
The following groups the Type $A$ potentials on $Q$ with finite-dimensional Jacobi algebra into sets, where all the Jacobi algebras in a given set are derived equivalent.
\begin{enumerate}
\item $\{\x^2+\x\y+\uplambda' \y^2$ $\mid$ $\uplambda' = \uplambda,\frac{1-4\uplambda}{4},\frac{1}{4(1-4\uplambda)},\frac{\uplambda}{4\uplambda-1},\frac{4\uplambda-1}{16\uplambda},\frac{1}{16\uplambda}\}$ for any $\uplambda \in \mathbb{C} \setminus \{0,\tfrac14\}$.
\item $\{\x^p+\x\y+\y^2$, $\x^2+\x\y+\y^p$, $\x^2+\x\y+\frac{1}{4}\y^2+ \x^p\}$ for $p \geq 3$.
\item $\{\x^p+\x\y+\y^q$, $\x^q+\x\y+\y^p \}$ for $p \geq 3$ and $q \geq 3$.    
\end{enumerate}
Moreover, the Jacobi algebras of the sets in \textnormal{(1)--(3)} are all mutually not derived equivalent, and in particular the Jacobi algebras of different sets in the same item are not derived equivalent. 
In \textnormal{(1)} there are no further basic algebras in the derived equivalence class, whereas in \textnormal{(2)--(3)} there are an additional finite number of basic algebras in the derived equivalence class.
\end{theorem}

Next, recall the definition of the quaternion type quiver algebra $A_{p,q}(\upmu)$ in \cite{E1, H}, which is the completion of the path algebra of the quiver $Q$ modulo the relations
\begin{equation*}
 a_1a_2b_2 -(a_1b_1)^{p-1}a_1, b_2b_1a_1-\upmu (b_2a_2)^{q-1}b_2, a_2b_2b_1-(b_1a_1)^{p-1}b_1, b_1a_1a_2-\upmu(a_2b_2)^{q-1}a_2,
\end{equation*}
where $\upmu \in \C$ and $p,q \geq 2$. Note we have fewer relations than in \cite{E1, H} since we are working with the completion. In fact $A_{p,q}(\upmu) \cong \Jac(Q,f)$, where 
\begin{equation*}
    f=\frac{1}{p}\x^p-\x\y+\frac{\upmu}{q}\y^q \cong \x^p+\x\y+(-1)^qp^{-\frac{q}{p}}q^{-1}\upmu\y^q.
\end{equation*}

The following improves various results of Erdmann and Holm \cite{E1, H}.
\begin{cor}[\ref{631}]\label{cor:Apq}
The following groups those algebras $A_{p,q}(\upmu)$ which are finite-dimensional into sets, where all the algebras in a given set are derived equivalent.
\begin{enumerate}
\item $\{A_{2,2}(\upmu')$ $\mid$ $\upmu' = \upmu,1-\upmu, \frac{1}{1-\upmu}, \frac{\upmu}{\upmu-1}, \frac{\upmu-1}{\upmu}, \frac{1}{\upmu}\}$ for $\upmu \in \C \setminus \{0,1\}$.
\item $\{A_{p,q}(1)$, $A_{q,p}(1)\}$ for $(p,q)  \neq (2,2)$.
\end{enumerate}
Moreover, the algebras in different sets in \textnormal{(1)--(2)} are all mutually not derived equivalent.  
In \textnormal{(1)} there are no further basic algebras in the derived equivalence class, whereas in \textnormal{(2)} there are an additional finite number of basic algebras in the derived equivalence class.
\end{cor}

Among the families of potentials in \ref{116}, the quadratic case in \textnormal{(1)} exhibits a particularly rich structure, with derived equivalence classes given by finite sets of parameters related by explicit transformations.
Corollary \ref{cor:Apq} presents the same family via quaternion type algebras $A_{2,2}(\upmu)$, making the underlying $\mathfrak{S}_3$-symmetry explicit.
In \S\ref{subsection:S3}, we give a geometric explanation of this symmetry in terms of flops of crepant resolutions of $cA_3$ singularities, and relate it to the classical $\mathfrak{S}_3$-action on the Legendre parameter of elliptic curves.

\subsection*{Conventions}\label{con}
Throughout this paper, we work over the complex numbers $\mathbb{C}$, which is necessary for various geometric statements in \S \ref{geometry}. 
The definitions of $Q_{n, I}$ and $\x_i$ are fundamental, and are repeated in \S\ref{mono}. 
We also adopt the following notation.
\begin{enumerate}
\item Throughout, $n$ is the number of vertices in the quiver $Q_{n, I}$, and $I \subseteq \{1,2, \dots, n\}$ is the set of vertices without loops in $Q_{n, I}$. 
\item Set $m\colonequals 2n-1-|I|$, which equals the number of $\x_i$ in $Q_{n, I}$. 
\item Vector space dimension will be written $\dim_{\C}V$.
\item We write $\mathbb{C}^{\times}$ for the multiplicative group of nonzero complex numbers.
\end{enumerate}

\subsection*{Acknowledgements}
This work forms part of the author’s PhD at the University of Glasgow, funded by the China Scholarship Council.
The author would like to thank his supervisor, Michael Wemyss, for valuable guidance and many helpful discussions.

\section{Algebraic Preliminaries}

To set notation, let $Q=(Q_0,Q_1,t,h)$ be a \emph{quiver}, where $Q_0$ is a finite set of vertices, $Q_1$ is a finite set of arrows, and $t,h\colon Q_1\to Q_0$ are the tail and head maps.
A \emph{loop} $a$ is an arrow satisfying $h(a) = t(a)$, and
a \emph{path} is a formal expression $a_1a_2  \dots a_n$ where $h(a_i) = t(a_{i+1})$ for all $1 \leq i \leq n-1$. A path $a$ is \emph{cyclic} if $h(a) = t(a)$.

Let $k$ be a field. The \emph{complete path algebra} $k\lbl Q\rbl$ is the completion of the usual path algebra $kQ$ with respect to the arrow ideal. That is, the elements of $k\lbl Q\rbl$ are possibly infinite $k$-linear combinations of paths in $Q$.

Write $\mathfrak{m}_Q$, or simply $\mathfrak{m}$, for the two-sided ideal of $k\lbl Q\rbl$ generated by the elements of $Q_1$, and write $A_Q$, or simply $A$, for the $k$-span of the elements of $Q_1$.

\begin{definition}\label{QP}
Suppose that Q is a quiver with arrow span $A$.
\begin{enumerate}
\item A \emph{relation} of $Q$ is a $k$-linear combination of paths in $Q$, each with the same head and tail. 

\item Given finitely many relations $R_1,\dots,R_n$, let $R$ be the closure in $k\lbl Q\rbl$ of the two-sided ideal $kQ R_1 kQ + \cdots + kQ R_n kQ$.
We call $(Q,R)$ a \emph{quiver with relations}, and $k\lbl Q\rbl/R$ its complete path algebra.

\item  A \emph{quiver with potential} $($\emph{QP} for short$)$ is a pair $(Q, W)$, where $W$ is a $k$-linear combination of cyclic paths in $Q$. 

\item For any $n \geq 1$, set $W_n$ to be the $nth$ homogeneous component of $W$ with respect to the path length.  
\item For each $a \in Q_1$ and cyclic path $a_1 \dots a_d$ in $Q$, define the \emph{cyclic derivative} as
\begin{equation*}
\partial_{a}\left(a_{1} \ldots a_{d}\right)=\sum_{i=1}^{d} \delta_{a, a_{i}} a_{i+1} \ldots a_{d} a_{1} \ldots a_{i-1}
\end{equation*}
(where $\delta_{a, a_{i}}$ is the Kronecker delta), and then extend $\partial_{a}$ by linearity.

\item The \emph{Jacobi ideal} $J(W)$ is the closure of the two-sided ideal of $k\lbl Q\rbl$ generated by $\partial_a W$ for all $a\in Q_1$.

\item The \emph{Jacobi algebra} $\Jac(Q,W)$ or $\Jac(A,W)$ is the quotient $k\lbl Q\rbl/J(W)$. We write $\Jac(W)$ when the quiver $Q$ is obvious.
\item For every potential $W$, write $\partial W$ for the $k$-span of $\partial_a W$ for all $a \in Q_1$.
\item We call a QP $(Q,W)$ \emph{reduced} if $W_2=0$. It is called \emph{trivial} if $W_n=0$ for all $n \geq 3$, and further $\partial W =A$.
\begin{comment}
 \item We define the \emph{trivial and reduced arrow spans} of a QP $(Q, W)$ as
\begin{equation*}
    A_{triv} = \partial W_2, \quad A_{red} = A/\partial W_2.
\end{equation*}   
\end{comment}
\end{enumerate}
\end{definition}

\begin{example}
Consider the one loop quiver $Q$ with potential $W=a^2$,
\[
\begin{tikzpicture}[bend angle=15, looseness=1]
\node (d) at (2,0) [vertex] {};
 \draw[->]  (d) edge [in=55,out=120,loop,looseness=12] node[above] {$\scriptstyle a$}  (d);
\end{tikzpicture}
\]
The complete path algebra $k\lbl Q\rbl$ is $k\lal a\ral$. In this case, $\Jac(Q,W) \cong k\lal a \ral/(a) \cong k$ since $\partial_{a}(a^2)=2a$. Since $W_n=0$ for all $n \geq 3$ and $\partial W = ka=A_{Q}$, this $QP$ $(Q,W)$ is trivial. 
\end{example}

\begin{notation}
For $\scrA \colonequals k\lbl  Q \rbl $, let $\{\scrA,\scrA\}$ denote the commutator subspace of $\scrA$, consisting of finite sums
\[
\sum_{i=1}^n k_i(p_iq_i-q_ip_i),
\]
with $p_i,q_i\in\scrA$ and $k_i\in k$.
Write $\llcurve \scrA,\scrA\rrcurve$ for the closure of the commutator vector space $\{ \scrA,\scrA\}$. 
\end{notation}

\begin{definition}
Two potentials $W$ and $W^{\prime}$ are \emph{cyclically equivalent} (written $W \sim W^{\prime}$) if $W - W^{\prime} \in \llcurve \scrA,\scrA\rrcurve$. We write $W \stackrel{i}{\sim} W'$ if $W \sim W'$ and $W-W' \in \mathfrak{m}^i$.
\begin{comment}
Given an ideal $\mathfrak{n}$ in $k\lbl Q\rbl$, write $W \stackbin[\mathfrak{n}]{i}{\sim} W'$ if $W \sim W'$ and $W-W' \in \mathfrak{n}^i$. When furthermore $\mathfrak{n}$ is $\mathfrak{m}_Q$, we write $W \stackrel{i}{\sim} W'$ for short.   
\end{comment}

\end{definition}

\begin{remark}\label{024}
Note that if two potentials $W$ and $W^{\prime}$ are cyclically equivalent, then $\partial_{a}W = \partial_{a}W^{\prime}$ for all $a \in Q_1$, and hence $\Jac(Q,W) = \Jac(Q,W^{\prime})$ \cite[3.3]{DWZ1}. Since we aim to classify the Jacobi algebras up to isomorphism, we always consider the potentials up to cyclic equivalence. 
\end{remark}

Given an algebra homomorphism $\varphi \colon k\lbl Q\rbl \rightarrow k\lbl Q^{\prime}\rbl$ such that $\varphi|_k= id$ which sends $\mathfrak{m}_Q$ to $\mathfrak{m}_{Q^{\prime}}$, write $\left.\varphi\right|_{A_Q}=\left(\varphi_1, \varphi_2 \right)$ where $\varphi_1 \colon A_Q \rightarrow A_{Q'}$ and $\varphi_2 \colon A_Q \rightarrow \mathfrak{m}_{Q'}^2$ are $k$-module homomorphisms.

\begin{prop} \cite[2.4]{DWZ1} \label{thm28}
Given two quivers $Q$ and $Q'$, any pair $(\varphi_1,\varphi_2)$ of k-module homomorphisms $\varphi_1 \colon A_Q \rightarrow A_{Q'}$ and $\varphi_2 \colon A_Q \rightarrow \mathfrak{m}_{Q'}^2$ gives rise to a unique homomorphism of algebras $\varphi \colon k\lbl Q\rbl \rightarrow k\lbl Q^{\prime}\rbl$ such that $\varphi|_k= id$ and $\left.\varphi\right|_{A_Q}=\left(\varphi_1, \varphi_2\right)$.  Furthermore, $\varphi$ is an isomorphism if and only if $\varphi_1$ is a $k$-module isomorphism $A_Q \rightarrow A_{Q'}$.
\end{prop}

From the above result, whenever we construct an automorphism $\varphi \colon k\lbl Q\rbl \rightarrow k\lbl Q\rbl$ in \S \ref{mono} and \S \ref{section:A3}, it will always be the case that $\varphi|_k= id$, so we will only describe $\varphi|_{A_Q}$.

\begin{definition}
An algebra homomorphism $\varphi \colon k\lbl Q\rbl \rightarrow k\lbl Q\rbl$ is called a \emph{unitriangular automorphism} if $\varphi|_k= id$ and $\varphi_1=id$. For $i \geq 1$, we say that $\varphi$ has \emph{depth $i$} provided that $\varphi_2(a) \in \mathfrak{m}_Q^{i+1}$ for all $a \in Q_1$.
\end{definition}

\begin{definition}\label{211}
Let $f$ and $g$ be potentials on a quiver $Q$.
\begin{enumerate}
\item We say that $f$ is \emph{isomorphic} to $g$ (written $f \cong g$) if their Jacobi algebras $\Jac(f) \cong \Jac(g)$ as algebras.
%\item We say that $f$ is \emph{derived equivalent} to $g$ (written $f \simeq g$) if $\mathrm{D}^{\mathrm{b}}(\Jac(f)) \cong \mathrm{D}^{\mathrm{b}}(\Jac(g))$ as triangulated categories.

\item If there exists an algebra isomorphism $\varphi\colon k\lbl Q\rbl \rightarrow k\lbl Q \rbl$ such that $\varphi|_k= id$ and $\varphi(f)=g$, then
we write $\varphi \colon f \mapsto g$ and say that $f$ is \emph{equivalent} to $g$.

\item If there exists an algebra isomorphism $\varphi\colon k\lbl Q\rbl \rightarrow k\lbl Q \rbl$ such that $\varphi|_k= id$ and $\varphi(f) \sim g$, then we write $\varphi \colon f \rightsquigarrow g$ and say that $f$ is \emph{right-equivalent} to $g$.

\item For $i \geq 1$, if there exists a unitriangular $\varphi\colon k\lbl Q\rbl \rightarrow k\lbl Q \rbl$ such that $\varphi$ has depth greater than or equal to $i$, and further $\varphi(f) \stackrel{i+1}{\sim} g$, then we write $\varphi \colon f \stackrel{i}{\rightsquigarrow} g$ and say that $f$ is \emph{path degree $i$ right-equivalent} to $g$.  
\end{enumerate}
\end{definition}

We follow the definition of right-equivalence in \cite[4.2]{DWZ1}. Moreover, from \cite[p12]{DWZ1}, $f \rightsquigarrow g$ induces $f \cong g$, and further a finite sequence of right-equivalences is still a right-equivalence. By \ref{thm28}, $f \stackrel{i}{\rightsquigarrow} g$ induces $f \rightsquigarrow g$.
Thus, together with the above definition, we obtain
\[
f \sim g \quad \text{or} \quad f \mapsto g \quad \text{or} \quad 
f \stackrel{i}{\rightsquigarrow} g \;\; \Longrightarrow \;\; f \rightsquigarrow g \;\; \Longrightarrow \;\; f \cong g .
\]

The Jacobi algebra isomorphism $\cong$ is the equivalence relation that we aim to classify the potentials up to. 
The strategy is to start with a potential $f$, then transform it by a sequence of automorphisms which chase terms into higher and higher degrees. Composing this sequence of automorphisms then gives a single automorphism which takes $f$ to the desired form (see \S\ref{mono} and \S\ref{A3_norm}). 

The subtle point is that at each stage, the automorphism only gives the desired potential up to cyclic equivalence (e.g. $\rightsquigarrow$, $\stackrel{i}{\rightsquigarrow}$). 
Given an infinite sequence of path degree $i$ right-equivalences $\varphi_i \colon f_i \stackrel{i}{\rightsquigarrow} f_{i+1}$ for $i \geq 1$, then the following asserts that $\lim f_i$ exists, and further there exists a right-equivalence $F \colon f_1 \rightsquigarrow \lim f_i$. 

\begin{theorem} \cite[2.9]{BW2} \label{31}
Let $f$ be a potential, and set $f_1=f$. Suppose  that there exist elements $f_2$, $f_3$, \dots and automorphisms $\varphi_1$, $\varphi_2$, \dots, such that
\begin{enumerate}
    \item Every $\varphi_i$ is unitriangular of depth of $\geq i$, and
    \item $\varphi_i(f_i) \stackrel{i+1}{\sim} f_{i+1}$, for all $i \geq 1$.
\end{enumerate}
Then $\lim f_i$ exists, and there exists an automorphism $F$ such that $F(f) \sim \lim f_i$.
\end{theorem}

\section{Geometric Preliminaries}\label{geometry}
In this section we recall basic concepts, including NCCRs and contraction algebras, and then summarise some facts about $cA_n$ singularities and their NCCRs from \cite{IW2}.

Throughout the remainder of the paper, we reserve the notation $\scrR$ for complete local $\C$-algebras of the following form.
\begin{definition}
A complete local $\mathbb{C}$-algebra $\scrR$ is called a \emph{compound Du Val (cDV) singularity} if
\begin{equation*}
\scrR \cong \frac{\mathbb{C} \lal u, v, x, t \ral}{f+t g}
\end{equation*}
where $f \in \mathbb{C} \lal u, v, x \ral$ defines a Du Val, or equivalently Kleinian, surface singularity and $g \in \mathbb{C} \lal u, v, x, t \ral$ is arbitrary.
\end{definition}

\begin{definition}
A projective birational morphism $\uppi \colon \scrX \rightarrow \Spec \scrR$ is called a \emph{crepant partial resolution} if $\omega_{\scrX} \cong \uppi^{*}\omega_{\scrR}$.  When $\scrX$ is furthermore smooth, we call $\uppi$ a \emph{crepant resolution}. 
If $\scrR$ is isolated, crepant partial resolutions and crepant resolutions are equivalently called \emph{flopping contractions} and \emph{smooth flopping contractions}, respectively $($see e.g. \cite[\S 1]{R}$)$.
\end{definition}

\subsection{NCCRs and Contraction Algebras}
In this subsection, we will introduce NCCRs and contraction algebras of cDV singularities and then recall some associated theorems.

We will write $\CM \, \scrR$ for the category of maximal Cohen--Macaulay $\scrR$-modules and $\underline{\CM} \, \scrR$ for the stable category of $\CM \, \scrR$.
\begin{definition} \label{thm32}
A \emph{noncommutative crepant resolution (NCCR)} of $\scrR$ is a ring of the form $\Lambda(M) \colonequals \End_{\scrR}(M)$ for some finitely generated reflexive $\scrR$-module M, such that $\Lambda(M) \in \CM \, \scrR$ and has finite global dimension.
\end{definition}
Since $\scrR$ is cDV, there is a bijection between NCCRs of $\scrR$ and crepant resolutions of
$\Spec\scrR$; see \cite{W1}:
\[
\left\{\, M\in \CM\,\scrR \ \middle|\ \End_{\scrR}(M)\ \text{is an NCCR}\,\right\}
\ \longleftrightarrow\
\left\{\,\uppi\colon \scrX \to \Spec\scrR \ \text{a crepant resolution}\,\right\}.
\]
\begin{comment}
Since $\scrR$ is cDV, there is a bijection \cite{W1},
\begin{equation*}
  \left\{  M \in \mathrm{CM} \, \scrR \textnormal{ and satisfies \ref{thm32}}  \right\}  \longleftrightarrow\left\{\text { crepant resolutions } \uppi: \scrX \rightarrow \Spec  \scrR \right\}.
\end{equation*}    
\end{comment}
The passage from left to right takes a given $M$ and associates a certain moduli space of representations of $\End_{\scrR}(M)$. In particular, passing from a crepant resolution to the corresponding NCCR retains the geometric
information relevant for our purposes.

We next explain the passage from right to left in detail. 
Let $\uppi \colon \scrX \rightarrow \Spec \scrR$ be a crepant resolution with exceptional curves $\Curve_1,\dots, \Curve_n$. For any $1 \leq i \leq n$, there is a unique bundle $\mathcal{M}_i$ on $\scrX$ \cite[3.5.4]{V1},
and 
\begin{equation*}
    \mathcal{M} \colonequals \scrO_{\scrX} \oplus \bigoplus_{i=1}^n \mathcal{M}_i
\end{equation*}
is a tilting bundle on $\scrX$ \cite[3.5.5]{V1}. Pushing forward via $\uppi$ gives $\uppi_{*}(\scrO_{\scrX})= \scrR$ and $\uppi_{*}(\mathcal{M}_i)= M_i$ for some $\scrR$-module $M_i$. Set $M= \scrR \oplus \bigoplus_{i=1}^n M_i$. Then $M \in \CM \, \scrR$ and $\End_{\scrR}(M)$ is an NCCR. Thus, there is an equivalent definition of NCCR associated to the crepant resolution $\uppi$,
\begin{equation*}
    \Lambda(\uppi) \colonequals \End_{\scrX}(\mathcal{M})  \cong \End_{\scrR}\left(M\right)=\Lambda(M),
\end{equation*}
where the isomorphism follows from the crepancy of $\uppi$; see \cite[3.2.10]{V1}.

The contraction algebra can be defined as a quotient of the NCCR as follows.
\begin{definition}
Define the \emph{contraction algebra} associated to a crepant resolution $\uppi$ to be the stable endomorphism algebra
\[
\Lambda_{\mathrm{con}}(\uppi)\colonequals \underline{\End}_{\scrR}(M)
=\End_{\scrR}(M)/\langle \scrR\rangle,
\]
where $\langle \scrR \rangle$ denotes the two-sided ideal consisting of all morphisms which factor through $\add \scrR$. We also write $\Lambda_{\mathrm{con}}(M)$ for $\Lambda_{\mathrm{con}}(\uppi)$.
\end{definition}

The difference between flopping contractions and divisor-to-curve contractions can be detected by the finite dimensionality (or otherwise) of the contraction algebra as follows.
\begin{theorem}\label{352}
Suppose that $\uppi \colon \scrX \rightarrow \Spec \scrR$ is a crepant partial resolution, and write $Z$ for the locus in $\Spec \scrR$ for which $\uppi$ is not an isomorphism. Then $\operatorname{Supp}_{\scrR}\Lambda_{\mathrm{con}}(\uppi)= Z$, and 
\begin{equation*}
    \uppi  \text{ is a flopping contraction}  \iff  \operatorname{dim}_{\mathbb{C}}\Lambda_{\mathrm{con}}(\uppi) < \infty.
\end{equation*}
If moreover $\scrX$ is smooth, then these conditions are equivalent to $\scrR$ being an isolated singularity.
\end{theorem} 

\begin{proof}
The equivalences preceding the final assertion are established in \cite[4.8]{DW1}, so we only justify the case where $\scrX$ is smooth. 
If $\scrX$ is smooth, then the locus over which $\uppi$ fails to be an isomorphism coincides with the singular locus of $\scrR$, that is, $Z= \Sing \scrR$.  Hence $\Supp_{\scrR}\Lambda_{\mathrm{con}}(\uppi)= \Sing \scrR$. 

Since $\scrR$ is complete local and $\Lambda_{\mathrm{con}}(\uppi)$ is finitely generated over $\scrR$, it follows from \cite[2.13, 2.15, 2.17]{E3} that
\[
\operatorname{dim}_{\mathbb{C}}\Lambda_{\mathrm{con}}(\uppi) < \infty \iff \Supp_{\scrR}\Lambda_{\mathrm{con}}(\uppi)= V(\m),
\]
where $\m$ denotes the maximal ideal of $\scrR$ and $V(\m)\subseteq \Spec \scrR$ is the corresponding Zariski closed point.
Therefore, $\dim_{\mathbb{C}}\Lambda_{\mathrm{con}}(\uppi) < \infty  \iff \Sing\scrR=V(\m)$, which is equivalent to $\scrR$ having an isolated singularity.
\end{proof}

Donovan and Wemyss conjectured that the contraction algebra distinguishes the analytic type of the flop \cite[1.4]{DW2}, which has been proved as follows. 
\begin{theorem} \cite[A.2]{JKM}\label{037}
Let $\uppi_i \colon \scrX_i \rightarrow \Spec \scrR_i$ be smooth flopping contractions of isolated \textnormal{cDV} $\scrR_i$ for $i=1,2$. Then $\Lambda_{\mathrm{con}}(\uppi_1)$ and $\Lambda_{\mathrm{con}}(\uppi_2)$ are derived equivalent if and only if the singularities $\scrR_1$ and $\scrR_2$ are isomorphic.  
\end{theorem}

The following result connects the derived equivalence of contraction algebras with the flops in geometry.  
\begin{theorem}\cite[5.2.2]{A2}\label{thm:august}
Given a flopping contraction $\uppi \colon \scrX \to \Spec\scrR$ where $\scrR$ is isolated \textnormal{cDV}, then the basic algebras derived equivalent to $\Lambda_{\mathrm{con}}(\uppi)$ are precisely those $\Lambda_{\mathrm{con}}(\uppi')$ where $\uppi'$ is obtained by a sequence of iterated flops from $\uppi$. In particular, there are finitely many such algebras.
\end{theorem}

\subsection{$cA_n$ singularities}
This subsection summarises several facts about $cA_n$ singularities and their NCCRs in \cite{IW1}.

Every $cA_{t-1}$ singularity $\scrR$ has the following form 
\begin{equation*}
    \scrR \cong \frac{\mathbb{C} \lal u, v, x, y \ral}{uv-g_0g_1 \dots g_{n}},
\end{equation*}
where $t$ is the order of the power series $g_0g_1 \dots g_{n}$ and each $g_i$ is a prime element of $\mathbb{C}\lal x,y \ral$. Moreover, $\scrR$ admits a crepant resolution if and only if each $g_i$ has a linear term by e.g. \cite[5.1]{IW1}. 

In this subsection we restrict to those $\scrR$ that admit a crepant resolution. In this case each $g_i$ has a linear term, and hence the order $t$ equals the number of factors, so $t=n+1$ and $\scrR$ is a $cA_n$ singularity.
Consider the $\CM$ $\scrR$-module 
\[
M\colonequals \scrR \oplus (u,g_{0}) \oplus (u,g_{0}g_{1}) \oplus \ldots \oplus (u,\prod_{i=0}^{n-1} g_{i}),
\]
and let $\uppi\colon \scrX \to\Spec\scrR$ be the associated crepant resolution in \ref{36} below.
We fix $\scrR$, $M$ and $\uppi$ throughout this subsection. 
\begin{notation}\label{notation:cAn}
We adopt the following notation.
\begin{enumerate}\label{notation: cAn}
\item Consider the symmetric group $\mathfrak{S}_{n+1}$. For any $\sigma \in \mathfrak{S}_{n+1}$, set 
\begin{equation*}
M^{\sigma} \colonequals \scrR \oplus (u,g_{\sigma(0)}) \oplus (u,g_{\sigma(0)}g_{\sigma(1)}) \oplus \ldots \oplus (u,\prod_{i=0}^{n-1} g_{\sigma(i)}) \in \CM \, \scrR.      
\end{equation*}
\item Write $\uppi^{\sigma}\colon \scrX^{\sigma} \rightarrow \Spec\scrR$ for the associated crepant resolution of $M^{\sigma}$ in \ref{36} below.
\item Now let $k \geq 1$ and consider the $k$-tuple $\mathbf{r}= (r_1,\dots , r_k)$ with each $1 \leq r_i \leq n$. Set
\begin{equation*}
    \sigma(\mathbf{r}) \colonequals (r_k,r_k+1) \cdots (r_2,r_2+1)(r_1,r_1+1) \in \mathfrak{S}_{n+1},
\end{equation*}
and $M^{\mathbf{r}}\colonequals M^{\sigma(\mathbf{r})}$. Write $\uppi^{\mathbf{r}}\colon \scrX^{\mathbf{r}} \rightarrow \Spec\scrR$ for $\uppi^{\sigma(\mathbf{r})}\colon \scrX^{\sigma(\mathbf{r})} \rightarrow \Spec\scrR$.
\item For $1\leq i \leq n$, write $\uppi^i$, $\scrX^i$ and $M^i$ for $\uppi^{(i)}$, $\scrX^{(i)}$ and $M^{(i)}$ respectively.
\end{enumerate}
\end{notation}

The following two results describe NCCRs of crepant resolutions of $cA_n$ singularities.

\begin{prop} \cite[5.1, 5.27]{IW1}\label{36}
The $\CM $ $\scrR$-modules that satisfy \textnormal{\ref{thm32}} are precisely $M^{\sigma}$ where $\sigma \in \mathfrak{S}_{n+1}$. Moreover, there is a bijection satisfying $\Lambda(\uppi^{\sigma}) \cong \End_{\scrR}(M^{\sigma})$,
\begin{align*}
 & \left\{ M^{\sigma} \mid  \sigma \in \mathfrak{S}_{n+1}\right\}  \longleftrightarrow\left\{\text { crepant resolutions of } \scrR \right\},\\
 & \qquad \qquad \qquad  M^{\sigma} \longleftrightarrow \uppi^{\sigma}: \scrX^{\sigma} \rightarrow \Spec  \scrR,
\end{align*}
where $\scrX^{\sigma}$ is given pictorially by
\[
\begin{array}{ccc}
\begin{array}{c}
\scrX^{\sigma}
\end{array} &
\begin{array}{c}
\begin{tikzpicture}[xscale=0.6,yscale=0.6]
\draw[black] (-0.1,-0.04,0) to [bend left=25] (2.1,-0.04,0);
\draw[black] (1.9,-0.04,0) to [bend left=25] (4.1,-0.04,0);
\node at (5.5,0,0) {$\hdots$};
\draw[black] (6.9,-0.04,0) to [bend left=25] (9.1,-0.04,0);
%\draw[black] (8.9,-0.04,0) to [bend left=25] (11.1,-0.04,0);
\node at (1,0.6,0) {$\scriptstyle \Curve_{1}$};
\node at (3,0.6,0) {$\scriptstyle \Curve_{2}$};
\node at (8,0.6,0) {$\scriptstyle \Curve_{n}$};
\filldraw [black] (0,0,0) circle (1pt);
\filldraw [black] (2,0,0) circle (1pt);
\filldraw [black] (4,0,0) circle (1pt);
\filldraw [black] (7,0,0) circle (1pt);
\filldraw [black] (9,0,0) circle (1pt);
\node at (0,-0.4,0) {$\scriptstyle g_{\sigma(0)}$};
\node at (2,-0.4,0) {$\scriptstyle g_{\sigma(1)}$};
\node at (4,-0.4,0) {$\scriptstyle g_{\sigma(2)}$};
\node at (7,-0.4,0) {$\scriptstyle g_{\sigma(n-1)}$};
\node at (9,-0.4,0) {$\scriptstyle g_{\sigma(n)}$};
\end{tikzpicture}
\end{array}
\end{array}
\]
\end{prop}

\begin{prop} \cite{IW1, W1} \label{35}
Given any $\sigma \in \mathfrak{S}_{n+1}$, let $\uppi^{\sigma} \colon \scrX^{\sigma} \to\Spec\scrR$ be the associated crepant resolution. Then the \textnormal{NCCR} $\Lambda(\uppi^{\sigma}) $ can be presented as the following quiver (with possible loops):

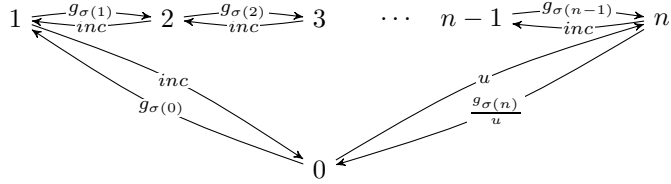
\[
\begin{tikzpicture}[bend angle=7, looseness=1.2]
\node (a) at (0,0)  {$1$};
\node (b) at (2,0)  {$2$};
\node (c) at (4,0) {$3$};
\node (d) at (5,0) {$\hdots$};
\node (e) at (6,0) {$n-1$};
\node (f) at (8.5,0)  {$n$};
\node (g) at (4,-2) {$0$};

\draw[<-,bend right] (g) to node [gap] {$\scriptstyle inc$} (a);
\draw[->,bend left] (g) to node [gap] {$\scriptstyle g_{\sigma(0)}$} (a);
\draw[<-,bend right] (a) to node [gap] {$\scriptstyle inc$} (b);
\draw[->,bend left] (a) to node [gap] {$\scriptstyle g_{\sigma(1)}$} (b);
\draw[<-,bend right] (b) to node[gap] {$\scriptstyle inc$} (c);
\draw[->,bend left] (b) to node[gap] {$\scriptstyle g_{\sigma(2)}$} (c);
\draw[<-,bend right] (e) to node[gap] {$\scriptstyle inc$} (f);
\draw[->,bend left] (e) to node[gap] {$\scriptstyle g_{\sigma(n-1)}$} (f);

\draw[<-,bend right] (g) to node[gap] {$\scriptstyle \frac{g_{\sigma(n)}}{u}$} (f);
\draw[->,bend left] (g) to node[gap] {$\scriptstyle u$} (f);
\end{tikzpicture}\\
\]
where the vertex $0$ represents $\scrR$ and the vertex $i$ represents $(u,\prod_{j=0}^{i-1} g_{\sigma(j)})$ for $1 \leq i \leq n$.

There is a loop labelled $t$ at vertex $0$ if and only if $(g_{\sigma(0)}, g_{\sigma(n)} )\subsetneq( x, y )$ and there exists $t \in ( x,y )$ such that $( g_{\sigma(0)}, g_{\sigma(n)},t )=( x,y )$ in the ring $\C\lal x,y\ral$.
Further, for any $1 \leq i \leq n$, the possible loops at vertex $i$ are given by the following rules:
\begin{enumerate}
\item the normal bundle of curve $\Curve_i$ is $\scrO(-1) \oplus \scrO(-1) \iff$ $(g_{\sigma(i-1)}, g_{\sigma(i)} )=( x, y )$ in $\C\lal x,y\ral$ $\iff$  add no loop at vertex $i$.
\item the normal bundle of curve $\Curve_i$ is $\scrO(-2) \oplus \scrO$ $\iff$ $( g_{\sigma(i-1)}, g_{\sigma(i)} )\subsetneq ( x, y )$ and there exists $t \in ( x,y )$ such that $( g_{\sigma(i-1)}, g_{\sigma(i)},t )=( x,y )$ in $\C\lal x,y\ral$ $\iff$ add a loop labelled $t$ at vertex $i$.
\end{enumerate}
\end{prop}
\begin{proof}
In general, \cite{IW1, W1} shows that either (1), (2) or the following third case holds.
\begin{enumerate}
\setcounter{enumi}{2}
\item $( g_{\sigma(i-1)}, g_{\sigma(i)} ) \subsetneq ( x, y )$ and there is no $t$ such that \textnormal{(2)} $\iff$ add two loops labelled $x$ and $y$ at vertex $i$.
\label{l3}
\end{enumerate}
We now prove that (3) is impossible when $\scrR$ is $cA_n$ and admits a crepant resolution. If there exist two loops at some vertex $i$, then $( g_{\sigma(i-1)}, g_{\sigma(i)} ) \subsetneq ( x, y )$ and there exists no $t \in (x,y )$ that satisfies $( g_{\sigma(i-1)}, g_{\sigma(i)},t)=(x,y )$. Hence both $g_{\sigma(i-1)}$ and $g_{\sigma(i)}$ must belong to $( x,y )^2$. But this contradicts the fact that $\scrR$ admits a crepant resolution $\scrX$.
\end{proof}

The following asserts that isomorphisms between contraction algebras of $cA_n$ singularities can only map $e_i$ to $e_i$ or $e_{n+1-i}$ for each $1 \leq i \leq n$, where $e_i$ denotes the trivial path at the vertex $i$.
\begin{prop}\label{37}
Let $\uppi_k \colon \scrX_k \rightarrow \Spec \scrR_k$ be two crepant resolutions of $cA_n$ singularities $\scrR_k$ for $k=1,2$. If there exists an algebra isomorphism $\phi \colon \Lambda_{\mathrm{con}}(\uppi_1) \xrightarrow{\sim} \Lambda_{\mathrm{con}}(\uppi_2)$, then $\phi$ must belong to one of the following two cases:
\begin{enumerate}
    \item $\phi(e_i)=e_i$ for $1 \leq i \leq n$,
    \item $\phi(e_i)=e_{n+1-i}$ for $1 \leq i \leq n$.
\end{enumerate}
\end{prop}
\begin{proof}
Write $\mod \Lambda_{\mathrm{con}}(\uppi_k)$ for the category of finitely generated right $ \Lambda_{\mathrm{con}}(\uppi_k)$-modules for $k=1,2$. Write $\scrS_i$ for the simple $\Lambda_{\mathrm{con}}(\uppi_1)$-module corresponding to the vertex $i$ in the quiver of $\Lambda_{\mathrm{con}}(\uppi_1)$ (see \cite[\S 5.2]{HW}). Similarly, write $\scrS_i'$ for the simple $\Lambda_{\mathrm{con}}(\uppi_2)$-module corresponding to the vertex $i$ in the quiver of $\Lambda_{\mathrm{con}}(\uppi_2)$. 
By \cite[2.11]{W1}, for each $1 \leq i \leq n$ the simple module $\scrS_i$ corresponds to the $i$-th exceptional curve $\Curve_i$ in $\scrX_1$. The same holds for $\scrS_i'$ and the exceptional curves in $\scrX_2$. 

The algebra isomorphism $\phi$ induces an equivalence $\varphi \colon \mod \Lambda_{\mathrm{con}}(\uppi_1) \xrightarrow{\sim} \mod \Lambda_{\mathrm{con}}(\uppi_2)$. By Morita theory, $\varphi$ maps simple modules to simple modules, and furthermore there is a $\sigma$ in the symmetric group $\mathfrak{S}_{n}$ such that $\varphi(\scrS_i)= \scrS'_{\sigma(i)}$.

We next use the intersection diagram of exceptional curves for a $cA_n$ singularity—namely, a Dynkin diagram of type~$A_n$, obtained from the diagram in \ref{35} by removing the vertex~$0$—together with the correspondence between $\scrS_i$, $\scrS_i'$ and exceptional curves, to constrain the permutation $\sigma$.

Since $\uppi_1$ is a crepant resolution of a $cA_n$ singularity, $\scrS_2$ is the unique simple $\Lambda_{\mathrm{con}}(\uppi_1)$-module other than $\scrS_1$ that satisfies $\Ext^1_{\Lambda_{\mathrm{con}}(\uppi_1)}(\scrS_1,\scrS_2) \neq 0$ by \ref{35} and the intersection theory of \cite[2.15]{W1}. Since $\mod \Lambda_{\mathrm{con}}(\uppi_1)$ is equivalent to $\mod \Lambda_{\mathrm{con}}(\uppi_2)$, there exists unique simple $\Lambda_{\mathrm{con}}(\uppi_2)$-module $\scrT$ other than $\scrS'_{\sigma(1)}$ such that $\Ext^1_{\Lambda_{\mathrm{con}}(\uppi_2)}(\scrS'_{\sigma(1)},\scrT) \neq 0$. Thus the exceptional curve $\sigma(1)$ in $\uppi_2$ must be a edge curve, by \ref{35} and the intersection theory of \cite[2.15]{W1}. Thus $\sigma(1)=1$ or $n$. We split the proof into two cases.

(1) $\sigma(1)=1$.

Since $\Ext^1_{\Lambda_{\mathrm{con}}(\uppi_1)}(\scrS_1,\scrS_2) \neq 0$ and $\mod \Lambda_{\mathrm{con}}(\uppi_1)$ is equivalent to $\mod \Lambda_{\mathrm{con}}(\uppi_2)$, we have $\Ext^1_{\Lambda_{\mathrm{con}}(\uppi_2)}(\scrS'_{\sigma(1)},\scrS'_{\sigma(2)}) \neq 0$, and so $\Ext^1_{\Lambda_{\mathrm{con}}(\uppi_2)}(\scrS'_{1},\scrS'_{\sigma(2)}) \neq 0$.
Thus the curve $\sigma(2)$ in $\uppi_2$ must be connected to the curve $\sigma(1)=1$, and so $\sigma(2)=2$ by \ref{35} and the intersection theory of \cite[2.15]{W1}.
Repeating the same process, we can prove $\sigma(i)=i$, and so  $\varphi(\scrS_i)= \scrS'_{i}$, and furthermore $\phi(e_i)=e_i$ for each $i$.

(2) $\sigma(1)=n$.

Since $\Ext^1_{\Lambda_{\mathrm{con}}(\uppi_1)}(\scrS_1,\scrS_2) \neq 0$ and $\mod \Lambda_{\mathrm{con}}(\uppi_1)$ is equivalent to $\mod \Lambda_{\mathrm{con}}(\uppi_2)$, we have $\Ext^1_{\Lambda_{\mathrm{con}}(\uppi_2)}(\scrS'_{\sigma(1)},\scrS'_{\sigma(2)}) \neq 0$, and so $\Ext^1_{\Lambda_{\mathrm{con}}(\uppi_2)}(\scrS'_{n},\scrS'_{\sigma(2)}) \neq 0$.
Thus the curve $\sigma(2)$ in $\uppi_2$ must be connected to the curve $\sigma(1)=n$, and so $\sigma(2)=n-1$ by \ref{35} and the intersection theory of \cite[2.15]{W1}.
Repeating the same process, we can prove $\sigma(i)=n+1-i$, and so  $\varphi(\scrS_i)= \scrS'_{n+1-i}$, and furthermore $\phi(e_i)=e_{n+1-i}$ for each $i$.
\end{proof}

\section{Monomialization}\label{mono}
This section introduces the quiver $Q_{n,I}$ and Type~$A$ potentials. 
In \S\ref{Mono} we prove that every reduced Type~$A$ potential on $Q_{n,I}$ is right-equivalent to a reduced monomialised Type~$A$ potential (see \ref{047}), which is the starting point for the geometric realisation in \S\ref{GR}. 
Finally, \S\ref{trans} shows that any monomialised Type~$A$ potential on $Q_{n,I}$ is isomorphic to a (possibly non-reduced) monomialised Type~$A$ potential on $Q_n$ (see \ref{410}), so it suffices to work on $Q_n$.

\begin{definition}\label{def:contain}
Given a quiver $Q$, let $f$, $g$ and $h$ be potentials on $Q$. Write $f= \sum_{i} \uplambda_i c_i$ as a linear combination of cycles where each $\uplambda_i \in  \C^{\times}$. 
\begin{enumerate}
\item We write $V(f)$ for the $\C$-span of all cycles $c$ such that $c\sim c_i$ for some $i$.
\item We say $f$ is \emph{orthogonal} to $g$ if $V(f) \cap V(g) =\{0\}$.
\item We write $f= g \oplus h$ if $f=g+h$ and $g$ is orthogonal to $h$.
\item We say $f$ \emph{contains} $g$ if $f \sim \uplambda g\oplus h'$ for some $\uplambda \in  \C^{\times}$ and potential $h'$.
\end{enumerate} 
\end{definition}

Recall the definition of the quiver $Q_n$ in \ref{mainresult}, which is the double of the usual $A_n$ quiver, with a single loop at each vertex as follows.
 \[
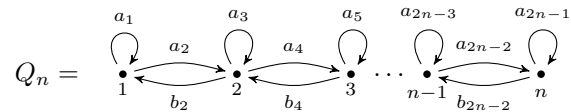
\begin{array}{c}
\begin{tikzpicture}[bend angle=15, looseness=1.2]
\node (a) at (-1.5,0) [vertex] {};
\node (b) at (0,0) [vertex] {};
\node (c) at (1.5,0) [vertex] {};
\node (c2) at (2,0) {$\hdots$};
\node (d) at (2.5,0) [vertex] {};
\node (e) at (4,0) [vertex] {};
\node (a1) at (-1.5,-0.2) {$\scriptstyle 1$};
\node (a2) at (0,-0.2) {$\scriptstyle 2$};
\node (a3) at (1.5,-0.2) {$\scriptstyle 3$};
\node (a4) at (2.5,-0.25) {$\scriptstyle n-1$};
\node (a5) at (4,-0.25) {$\scriptstyle n$};
\draw[->,bend left] (a) to node[above] {$\scriptstyle a_{2}$} (b);
\draw[<-,bend right] (a) to node[below] {$\scriptstyle b_{2}$} (b);
\draw[->,bend left] (b) to node[above] {$\scriptstyle a_{4}$} (c);
\draw[<-,bend right] (b) to node[below] {$\scriptstyle b_{4}$} (c);
\draw[->,bend left] (d) to node[above] {$\scriptstyle a_{2n-2}$} (e);
\draw[<-,bend right] (d) to node[below] {$\scriptstyle b_{2n-2}$} (e);
\draw[<-]  (a) edge [in=120,out=55,loop,looseness=10] node[above] {$\scriptstyle a_{1}$} (a);
\draw[<-]  (b) edge [in=120,out=55,loop,looseness=11] node[above] {$\scriptstyle a_{3}$} (b);
\draw[<-]  (c) edge [in=120,out=55,loop,looseness=11] node[above] {$\scriptstyle a_{5}$} (c);
\draw[<-]  (d) edge [in=120,out=55,loop,looseness=11] node[above] {$\scriptstyle a_{2n-3}$} (d);
\draw[<-]  (e) edge [in=120,out=55,loop,looseness=11] node[above] {$\scriptstyle a_{2n-1}$} (e);
\node (z) at (-2.5,0) {$Q_{n} =$};
\end{tikzpicture}
\end{array}
\]
For any $I \subseteq \{1,2, \dots, n\}$, define the quiver \emph{$Q_{n, I}$} by removing the loop in $Q_n$ at each vertex $i \in I$, and then relabel $a_i$ and $b_i$ from left to right. 
As before, we now set $b_i\colonequals e_i$ whenever $a_i$ is a loop in $Q_{n, I}$, and set $\x_i \colonequals a_ib_i$ and $\x_i' \colonequals b_ia_i$ for each $i$. In particular, the quiver \(Q_{n,I}\) contains \(2n-1-|I|\) paths of each type \(a_i\), \(b_i\), \(\x_i\), and \(\x_i'\).

For example,
 \[
  \begin{array}{cl}
\begin{array}{c}
\begin{tikzpicture}
\node (a) at (-1,0) [vertex] {};
\node (b) at (0,0) [vertex] {};
\node (c) at (1,0) [vertex] {};

\node (a1) at (-1,-0.2) {$\scriptstyle 1$};
\node (a2) at (0,-0.2) {$\scriptstyle 2$};
\node (a3) at (1,-0.2) {$\scriptstyle 3$};

\draw[->,bend left] (a) to node[above] {$\scriptstyle a_{2}$} (b);
\draw[<-,bend right] (a) to node[below] {$\scriptstyle b_{2}$} (b);
\draw[->,bend left] (b) to node[above] {$\scriptstyle a_{3}$} (c);
\draw[<-,bend right] (b) to node[below] {$\scriptstyle b_{3}$} (c);
\draw[<-]  (a) edge [in=120,out=55,loop,looseness=13] node[above] {$\scriptstyle a_{1}$} (a);
\draw[<-]  (c) edge [in=120,out=55,loop,looseness=13] node[above] {$\scriptstyle a_{4}$} (c);
\node (z) at (-2,0) {$Q_{3,\{2\}}=$};
\end{tikzpicture}
\end{array}
&
\begin{array}{l}
b_1=e_1,\ \x_1=\x_1'=a_1\\
b_4=e_4,\ \x_4=\x_4'=a_4\\
\end{array}
\end{array}
\]
whereas $\x_2=a_2b_2$, $\x_2'=b_2a_2$, and $\x_3=a_3b_3$, $\x_3'=b_3a_3$. 
In this example, the quiver $Q_{3,\{2\}}$ is obtained by removing the loop at vertex $2$ in $Q_3$.
More precisely, the quiver $Q_3$ contains arrows $a_i$ and $b_i$ for $1 \leq i \leq 5$.
Since the loop at vertex $2$ corresponds to the arrows $a_3$ and $b_3$, these are removed.
The arrows $a_1, b_1, a_2,$ and $b_2$ are left unchanged, while the indices of
$a_4, b_4, a_5,$ and $b_5$ in $Q_3$ are shifted down by one, becoming
$a_3, b_3, a_4,$ and $b_4$ in $Q_{3,\{2\}}$.

\begin{notation}\label{no1}
Throughout this paper, $n$ denotes the number of vertices of $Q_{n, I}$, and $I \subseteq \{1,2, \dots, n\}$ is the set of vertices \emph{without} loop in $Q_{n, I}$. Note that $Q_{n,\emptyset}$ is just $Q_n$. Furthermore, set $m\colonequals 2n-1-|I|$, which equals the number of $\x_i$ in $Q_{n, I}$
\end{notation}

We now give several definitions and notations with respect to $Q_{n, I}$.

\begin{definition}\label{deg}
Given a cycle $c$ on $Q_{n, I}$, write $c$ as a composition of arrows.
For $1 \leq i \leq m$, let $q_i$ be the number of times $a_i$ appears in this composition. Then set $\mathbf{T}(c) \colonequals (q_1,q_2, \dots ,q_m)$, and define the \emph{degree} of $c$ to be $\deg(c) \colonequals \sum_{i=1}^m q_i$.
\end{definition}

\begin{definition}\label{TypeA}
We say that a potential $f$ on $Q_{n, I}$ is \emph{reduced Type $A$} if $f$ is reduced in the sense of \textnormal{\ref{QP}} and $f$ contains $\x_i^{\prime}\x_{i+1}$ in the sense of \textnormal{\ref{def:contain}} for each $1 \leq i \leq m-1$. Further, we say that a (possibly non-reduced) potential $f$ on $Q_{n, I}$ is \emph{Type $A$}
 if
\begin{enumerate}
\item All terms of $f$ have degrees greater than or equal to two in the sense of \textnormal{\ref{deg}}.
\item The reduced part $f_{\textnormal{red}}$ is Type $A$ on $Q_{n,I'}$ for some $I \subseteq I' \subseteq \{1,2,\dots,n\}$.
\end{enumerate}
\end{definition}
The Splitting Theorem \cite[4.6]{DWZ1} gives the existence and uniqueness of $f_{\textnormal{red}}$, so \ref{TypeA} is well defined.

\begin{lemma}\label{lam}
Given any potential $f \sim \sum_{i=1}^{m-1}\uplambda_i\x_i'\x_{i+1}+h$ where each $\uplambda_i \in \C^{\times}$ on $Q_{n, I}$, there exists a right-equivalence $f \rightsquigarrow f'$ such that 
\[
f'=\sum_{i=1}^{m-1}\x_i'\x_{i+1}+g
\]
for some potential $g$.
\end{lemma}
\begin{proof}
Applying $a_i \mapsto k_ia_i$ where $k_i\in \C$ for each $1\leq i \leq m$ gives
\begin{equation*}
    f \rightsquigarrow \sum_{i=1}^{m-1}k_ik_{i+1}\uplambda_i\x_i'\x_{i+1}+g,
\end{equation*}
for some potential $g$. Since each $\uplambda_i \neq 0$, we can always find some $(k_1,k_2,\dots,k_m)$ that ensures $k_ik_{i+1}\uplambda_i=1$ holds for $1\leq i \leq m-1$.
\end{proof}

\begin{remark}\label{rmk:middle_term}
The above lemma shows that any reduced Type $A$ potential $f$ can be transformed to the form of $ \sum_{i=1}^{m-1}\x_i'\x_{i+1} \oplus g$ for some potential $g$. Thus, in this paper, for any reduced Type $A$ potential $f$ on $Q_{n,I}$, we always assume that $f=\sum_{i=1}^{m-1}\x_i'\x_{i+1} \oplus g$.
\end{remark}

\begin{definition}\label{mTypeA}
We call a potential $f$ on $Q_{n, I}$ \emph{monomialized Type $A$} if $f  \sim \sum_{i=1}^{m-1}\x_i^{\prime}\x_{i+1} + \sum_{i=1}^m\sum_{j=2}^{\infty} \upkappa_{ij} \x_i^{j} $ for some $\upkappa_{ij} \in \C$.    
\end{definition}

Given any monomialized Type $A$ potential $f$, it is clear that $f$ is Type A. Moreover, $f$ is reduced if and only if $\upkappa_{s2}=0$ whenever $\x_s$ is a loop.

%We next introduce some notations, which will be used in the next section.
\begin{definition}\label{length}
Given a cycle $c$ on $Q_{n, I}$, consider $\mathbf{T}(c)$ from \textnormal{\ref{deg}}.
Define $\l(c)$ to be the smallest $i$ such that $q_i >0$, and $\r(c)$ to be the largest $i$ such that $q_i >0$. We then define the \emph{length} of $c$ by $\le(c)\colonequals \r(c)-\l(c)+1$.      
\end{definition}
From the above definition, if $\le(c)=1$ then $c \sim \x_i^j$ for some $1 \leq i \leq m$ and $j \geq 1$.
\begin{notation}\label{notation:dij}
We adopt the following notation regarding cycles on $Q_{n, I}$.
\begin{enumerate}
\item Write $F$ for the $\C$-span of $\{c \mid c\text{ is a cycle with } \deg(c) \geq 1 \}$ where the degree is defined in \ref{deg}.
\item For any $i \in \N$, write $D_i$ for the $\C$-span of $\{c \mid \text{c is a cycle with } \deg(c)=i \}$.
\item For any $i \in \N$, write $L_i$ for the $\C$-span of $\{c \mid \text{c is a cycle with } \le(c)=i \}$ where length is defined in \ref{length}.
\item For any $i$ and $j \in \N$ satisfying $1 \leq i \leq j \leq m$, write $V_{ij}$ for the $\C$-span of $\{c \mid \text{c is a cycle with } \l(c)=i \text{ and } \r(c)=j  \}$.
\end{enumerate}
\end{notation}
It is clear that $F = \bigoplus_{i} D_i$, $F = \bigoplus_{i} L_i$ and $F = \bigoplus_{i\leq j} V_{ij}$.

\begin{notation}\label{def:deg>}
Let $f$ be a potential on $Q_{n, I}$.
\begin{enumerate}
\item Write $\deg(f)=i$ if $f \in D_i$. Similarly write $\deg(f) \geq i$ if $f \in \oplus_{j \geq i} D_j$, with natural self-documenting variations such as $\deg(f)  \leq i$.
\item Write $\le(f)=i$ if $f \in L_i$. Similarly write $\le(f) \geq i$ if $f \in \oplus_{j \geq i} L_j$, with natural self-documenting variations such as $\le(f)  \leq i$. 
\end{enumerate}
\end{notation}
The above degree and length notations will be important, and they will replace the common notations such as path length.

\begin{notation}\label{no5}
Let $f$ and $g$ be potentials on $Q_{n, I}$. With the notation in \ref{notation:dij}, since $f, g \in F$, $F = \bigoplus_{i} D_i$ and $F = \bigoplus_{i\leq j}V_{ij}$, we will adopt the following notation.
\begin{enumerate}
\item Define $f_d$ by decomposing $f=\sum_d f_d$ where each $f_d \in D_d$. 
\item Define $f_{<d}= \sum_{i<d}f_i$ and $f_{>d}= \sum_{i>d}f_i$, with natural self-documenting variations such as $f_{\leq d}$ and $f_{\geq d}$. Thus, if $\deg(f) \geq 2$ then $f=f_2 +f_3 + f_{>3}$.
\item Define $f_{ij}$ by decomposing 
\[
f=\sum_{i,j:1 \leq i\leq j \leq m} f_{ij},
\]
where each $f_{ij} \in V_{ij}$. Variations such as $f_{ij,d}$, $f_{ij,<d}$, $f_{ij,\leq d}$, $f_{ij,>d}$ and $f_{ij,\geq d}$ are obtained by applying (1) and (2) to $f_{ij}$.

\item Given $s$ such that $1\leq s\leq m$, set
\[
f_{[s]} \colonequals\sum_{i,j: 1 \leq i \leq s \leq j \leq m} f_{ij}.
\]
Variations such as $f_{[s],d}$, $f_{[s],<d}$, $f_{[s],\leq d}$, $f_{[s],>d}$ and $f_{[s],\geq d}$ are obtained by applying (1) and (2) to $f_{[s]}$.
\item Write $f=g+\scrO_d$ if $f-g \in \bigoplus_{k \geq d}D_k$, and $f=g+\scrO_{ij,d}$ if $f-g \in V_{ij} \bigcap \bigoplus_{k \geq d}D_k$.
\end{enumerate}
\end{notation}
\begin{remark}
We will frequently work with sequences of potentials $(\mathsf{f}_d)_{d \geq 1}$ on $Q_{n, I}$, and write $f_d$ for the degree $d$ pieces of $f$ (see \ref{no5}). To avoid confusion, we will systematically use Greek font $\mathsf{f}_d$ to denote the $d$-th elements in a sequence, and not the $d$-th degree piece.
\end{remark}

\subsection{Monomialization}\label{Mono}
This subsection will prove that any reduced Type $A$ potential on $Q_{n, I}$ is right-equivalent to some reduced monomialized Type $A$ potential (see \ref{047}). 

\begin{notation}\label{notation:mono}
To ease notation, in this subsection $f$ will always refer to a reduced Type $A$ potential on $Q_{n, I}$ of the form $\sum_{i=1}^{m-1}\x_i^{\prime}\x_{i+1} \oplus g$ (see \ref{rmk:middle_term}).    
In the statements below, to ease notation, the $c$ and $c_k$ will refer to a cycle on $Q_{n, I}$, possibly with a coefficient.
\end{notation}

The following lemma allows us to monomialize the degree $2$ terms in $f$. 
\begin{lemma}\label{041}
Suppose that $g= h + c$ where $\le(c) \geq 3$ and $\deg(c)=2$. Then there exists a path degree $1$ right-equivalence (in the sense of \textnormal{\ref{211}}),
    \begin{equation*}
        \rho_c \colon f \stackrel{1}{\rightsquigarrow}   \sum_{i=1}^{m-1}\x_i^{\prime}\x_{i+1}\oplus (h+c_1)+\scrO_{3},
    \end{equation*}
such that $\le(c_1)=1$ and $\deg(c_1)=2$.
\end{lemma}
\begin{proof}
Since $\deg(c) =2$ and $\le(c) \geq 3$, $c$ must have the form of $c \sim \uplambda \x_{s-1}'\x_{s+1}$ for some $\uplambda \in  \C^{\times}$, where $s$ is such that $\x_s$ is a loop.  
Since $\x_s$ is a loop and $f$ is reduced, $f$ does not contain $\x_s^2$, and so $f_{[s],2}=\x_{s-1}'\x_s+\x_s'\x_{s+1}$. 

Rewrite $f=f_{[s],2} \oplus f_{[s],\geq 3} \oplus r$. Being a loop, $\x_s=a_s$, so applying the depth one unitriangular automorphism $\rho_c \colon a_s \mapsto a_s-\uplambda b_{s-1}a_{s-1}$ (in other words, $\x_s \mapsto \x_s- \uplambda \x_{s-1}'$) gives

\begin{align*}
 \rho_c(f) & = \x_{s-1}'(\x_s- \uplambda \x_{s-1}')+(\x_s- \uplambda \x_{s-1}')\x_{s+1} + f_{[s],\geq 3}+r +\scrO_{3} \\
 & = f- \uplambda (\x_{s-1}')^2 -\uplambda \x_{s-1}'\x_{s+1}+\scrO_3 \tag{$f=f_{[s],2}+ f_{[s],\geq 3} + r$}\\
 & = \sum_{i=1}^{m-1}\x_i^{\prime}\x_{i+1}+h+c- \uplambda (\x_{s-1}')^2 -\uplambda \x_{s-1}'\x_{s+1}+\scrO_3 \tag{$f= \sum_{i=1}^{m-1}\x_i^{\prime}\x_{i+1}+h+c$ }\\
 & \stackrel{2}{\sim} \sum_{i=1}^{m-1}\x_i^{\prime}\x_{i+1}+h- \uplambda \x_{s-1}^2+\scrO_3  \tag{$c \sim \uplambda \x_{s-1}'\x_{s+1}, \ (\x_{s-1}')^2 \sim \x_{s-1}^2$}\\
 & =\sum_{i=1}^{m-1}\x_i^{\prime}\x_{i+1} \oplus (h- \uplambda \x_{s-1}^2)+\scrO_3. \tag{$f=\sum_{i=1}^{m-1}\x_i^{\prime}\x_{i+1} \oplus (h + c), \ \le(c) \geq 3$}
\end{align*}

Set $c_1 = -\uplambda \x_{s-1}^2$, which satisfies $\le(c_1)=1$ and $\deg(c_1)=2$, and we are done. 
\end{proof}

The following lemmas allow us to monomialise the terms of degree at least three in $f$.
More precisely, given a cycle $c$ with $\le(c)\ge 2$ occurring in $f$, we repeatedly apply
right-equivalences that decrease $\r(c)$ (see \ref{043}) until all resulting terms have length one
(see \ref{044}).
 
\begin{lemma}\label{042}
Suppose that $\le(f_2) \leq 2$ and $g=  h + c $ where $\le(c) \geq 2$, $d \colonequals \deg(c) \geq 3$. Then there exists a path degree $d-1$ right-equivalence,
\begin{equation*}
    \upvartheta \colon f \stackrel{d-1}{\rightsquigarrow} \sum_{i=1}^{m-1}\x_i^{\prime}\x_{i+1} \oplus  (h + c_1 + c_2)  + \scrO_{d+1},
\end{equation*}
such that each $c_k$ is either zero or satisfies $\r(c_k) \leq \r(c)$, $\deg(c_k)=\deg(c)$ and $\mathbf{T}(c_k)_{\r(c)}=\mathbf{T}(c)_{\r(c)}-1$.   
\end{lemma}
\begin{proof}
Set $s= \r(c)-1$. The assumption $\le(f_2) \leq 2$ says that the degree two part of $f$ (wrt.$\ \x_i$, as in \ref{def:deg>}) must be spread over at most two variables. 
Hence the only degree two cycles in $f$ involving $\x_s$ are $\x_{s-1}'\x_s$, $\x_s^2$ and $\x_{s}'\x_{s+1}$. In the notation of $\ref{no5}$, this gives $f_{[s],2}= \x_{s-1}'\x_s+\upkappa \x_s^2+\x_{s}'\x_{s+1}$ for some $\upkappa \in \C$. 
Decomposing $f$ into terms that do and do not involve $\x_s$, we may write
$f=f_{[s],2}\oplus f_{[s],\ge 3}\oplus r$. We treat two cases.

(1) $\x_{s}$ is not a loop.

The assumptions that $\le(c) \geq 2$ and $\r(c)=s+1$ imply that $a_s$, $b_s$, $a_{s+1}$ and $b_{s+1}$ all appear in $c$. Note that $\x_s$ is not a loop, thus $\x_s=a_sb_s$. Locally $Q_{n, I}$ looks like the following.

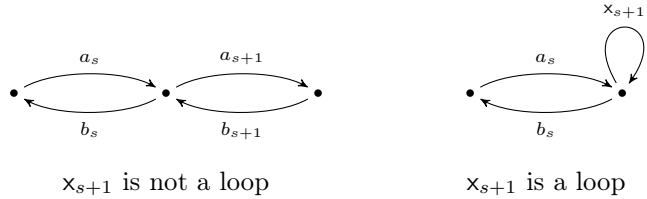
\[
\begin{tikzpicture}[scale=2]
\node (a) at (-1,0) [vertex] {};
\node (b) at (0,0) [vertex] {};
\node (c) at (1,0) [vertex] {};
\node (d) at (0,-0.6) {$\x_{s+1}$ is not a loop};
 
%\node (a1) at (-1,-0.25) {$\scriptstyle s-1$};
%\node (a2) at (0,-0.25) {$\scriptstyle s$};
%\node (a3) at (1,-0.25) {$\scriptstyle s+1$};

\draw[->,bend left,looseness=0.7] (a) to node[above] {$\scriptstyle a_{s}$} (b);
\draw[<-,bend right,looseness=0.7] (a) to node[below] {$\scriptstyle b_{s}$} (b);
\draw[->,bend left,looseness=0.7] (b) to node[above] {$\scriptstyle a_{s+1}$} (c);
\draw[<-,bend right,looseness=0.7] (b) to node[below] {$\scriptstyle b_{s+1}$} (c);

\node (e) at (2,0) [vertex] {};
\node (f) at (3,0) [vertex] {};
\node (g) at (2.5,-0.6) {$\x_{s+1}$ is a loop};

%\node (b1) at (2,-0.25) {$\scriptstyle s-1$};
%\node (b2) at (3,-0.25) {$\scriptstyle s$};

\draw[->,bend left,looseness=0.7] (e) to node[above] {$\scriptstyle a_{s}$} (f);
\draw[<-,bend right,looseness=0.7] (e) to node[below] {$\scriptstyle b_{s}$} (f);
\draw[->]  (f) edge [in=55,out=120,loop,looseness=18] node[above]{$\scriptstyle \x_{s+1}$}  (f);
\end{tikzpicture}
\]
Since $\r(c)=s+1$, we may assume that the cycle $c$ starts with $\x_{s+1}$, up to cyclic equivalence. Then $c$ begins with some power of $\x_{s+1}$, and the next arrow must be $b_s$.
Thus we may write
\begin{equation*}
    c \sim \uplambda \x_{s+1}^N b_{s}pa_{s}r \sim \uplambda  b_{s}pa_{s}r\x_{s+1}^{N}
\end{equation*}
for some $\uplambda \in \C^{\times}$, an integer $N \geq 1$, and paths $p$, $r$. Consider the path $q \colonequals r\x_{s+1}^{N-1}$, and rewrite $c \sim \uplambda b_{s}pa_{s}q\x_{s+1}$. 
Since $\deg(c) =d\geq 3$ and $\deg(\x_s)=\deg(\x_{s+1})=1$, we have $\deg(p)+\deg(q)=d-2\geq 1$.

Then applying the depth $d-1$ unitriangular automorphism $\upvartheta \colon a_{s} \mapsto a_{s}- \uplambda pa_{s}q$ gives
\begin{align*}
 \upvartheta (f)& =  \x_{s-1}^{\prime}(a_{s}- \uplambda pa_{s}q)b_{s}+\upkappa[(a_{s}- \uplambda pa_{s}q)b_{s}]^2
+b_{s}(a_{s}- \uplambda pa_{s}q)\x_{s+1}  +f_{[s],\geq 3}+r +\scrO_{d+1}\\
& \stackrel{d}{\sim} f - \uplambda \x_{s-1}'pa_{s}qb_{s}-2\uplambda \upkappa \x_{s}pa_{s}qb_{s}-\uplambda b_{s}pa_{s}q\x_{s+1} +\scrO_{d+1} \tag{$f=f_{[s],2} + f_{[s],\geq 3} + r$}\\
& =\sum_{i=1}^{m-1}\x_i^{\prime}\x_{i+1}  - \uplambda \x_{s-1}'pa_{s}qb_{s}-2\uplambda \upkappa \x_{s}pa_{s}qb_{s}-\uplambda b_{s}pa_{s}q\x_{s+1}+ c+h+\scrO_{d+1} \tag{$f= \sum_{i=1}^{m-1}\x_i^{\prime}\x_{i+1}+h+c$ }\\
& \stackrel{d}{\sim} \sum_{i=1}^{m-1}\x_i^{\prime}\x_{i+1}  - \uplambda \x_{s-1}'pa_{s}qb_{s}-2\uplambda \upkappa \x_{s}pa_{s}qb_{s}+h+\scrO_{d+1} \tag{$c  \sim \uplambda b_{s}pa_{s}q\x_{s+1}$}\\
& = \sum_{i=1}^{m-1}\x_i^{\prime}\x_{i+1} \oplus (- \uplambda \x_{s-1}'pa_{s}qb_{s}-2\uplambda \upkappa \x_{s}pa_{s}qb_{s}+h)+\scrO_{d+1}. 
\end{align*}
Since $\deg(p)+\deg(q)=d-2$, all terms not displayed explicitly lie in $\scrO_{d+1}$.

Set $c_1 =- \uplambda \x_{s-1}'pa_{s}qb_{s}$ and $c_2 = -2\uplambda \upkappa \x_{s}pa_{s}qb_{s}$. Since $\deg(p)+\deg(q) =d-2$, it follows that $\deg(c_1)= d = \deg(c_2)$.  
The conclusions for $c_1$ are clear. Either $c_2$ is zero or $\upkappa \neq 0$. In that case, the conclusions for $c_2$ are also clear. 

(2) $\x_{s}$ is a loop.

Since $\x_s$ is a loop, from the shape of the quiver $Q_{n, I}$, $\x_{s+1}$ is not a loop. Since $\r(c)=s+1$, we can assume that the cycle $c$ ends with $\x_{s+1}$, up to cyclic equivalence. Thus $c \sim \uplambda p\x_{s+1}$ for some path $p$ and $\uplambda \in  \C^{\times}$. 

Since $\deg(c) =d\geq 3$ and $\deg(\x_{s+1}) =1$, $\deg(p) =d-1\geq 2$. 
Moreover, since $\x_{s}$ is a loop and $f$ is reduced, the coefficient $\upkappa$ in $f_{[s],2}= \x_{s-1}'\x_s+\upkappa \x_s^2+\x_{s}'\x_{s+1}$ is zero. Locally $Q_{n, I}$ looks like the following.

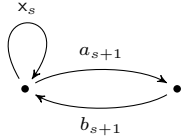
\[
\begin{tikzpicture}[scale=2]

\node (e) at (2,0) [vertex] {};
\node (f) at (3,0) [vertex] {};
%\node (b1) at (2,-0.25) {$\scriptstyle s-1$};
%\node (b2) at (3,-0.25) {$\scriptstyle s$};

\draw[->,bend left,looseness=0.7] (e) to node[above] {$\scriptstyle a_{s+1}$} (f);
\draw[<-,bend right,looseness=0.7] (e) to node[below] {$\scriptstyle b_{s+1}$} (f);
\draw[->]  (e) edge [in=55,out=120,loop,looseness=18] node[above]{$\scriptstyle \x_{s}$}  (e);
\end{tikzpicture}
\]

Being a loop, $\x_s=a_s$, so applying the depth $d-1$ unitriangular automorphism $\upvartheta \colon a_{s} \mapsto a_{s} - \uplambda p$ (in other words, $\x_{s} \mapsto \x_{s} - \uplambda p$) gives
\begin{align*}
\upvartheta(f) & = \x_{s-1}'(\x_{s} - \uplambda p)+(\x_{s} - \uplambda p)\x_{s+1} +f_{[s],\geq 3}+r +\scrO_{d+1}\\
& \stackrel{d}{\sim} f-\uplambda \x_{s-1}'p-\uplambda p\x_{s+1} +\scrO_{d+1} \tag{$f=f_{[s],2} + f_{[s],\geq 3} + r$}\\
&= \sum_{i=1}^{m-1}\x_i^{\prime}\x_{i+1} -\uplambda \x_{s-1}'p-\uplambda p\x_{s+1}+c+h +\scrO_{d+1} \tag{$f= \sum_{i=1}^{m-1}\x_i^{\prime}\x_{i+1}+h+c$ }\\
&\stackrel{d}{\sim} \sum_{i=1}^{m-1}\x_i^{\prime}\x_{i+1} -\uplambda \x_{s-1}'p+h +\scrO_{d+1} \tag{$c \sim \uplambda p\x_{s+1}$}\\
& = \sum_{i=1}^{m-1}\x_i^{\prime}\x_{i+1} \oplus (-\uplambda \x_{s-1}'p+h) +\scrO_{d+1}. 
\end{align*}
Note that since $\deg(p) =d-1$, all terms other than those explicitly displayed above contribute only to the higher degree part $\scrO_{d+1}$.

Set $c_1 = -\uplambda \x_{s-1}'p$ and $c_2 =0$. Since $\deg(p) =d-1$, it follows that $\deg(c_1)= d$.  
The conclusions for $c_1$ and $c_2$ are clear.
\end{proof}

We next apply the previous lemma multiple times to decrease $\r(c)$.
\begin{cor}\label{043}
Suppose that $\le(f_2) \leq 2$ and $g=  h+c $ where $\le(c) \geq 2$, $d \colonequals \deg(c) \geq 3$. Then there exists a path degree $d-1$ right-equivalence
\begin{equation*}
    \upvartheta \colon f \stackrel{d-1}{\rightsquigarrow} \sum_{i=1}^{m-1}\x_i^{\prime}\x_{i+1} \oplus ( h +\sum_k  c_k)  + \scrO_{d+1} ,
\end{equation*}
such that $\r(c_k)  \leq \r(c) -1$ and $\deg(c_k)=\deg(c)$ for each $k$. 
%Furthermore, there also exists a degree $d-1$ right-equivalence  $\theta^{\prime} \colon f^{\prime} \stackrel{d-1}{\rightsquigarrow} f+\scrO_{d+1}$.    
\end{cor}
\begin{proof}
Set $\mathbf{q}=\mathbf{T}(c)$ and $j=\r(c)$. By \ref{042}, there exists a path degree $d-1$ right-equivalence,
\begin{equation*}
    \upvartheta_1 \colon f \stackrel{d-1}{\rightsquigarrow} \mathsf{f}_1 \colonequals  \sum_{i=1}^{m-1}\x_i^{\prime}\x_{i+1} \oplus( h + \sum_{s=1}^2 w_s) + \scrO_{d+1},
\end{equation*}
such that $w_s$ is either zero, or satisfies $\r(w_s) \leq \r(c)$, $\mathbf{T}(w_s)_{j}=q_j-1$ and $\deg(w_s)=\deg(c)$ for each $s$. 
%Furthermore, there also exists a degree $d-1$ right-equivalence $\theta_1^{\prime} \colon f_1 \stackrel{d-1}{\rightsquigarrow} f+\scrO_{d+1}$.

If all $w_s$ equal zero, or all satisfy $\mathbf{T}(w_s)_j=0$, we are done. Otherwise, we continue to apply \ref{042} to decrease $\mathbf{T}(w_s)_j$, as follows.
%for the $s$ with $w_s \neq 0$ and $\mathbf{T}(w_s)_j>0$. 
\begin{equation*}
    \upvartheta_2 \colon \mathsf{f}_1 \stackrel{d-1}{\rightsquigarrow} \mathsf{f}_2  \colonequals
     \sum_{i=1}^{m-1}\x_i^{\prime}\x_{i+1}\oplus ( h + \sum_{s=1}^2\sum_{t=1}^2 w_{st})+\scrO_{d+1},
\end{equation*}
such that each $w_{st}$ is either zero, or $\r(w_{st}) \leq \r(w_s) \leq \r(c)$, $\mathbf{T}(w_{st})_j=q_j-2$ and $\deg(w_{st})=\deg(c)$. 
The proof follows by induction.
\end{proof}

We iteratively apply \ref{043} to reach the case where all resulting terms have length one, i.e.\ a monomial-type potential.

\begin{cor}\label{044}
Suppose that $\le(f_2)\leq 2$ and $g=h + c$ where $\le(c) \geq 2$, $d \colonequals \deg(c) \geq 3$. Then there exists a path degree $d-1$ right-equivalence
\begin{equation*}
    \rho_c\colon f \stackrel{d-1}{\rightsquigarrow} \sum_{i=1}^{m-1}\x_i^{\prime}\x_{i+1} \oplus ( h+ \sum_{k} c_k) +  \scrO_{d+1}, 
\end{equation*}  
such that $\le(c_k)=1$ and $\deg(c_k)=\deg(c)$ for each $k$.  
\end{cor}
\begin{proof}
Set $j \colonequals \r(c)$.
By \ref{043}, there exists a path degree $d-1$ right-equivalence \begin{equation*}
    \upvartheta_1 \colon f \stackrel{d-1}{\rightsquigarrow} \mathsf{f_1} \colonequals \sum_{i=1}^{m-1}\x_i^{\prime}\x_{i+1}\oplus (h +\sum_s  w_s)  + \scrO_{d+1} ,
\end{equation*}
such that $\deg(w_s)=d$  and $\r(w_s)  \leq j -1$ for each $s$.

If all $\le(w_s)=1$, we are done. Otherwise, we continue to apply \ref{043} to those $\le(w_s)>1$ to decrease $\r(w_s)$, as follows.
\begin{equation*}
    \upvartheta_2 \colon \mathsf{f}_1 \stackrel{d-1}{\rightsquigarrow} \mathsf{f}_2 \colonequals  \sum_{i=1}^{m-1}\x_i^{\prime}\x_{i+1}\oplus( h + \sum_{s,t} w_{st})+ \scrO_{d+1} ,
\end{equation*}
such that $\deg(w_{st})=\deg(c)$ and the $w_{st}$ satisfies $\r(w_{st}) \leq j-2$. 

If all $\le(w_{st})=1$, we are done. 
Otherwise, we can repeat this process at most $j-1$ times, as follows.
\begin{equation*}
    \rho_c \colon f \stackrel{d-1}{\rightsquigarrow} \mathsf{f}_{j-1} \colonequals  \sum_{i=1}^{m-1}\x_i^{\prime}\x_{i+1} \oplus( h  +\sum_{k} c_k)+ \scrO_{d+1}  ,
\end{equation*}
such that $\deg(c_{k})=\deg(c)$, and either each $\le(c_k)=1$ or $\r(c_k)=1$. However if $\r(c_k)=1$, then $\le(c_k)=1$, we are done. 
\end{proof}

We now monomialise $f$ degree by degree, using the previous lemmas. 
We begin with the degree two part.
\begin{prop}\label{045}
There exists a path degree $1$ right-equivalence,
   \begin{equation*}
        \rho_2 \colon f \stackrel{1}{\rightsquigarrow}   \sum_{i=1}^{m-1}\x_i^{\prime}\x_{i+1} \oplus  h+\scrO_{3},
    \end{equation*}
such that $\le(h )=1$ and $\deg(h )=2$.   
\end{prop}
\begin{proof}
Decompose $g$ (from \ref{notation:mono}) by degree  (wrt.$\ \x_i$, as in \ref{def:deg>}) as $g=g_2\oplus g_{\ge 3}$, and write $g_2= \oplus_{k=1}^s c_k$ 
as a $\C$-linear combination of degree-two cycles.
Since there are only a finite number of cycles with degree two on $Q_{n, I}$, necessarily $s$ is finite.
Then we write
\begin{equation*}
    f= \sum_{i=1}^{m-1}\x_i^{\prime}\x_{i+1} \oplus g= \sum_{i=1}^{m-1}\x_i^{\prime}\x_{i+1} \oplus g_2\oplus g_{\geq 3}.
\end{equation*}
Since $f=\sum_{i=1}^{m-1}\x_i'\x_{i+1}\oplus g$ and $g_2$ is orthogonal to $\sum_{i=1}^{m-1}\x_i'\x_{i+1}$,
the degree two part $g_2$ contains no length two terms. Hence $\le(c_k)=1$ or $\le(c_k)\ge 3$ for each $k$.

If $\le(c_1)=1$, set $\rho_{c_1}=\mathrm{Id}$. Otherwise $\le(c_1) \geq 3$, so by \ref{041} there exists
\begin{equation*}
    \rho_{c_1} \colon f \stackrel{1}{\rightsquigarrow}   \mathsf{f}_1 \colonequals \sum_{i=1}^{m-1}\x_i^{\prime}\x_{i+1} \oplus(\sum_{k = 2}^s c_k  + \mathrm{h}_1)+\scrO_{3},
\end{equation*}
such that $\le(\mathrm{h}_1)=1$ and $\deg(\mathrm{h}_1)=2$. 

If $\le(c_2)=1$, set $\rho_{c_2}=\mathrm{Id}$. Otherwise $\le(c_2) \geq 3$, so again by \ref{041} there exists
\begin{equation*}
    \rho_{c_2} \colon  \mathsf{f}_1 \stackrel{1}{\rightsquigarrow}    \mathsf{f}_2 \colonequals \sum_{i=1}^{m-1}\x_i^{\prime}\x_{i+1} \oplus (\sum_{k=3}^s c_k + \sum_{k=1}^2 \mathrm{h}_k)+\scrO_{3},    
\end{equation*}
such that $\le(\mathrm{h}_2)=1$ and $\deg(\mathrm{h}_2)=2$. 

We repeat this process $s$ times and set $\rho_2\colonequals \rho_{c_s} \circ \dots \circ \rho_{c_2} \circ \rho_{c_1}$. It follows that,
\begin{equation*}
    \rho_{2} \colon f \stackrel{1}{\rightsquigarrow}   \sum_{i=1}^{m-1}\x_i^{\prime}\x_{i+1} \oplus \sum_{k=1}^s \mathrm{h}_k+\scrO_{3},
\end{equation*}
such that $\le(\mathrm{h}_k)=1$, $\deg(\mathrm{h}_k)=2$ for each $k$. Set $h=\sum_{k=1}^s\mathrm{h}_k$, we are done.
\end{proof}

The following will allow us to monomialize the higher degree terms.
\begin{prop}\label{046}
Suppose that $\le(f_2)\leq 2$. For any $d \geq 3$, there exists a path degree $d-1$ right-equivalence,
\begin{equation*}
    \rho_d \colon f \stackrel{d-1}{\rightsquigarrow}    \sum_{i=1}^{m-1}\x_i^{\prime}\x_{i+1} \oplus( g_{<d}+h)+\scrO_{d+1}, 
\end{equation*}
such that $\le(h)=1$ and $\deg(h)=d$. 
\end{prop}
\begin{proof}
Decompose $g$ (from \ref{notation:mono}) by degree (wrt. $\x_i$, as in \ref{def:deg>}) as
$g=g_{<d}\oplus g_d\oplus g_{>d}$, and write
$g_d= \oplus_{k=1}^s c_k$
as a  $\C$-linear combination of degree $d$ cycles. Since there are only a finite number of cycles with degree $d$ on $Q_{n, I}$, $s$ is finite.

If $\le(c_1)=1$, set $\rho_{c_1}=\mathrm{Id}$. Otherwise, by \ref{044} there exists
\begin{equation*}
\rho_{c_1} \colon f \stackrel{d-1}{\rightsquigarrow}   \mathsf{f}_1 \colonequals \sum_{i=1}^{m-1}\x_i^{\prime}\x_{i+1} \oplus( g_{<d}  +\sum_{k = 2}^s c_k  +\mathrm{h}_1)+\scrO_{d+1}, 
\end{equation*}
such that $\le(\mathrm{h}_1)=1$ and $\deg(\mathrm{h}_1)=d$. %Furthermore, there also exists a degree $d-1$ right-equivalence $\rho_{c_1}^{\prime} \colon f_1 \stackrel{d-1}{\rightsquigarrow} f+\scrO_{d+1}$.  

If $\le(c_2)=1$, set $\rho_{c_2}=\mathrm{Id}$. Otherwise, again by \ref{044} there exists
\begin{equation*}
\rho_{c_2} \colon \mathsf{f}_1 \stackrel{d-1}{\rightsquigarrow}   \mathsf{f}_2 \colonequals \sum_{i=1}^{m-1}\x_i^{\prime}\x_{i+1}\oplus( g_{<d}  +\sum_{k = 3}^s c_k +\sum_{k=1}^2 \mathrm{h}_k)+\scrO_{d+1},  
\end{equation*}
such that $\le(\mathrm{h}_k)=1$ and $\deg(\mathrm{h}_k)=d$. %Furthermore, there also exists a degree $d-1$ right-equivalence $\rho_{c_2}^{\prime} \colon f_2 \stackrel{d-1}{\rightsquigarrow} f_1+\scrO_{d+1}$.

We repeat this process $s$ times and set $\rho_d\colonequals \rho_{c_s} \circ \dots \circ \rho_{c_2} \circ \rho_{c_1}$. %$\rho_d'= \rho'_{c_1} \circ \dots \circ \rho'_{c_s-1} \circ \rho'_{c_s}$. 
It follows that,
\begin{equation*}
    \rho_{d} \colon f \stackrel{d-1}{\rightsquigarrow}   \sum_{i=1}^{m-1}\x_i^{\prime}\x_{i+1} \oplus ( g_{<d}+\sum_{k=1}^s \mathrm{h}_k)+\scrO_{d+1}, 
\end{equation*}
such that $\le(\mathrm{h}_k )=1$, $\deg(\mathrm{h}_k )=d$ for each $k$. %Furthermore, $\rho_{d}^{\prime}$ satisfies $\rho_{d}^{\prime} \colon f_s \stackrel{d-1}{\rightsquigarrow} f+\scrO_{d+1}$.
Set $h=\sum_{k=1}^s \mathrm{h}_k$, we are done. 
\end{proof}

The following is the main result of this subsection.
\begin{theorem}\label{047}
For any reduced Type $A$ potential $f$ on $Q_{n, I}$, there exists a right-equivalence $\rho \colon f \rightsquigarrow f'$ such that $f'$ is a reduced monomialized Type A potential. 
\end{theorem}
\begin{proof}
We first apply the $\rho_2$ in \ref{045}, 
\begin{equation*}
    \rho_2 \colon f \stackrel{1}{\rightsquigarrow}   \mathsf{f}_1 \colonequals \sum_{i=1}^{m-1}\x_i^{\prime}\x_{i+1}\oplus \mathrm{h}_2+\scrO_{3},
\end{equation*}
such that $\le(\mathrm{h}_2)=1$ and $\deg(\mathrm{h}_2)=2$. 

Since $(\mathsf{f}_1)_2=\sum_{i=1}^{m-1}\x_i^{\prime}\x_{i+1}\oplus \mathrm{h}_2$, it is clear that $\le((\mathsf{f}_1)_2) \leq 2$.
Thus by \ref{046} applied to $\mathsf{f}_1$, there exists
\begin{equation*}
    \rho_3 \colon \mathsf{f}_1\stackrel{2}{\rightsquigarrow}   \mathsf{f}_2 \colonequals \sum_{i=1}^{m-1}\x_i^{\prime}\x_{i+1} \oplus \sum_{j=2}^3 \mathrm{h}_j+\scrO_{4}, 
\end{equation*}
such that $\le(h_3)=1$, $\deg(h_3)=3$. 
Iterating this procedure, for each $s \ge 3$ we obtain
\begin{equation*}
    \rho_s \circ \dots \circ \rho_3 \circ \rho_2 \colon f \rightsquigarrow  \mathsf{f}_s \colonequals \sum_{i=1}^{m-1}\x_i^{\prime}\x_{i+1} \oplus \sum_{j=2}^s \mathrm{h}_j+\scrO_{s+1}, 
\end{equation*}
such that $\le(\mathrm{h}_j)=1$ and $\deg(\mathrm{h}_j)=j$ for all $2\leq j\leq s$.

Since $\rho_d$ is a path degree $d-1$ right-equivalence for each $d \geq 2$ by \ref{045} and \ref{046}, by \ref{31} 
$\rho \colonequals \lim_{s \rightarrow \infty} \rho_s \circ \dots \circ \rho_3 \circ \rho_2$ exists, and further
\begin{equation*}
    \rho\colon f \rightsquigarrow \sum_{i=1}^{m-1}\x_i^{\prime}\x_{i+1} \oplus \sum_{j=2}^{\infty}\mathrm{h}_j,
\end{equation*}
such that $\le(\mathrm{h}_j)=1$ and $\deg(\mathrm{h}_j)=j$ for each $j$. 

Set $f'= \sum_{i=1}^{m-1}\x_i^{\prime}\x_{i+1}+\sum_{j=2}^{\infty}\mathrm{h}_j$. Since $\le(\mathrm{h}_j)=1$ for each $j$, $f'$ is a monomialized Type A potential. 
Moreover, since $f$ is reduced, $f'$ is also reduced.
\end{proof}

\subsection{Transform monomialized Type A potentials on $Q_{n, I}$ to $Q_n$}\label{trans}

To state unified results later, it will be convenient to show that any monomialized Type $A$ potential on $Q_{n, I}$ is isomorphic to a (possibly non-reduced) monomialized Type $A$ potential on $Q_n$. 
This required the following results, which provide a precise construction for adding a loop to $Q_{n, I}$.
\begin{lemma}\label{048}
Given any $ I \neq\{1,2, \dots, n\}$ and $i \in I^c$, let $\x_t$ be the loop at vertex $i$ of $Q_{n, I}$. Suppose that $h=\sum_{i=1}^{t-2}\x_i^{\prime}\x_{i+1} + \x_{t-1}'\x_{t+1}+\sum_{i=t+1}^{m-1}\x_i^{\prime}\x_{i+1}-\frac{1}{2}\x_t^2+\sum_{i\neq t}\sum_{j=2}^{\infty} \upkappa_{ij} \x_i^{j}$ where all $\upkappa_{ij} \in \C$.
There exists a right-equivalence
\begin{equation*}
    h \rightsquigarrow \sum_{i=1}^{m-1}\x_i^{\prime}\x_{i+1} +\sum_{i=1}^{m}\sum_{j=2}^{\infty} \upkappa_{ij}' \x_i^{j},
\end{equation*}
where $\upkappa_{ij}'$ are some scalars, and further $\upkappa'_{t2}\neq 0$.
\end{lemma}

\begin{proof}
Since $\x_t$ is the loop at vertex $i$ of $Q_{n, I}$, the quiver looks like the following locally,
\[
\begin{tikzpicture}[scale=2]
\node (e) at (2,0) [vertex] {};
\node (f) at (3,0) [vertex] {};
\node (h) at (4,0) [vertex] {};

\node (b2) at (3,-0.15) {$\scriptstyle i$};

\draw[->,bend left,looseness=0.7] (e) to node[above] {$\scriptstyle a_{t}$} (f);
\draw[<-,bend right,looseness=0.7] (e) to node[below] {$\scriptstyle b_{t}$} (f);
\draw[->]  (f) edge [in=55,out=120,loop,looseness=18] node[above]{$\scriptstyle \x_{t}$}  (f);
\draw[->,bend left,looseness=0.7] (f) to node[above] {$\scriptstyle a_{t+1}$} (h);
\draw[<-,bend right,looseness=0.7] (f) to node[below] {$\scriptstyle b_{t+1}$} (h);
\end{tikzpicture}
\]
Being a loop, $\x_t=a_t$, so applying the automorphism $a_t \mapsto a_t -b_{t-1}a_{t-1}-a_{t+1}b_{t+1}$ (in other words, $\x_t \mapsto \x_t - \x_{t-1}'-\x_{t+1}$) gives,
\begin{align}
 h & \mapsto   \sum_{i=1}^{t-2}\x_i^{\prime}\x_{i+1} + \x_{t-1}'\x_{t+1}+\sum_{i=t+1}^{m-1}\x_i^{\prime}\x_{i+1}-\frac{1}{2}(\x_t - \x_{t-1}'-\x_{t+1})^2+\sum_{i\neq t}\sum_{j=2}^{\infty} \upkappa_{ij} \x_i^{j} \notag \\
& \sim \sum_{i=1}^{m-1}\x_i^{\prime}\x_{i+1} -\frac{1}{2}\x_{t-1}^2-\frac{1}{2}\x_t^2-\frac{1}{2}\x_{t+1}^2+\sum_{i\neq t}\sum_{j=2}^{\infty} \upkappa_{ij}\x_i^{j}. \label{eqaution:421}
\end{align}
The last step uses the cyclic equivalences
$\x_{t-1}'\x_t \sim \x_t\x_{t-1}'$, $\x_t\x_{t+1} \sim \x_{t+1}\x_t$, and
$\x_{t-1}'\x_{t+1} \sim \x_{t+1}\x_{t-1}'$.
Then define the scalars $\upkappa'_{ij}$ by the identity
\[
\sum_{i=1}^{m}\sum_{j=2}^{\infty} \upkappa_{ij}' \x_i^{j}
= -\frac{1}{2}\x_{t-1}^2-\frac{1}{2}\x_t^2-\frac{1}{2}\x_{t+1}^2+\sum_{i\neq t}\sum_{j=2}^{\infty} \upkappa_{ij}\x_i^{j}.
\]
Since $\upkappa_{t2}'$ is the coefficient of $\x_t^2$ in \eqref{eqaution:421}, $\upkappa_{t2}'=-\frac{1}{2} \neq 0$.
\end{proof}

\begin{cor}\label{049}
Given any $I \neq \emptyset$, $i \in I$ and a monomialized Type A potential $f$ on $Q_{n, I}$, then there exists a monomialized Type A potential $g$ on $Q_{n, I/i}$ such that $\Jac(Q_{n, I}, f) \cong \Jac(Q_{n, I/i}, g)$ and $g$ contains the square of the loop at vertex $i$.
\end{cor}
\begin{proof}
Let $\x_t$ be the loop at vertex $i$ of $Q_{n, I/i}$. Locally, $Q_{n, I}$ and $Q_{n, I/i}$ look like the following, respectively.
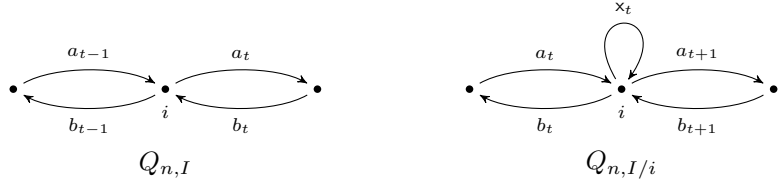
\[
\begin{tikzpicture}[scale=2]
\node (a) at (-1,0) [vertex] {};
\node (b) at (0,0) [vertex] {};
\node (c) at (1,0) [vertex] {};
\node (d) at (0,-0.5) {$Q_{n, I}$};
\node (a2) at (0,-0.15) {$\scriptstyle i$};

\draw[->,bend left,looseness=0.7] (a) to node[above] {$\scriptstyle a_{t-1}$} (b);
\draw[<-,bend right,looseness=0.7] (a) to node[below] {$\scriptstyle b_{t-1}$} (b);
\draw[->,bend left,looseness=0.7] (b) to node[above] {$\scriptstyle a_{t}$} (c);
\draw[<-,bend right,looseness=0.7] (b) to node[below] {$\scriptstyle b_{t}$} (c);

\node (e) at (2,0) [vertex] {};
\node (f) at (3,0) [vertex] {};
\node (h) at (4,0) [vertex] {};
\node (g) at (3,-0.5) {$Q_{n,I/i}$};

\node (b2) at (3,-0.15) {$\scriptstyle i$};

\draw[->,bend left,looseness=0.7] (e) to node[above] {$\scriptstyle a_{t}$} (f);
\draw[<-,bend right,looseness=0.7] (e) to node[below] {$\scriptstyle b_{t}$} (f);
\draw[->]  (f) edge [in=55,out=120,loop,looseness=18] node[above]{$\scriptstyle \x_{t}$}  (f);
\draw[->,bend left,looseness=0.7] (f) to node[above] {$\scriptstyle a_{t+1}$} (h);
\draw[<-,bend right,looseness=0.7] (f) to node[below] {$\scriptstyle b_{t+1}$} (h);
\end{tikzpicture}
\]

Relabeling the paths allows us to consider $f$ as a potential on $Q_{n, I/i}$.
More precisely, we replace the $a_k$ and $b_k$ in $f$ by $a_{k+1}$ and $b_{k+1}$ respectively for any $k \geq t$.
Then set $h\colonequals f- \frac{1}{2}\x_t^2$. It is clear that $ \Jac(Q_{n, I},f) \cong \Jac(Q_{n, I/ i},h)$. 

By \ref{048}, there exists a right-equivalence $h \rightsquigarrow g$ such that $g$ is a monomialized Type A potential on $Q_{n, I/ i}$ and $g$ contains $\x_t^2$. Thus $\Jac(Q_{n, I/ i},h) \cong \Jac(Q_{n, I/ i},g)  $, and so $\Jac(Q_{n, I},f) \cong \Jac(Q_{n, I/ i},g)$.
\end{proof}

\begin{prop}\label{410}
Given any $I$ and a monomialized Type $A$ potential $f=\sum_{i=1}^{m-1}\x_i^{\prime}\x_{i+1} + \sum_{i=1}^m\sum_{j=2}^{\infty} \upkappa'_{ij} \x_i^{j}$ on $Q_{n, I}$ where all $\upkappa'_{ij} \in \C$, then there exists a monomialized Type $A$ potential $g$ on $Q_n$, namely
\begin{equation*}
  g= \sum_{i=1}^{2n-2}\x_i^{\prime}\x_{i+1} + \sum_{i=1}^{2n-1}\sum_{j=2}^{\infty} \upkappa_{ij} \x_i^{j}
\end{equation*}
for some $\upkappa_{ij} \in \C$, such that $\Jac(Q_{n}, g) \cong \Jac(Q_{n, I}, f)$ and $\upkappa_{2i-1,2}\neq 0$ for each $i \in  I$.
\end{prop}
\begin{proof}
If $I=\emptyset$, there is nothing to prove. Otherwise, given any $i \in  I$, by \ref{049} there exist a monomialized Type $A$ potential $\mathsf{g}_1$ on $Q_{n, I \backslash i  }$ such that $\Jac(Q_{n,I \backslash i},\mathsf{g}_1)\cong \Jac(Q_{n,I},f)$, where $\mathsf{g}_1$ contains the square of the loop at vertex $i$.

Similarly, by \ref{049} we can repeat the same argument to $\mathsf{g}_1$ on $Q_{n, I \backslash i }$ and any $j \in I \backslash \{i\}$ to construct a monomialized Type $A$ potential $\mathsf{g}_2$ on $Q_{n, I \backslash \{i,j\} }$ such that $\Jac(Q_{n,I \backslash \{i,j\}}, \mathsf{g}_2)\cong \Jac(Q_{n,I \backslash i},\mathsf{g}_1)$, where $\mathsf{g}_2$ contains the square of the loop at vertex $i$ and vertex $j$.
 
Set $s=|I|$. Thus we can repeat this process $s$ times to construct a monomialized Type $A$ potential $\mathsf{g}_s$ on $Q_{n,\emptyset}$ such that $\Jac(Q_{n,\emptyset}, \mathsf{g}_{s})\cong \Jac(Q_{n,I}, f)$, and $\mathsf{g}_s$ contains the square of all the loops at all vertices $i \in I$. 

Set $g \colonequals \mathsf{g}_{s}$. Since $\upkappa_{2i-1,2}$ are the coefficients of the square of the loops at the vertices $i \in I$, the statement follows. 
\end{proof}

\section{Geometric Realisation}\label{GR}
In \S\ref{Gr} we show that every Type~$A$ potential on $Q_{n,I}$ is realised by a crepant resolution of a $cA_n$ singularity (see~\ref{511}), thereby proving the Realisation Conjecture of Brown--Wemyss~\cite{BW2} for Type~$A$
potentials.  Conversely, \S\ref{Cor} establishes the reverse direction (see~\ref{514}) and then proves a correspondence between crepant resolutions of $cA_n$ singularities and our intrinsic Type~$A$ potentials on $Q_n$
(see~\ref{515} and~\ref{516}).

\subsection{Geometric Realisation}\label{Gr}

In this subsection we prove that, given any Type~$A$ potential $f$ on $Q_{n,I}$, there exists a crepant resolution
$\uppi$ of a $cA_n$ singularity such that $\Jac(f)\cong \Lambda_{\mathrm{con}}(\uppi)$ (see~\ref{511}).  This verifies
the Realisation Conjecture of Brown--Wemyss~\cite{BW2} in the setting of Type~$A$ potentials.

\begin{notation}\label{notation: gr1}
Fix a monomialised Type~$A$ potential $f$ on $Q_n$ of the form
\begin{equation}\label{f1}
    f=\sum_{i=1}^{2n-2}\x_i^{\prime}\x_{i+1} + \sum_{i=1}^{2n-1}\sum_{j=2}^{\infty}\upkappa_{ij}\x_i^j.
\end{equation}

Consider the system of equations \eqref{501} in unknowns $g_0,\dots,g_{2n}\in\C\lal x,y\ral$:
\begin{align}
   g_0+\sum_{j=2}^{\infty}j\upkappa_{1j}g_1^{j-1}+g_2&=0 \notag\\
     g_1+\sum_{j=2}^{\infty}j\upkappa_{2j}g_2^{j-1}+g_3&=0  \notag\\
 &\vdotswithin{=}  \label{501} \\
   g_{2n-2}+\sum_{j=2}^{\infty}j\upkappa_{2n-1,j}g_{2n-1}^{j-1}+g_{2n}&=0.  \notag
\end{align}
\end{notation}

The following lemma allows us to construct the geometric realisation of $f$ \eqref{f1} in \ref{notation: gr2} \eqref{gr2 1} by the system of equations \eqref{501}.
\begin{lemma}\label{054}
With notation in \textnormal{\ref{notation: gr1}}, fix some integer $t$ satisfying $0 \leq t \leq 2n-1$, and set $g_t=y$, $g_{t+1}=x$. 
Then there exists a sequence $(g_0,g_1,\dots,g_{2n})$ satisfying \eqref{501} such that each
$g_s\in (x,y ) \subseteq \C\lal x,y\ral$ is a prime element with a linear term.
Moreover, 
\begin{enumerate}
\item For any $0\leq s \leq 2n-1$, $( g_s,g_{s+1}) = (x,y)$.
\item For any $1\leq s \leq 2n-1$, $(g_{s-1},g_{s+1})\subsetneq (x,y)$ when $\upkappa_{s2}= 0$, and $( g_{s-1},g_{s+1}) = ( x,y )$ when $\upkappa_{s2}\neq 0$.
\end{enumerate}
\end{lemma}

\begin{proof}
We start with the equation $g_t+\sum_{j=2}^{\infty}j\upkappa_{t+1,j}g_{t+1}^{j-1}+g_{t+2}=0 $ in \eqref{501} which defines $g_{t+2}= -y-\sum_{j=2}^{\infty}j\upkappa_{t+1,j}x^{j-1} \in  (  x,y)$. 
Then we consider $g_{t+1}+\sum_{j=2}^{\infty}j\upkappa_{t+2,j}g_{t+2}^{j-1}+g_{t+3}=0$ which also defines $g_{t+3} \in  (  x,y )$. Thus we can repeat this process to construct $g_s\in  ( x,y )$ for $t+2 \leq s \leq 2n$.
Similarly, the equation $g_{t-1}+\sum_{j=2}^{\infty}j\upkappa_{t,j}g_{t}^{j-1}+g_{t+1}=0$ defines $g_{t-1} \in  ( x,y )$. We can repeat this process to construct $g_s\in  ( x,y)$ for $0 \leq s \leq t-1$.

(1) For any $0 \leq s \leq 2n-2$, using $g_s+\sum_{j=2}^{\infty}j\upkappa_{s+1,j}g_{s+1}^{j-1}+g_{s+2}=0 $ in \eqref{501}, we have $( g_s,g_{s+1})= (g_{s+1},g_{s+2})$. Moving either to the left or right until we hit $t$, it follows that $( g_s,g_{s+1})= ( g_{t},g_{t+1})= ( y,x )$. 
Hence $(g_s,g_{s+1})=(x,y)$ for all $0\leq s\leq 2n-1$. 

In particular, no $g_s$ can lie in $( x,y )^2$,
so each $g_s$ has a nonzero linear term. Moreover, since $(g_s,g_{s+1} )$ is the maximal ideal of $\C\lal x,y\ral$,
each $g_s$ is irreducible, hence prime.

(2) For any $1\leq s \leq 2n-1$, using $g_{s-1}+\sum_{j=2}^{\infty}j\upkappa_{sj}g_{s}^{j-1}+g_{s+1}=0 $ in \eqref{501}, we have $( g_{s-1},g_{s+1} ) =( g_{s-1},\sum_{j=2}^{\infty}j\upkappa_{sj}g_{s}^{j-1} )$. Thus, if $\upkappa_{s2}= 0$ then $( g_{s-1},g_{s+1}) \subsetneq ( x,y)$, and if $\upkappa_{s2}\neq 0$ then $(g_{s-1},g_{s+1}) = ( g_{s-1},g_{s})$ which equals $(x,y)$ by (1).
\end{proof}

\begin{notation}\label{notation: gr2}
For any $t$ with $0 \leq t \leq 2n-1$, \ref{054} calculates a solution of \eqref{501}.
Fix any such solution, say $(g_0,g_1, \dots ,g_{2n})$. From this, we adopt the following notation.
\begin{enumerate}
\item We first define the $cA_n$ singularity
\begin{equation*}
\scrR \colonequals \frac{\mathbb{C} \lal u, v, x, y \ral}{uv-g_{0}g_{2} \dots g_{2n}}.
\end{equation*}
Note that each $g_i$ is a prime element of $\C \lal x,y \ral$ with a linear term by \ref{054}, so $\scrR$ is a $cA_n$ singularity.
Then consider the $\CM$ $\scrR$-module 
\[
M\colonequals \scrR \oplus (u,g_{0}) \oplus (u,g_{0}g_{2}) \oplus \ldots \oplus (u,\prod_{j=0}^{n-1} g_{2j}).
\] \label{gr2 1}

\item We next define
\begin{equation*}
 \scrS_1 \colonequals \frac{ \C \lal u,v,x_0,x_1, x_2,x_3 \dots , x_{2n-1},x_{2n} \ral}{uv- x_0x_2 \dots x_{2n}}.
\end{equation*} \label{gr2 4}

\item Define a sequence $h_1, h_2,\dots , h_{2n-1} \in \scrS_1$ to be
\begin{align*}
    h_i \colonequals x_{i-1}+\sum_{j=2}^{\infty}j\upkappa_{ij}x_i^{j-1}+x_{i+1},
\end{align*}
and set $\scrS_i \colonequals \scrS_1/(h_1, h_2, \dots, h_{i-1} )$ for $2 \leq i \leq 2n$. \label{gr2 5}

\item For $1 \leq i \leq 2n$, by abuse of notation we regard
$(u, x_{0}), \ (u, x_{0}x_{2}), \ \ldots, \ \bigl(u, \prod_{j=0}^{\,n-1} x_{2j}\bigr)$ as $\scrS_i$-modules.  
Then we define the $\scrS_i$-module
\[
N_i \colonequals \scrS_i \oplus (u, x_{0}) \oplus (u, x_{0}x_{2}) \oplus \cdots \oplus \bigl(u, \prod_{j=0}^{\,n-1} x_{2j}\bigr).
\] 

\item Write $\uppi_1 \colon \scrX_1 \rightarrow \Spec \scrS_1$ for the universal flop of $\Spec \scrS_1$ corresponding to $N_1$, as constructed in the complete local setting in \cite[\S5]{IW2}. 
For $2 \leq i \leq 2n$, consider the morphism $\Spec \scrS_{i} \rightarrow \Spec \scrS_1$, and the fiber product $\scrX_i \colonequals \scrX_1 \times_{\Spec\scrS_1} \Spec\scrS_i$. These morphisms fit into the following commutative diagram.
\[
\begin{tikzpicture}[bend angle=8, looseness=1.2]

\node (a2) at (-2,0) {$\Spec \scrS_1$};
\node (a3) at (-4,0) {$\Spec \scrS_2$};
\node (a4) at (-8,0) {$\Spec \scrS_{2n}$};

\node (b2) at (-2,1.5) {$\scrX_1$};
\node (b3) at (-4,1.5) {$\scrX_2$};
\node (b4) at (-8,1.5) {$\scrX_{2n}$};

\draw[->] (-2,1.2) -- node[right] {$\scriptstyle \uppi_1$} (-2,0.3);
\draw[->] (-4,1.2) -- node[right] {$\scriptstyle \uppi_2$} (-4,0.3);
\draw[->] (-8,1.2) -- node[right] {$\scriptstyle \uppi_{2n}$} (-8,0.3);

\draw[->] (-3.7,1.5) -- node[above] {} (-2.3,1.5);
\draw[->] (-5.3,1.5) -- node[above] {} (-4.3,1.5);
\node (c2) at (-6,1.5) {$\dots$};
\draw[->] (-7.6,1.5) -- node[above] {} (-6.7,1.5);

\draw[->] (-3.3,0) -- node[above] {} (-2.7,0);
\draw[->] (-5.3,0) -- node[above] {} (-4.7,0);
\node (c2) at (-6,0) {$\dots$};
\draw[->] (-7.3,0) -- node[above] {} (-6.7,0);
\end{tikzpicture}
\] \label{gr2 7}

\item Consider the following quiver $Q$.
\[
\begin{tikzpicture}[bend angle=8, looseness=1.2]
\node (a) at (0,0)  {$1$};
\node (b) at (2.5,0)  {$2$};
\node (c) at (5,0) {$3$};
\node (d) at (6.5,0) {$\hdots \hdots$};
\node (e) at (8,0) {$n-1$};
\node (f) at (11,0)  {$n$};
\node (g) at (5.5,-3) {$0$};

\draw[<-,bend right,looseness=0.7] (g) to node [gap] {$\scriptstyle b_{0}$} (a);
\draw[->,bend left,looseness=0.7] (g) to node [gap] {$\scriptstyle a_{0}$} (a);
\draw[<-,bend right] (a) to node [gap] {$\scriptstyle b_{2}$} (b);
\draw[->,bend left] (a) to node [gap] {$\scriptstyle a_{2}$} (b);
\draw[<-,bend right] (b) to node[gap] {$\scriptstyle b_{4}$} (c);
\draw[->,bend left] (b) to node[gap] {$\scriptstyle a_{4}$} (c);
\draw[<-,bend right] (e) to node[gap] {$\scriptstyle b_{2n-2}$} (f);
\draw[->,bend left] (e) to node[gap] {$\scriptstyle a_{2n-2}$} (f);

\draw[<-,bend right,looseness=0.7] (g) to node[gap] {$\scriptstyle a_{2n}$} (f);
\draw[->,bend left,looseness=0.7] (g) to node[gap] {$\scriptstyle b_{2n}$} (f);
\draw[<-]  (g) edge [in=-120,out=-55,loop,looseness=8] node[below] {$\scriptstyle l_{0,0}, l_{0,1}, \dots,  l_{0,2n}$} (g);
\draw[<-]  (a) edge [in=120,out=55,loop,looseness=8] node[above] {$\scriptstyle l_{1,0}, l_{1,1}, \dots,  l_{1,2n}$} (a);
\draw[<-]  (b) edge [in=120,out=55,loop,looseness=8] node[above] {$\scriptstyle l_{2,0}, l_{2,1}, \dots,  l_{2,2n}$} (b);
\draw[<-]  (c) edge [in=120,out=55,loop,looseness=8] node[above] {$\scriptstyle l_{3,0}, l_{3,1}, \dots,  l_{3,2n}$} (c);
\draw[<-]  (e) edge [in=120,out=55,loop,looseness=8] node[above] {$\scriptstyle l_{n-1,0}, l_{n-1,1}, \dots,  l_{n-1,2n}$} (e);
\draw[<-]  (f) edge [in=120,out=55,loop,looseness=8] node[above] {$\scriptstyle l_{n,0}, l_{n,1}, \dots,  l_{n,2n}$} (f);
\end{tikzpicture}
\]
Then define the relations $R_1$ of $Q$ as follows.
\begin{equation}\label{22}
R_1 \colonequals
\begin{cases}
    l_{t,i}a_{2t}=a_{2t}l_{t+1,i}, \ l_{t+1,i}b_{2t}=b_{2t}l_{t,i}, \ l_{t,i}l_{t,j}=l_{t,j}l_{t,i}, \\
    l_{t,2t}=a_{2t}b_{2t}, \ l_{t+1,2t}=b_{2t}a_{2t} \text{ for any }t \in \mathbb{Z}/(n+1) \text{ and } 0 \leq i,  j \leq 2n.
\end{cases}
\end{equation}
For $2 \leq s \leq 2n$, define $R_s$ to be $R_1$ with the additional relations
\begin{equation} \label{21}
 l_{t,i-1}+\sum_{j=2}^{\infty}j\upkappa_{ij}l_{t,i}^{j-1}+l_{t,i+1}=0
  \text{ for any } 0 \leq t \leq n \text{ and } 1 \leq i \leq s-1.
\end{equation} \label{gr2 9}
\end{enumerate}
\end{notation}

To prepare for the main construction \ref{gr}, we now establish in \ref{lemma: gr1}–\ref{lemma: gr3} a quiver presentation of the NCCR $\End_{\scrR}(M)$, where $\scrR$ and $M$ are as in \ref{notation: gr2}\eqref{gr2 1}.

\begin{lemma}\label{lemma: gr1}
With notation in \textnormal{\ref{notation: gr2}}, for $2 \leq k \leq 2n$, $\scrS_k$ is an integral domain and normal. Furthermore, there exists a ring isomorphism $\varphi \colon \scrS_{2n} \xrightarrow{\sim} \scrR$ such that $\varphi(N_{2n}) = M$.
\end{lemma}
\begin{proof}
Fix some $k$ with $2 \leq k \leq 2n$.
By the definition in \ref{notation: gr2}\eqref{gr2 4} and \ref{notation: gr2}\eqref{gr2 5},
\begin{equation*}
    \scrS_k \cong \frac{ \C \lal u,v,x_0,x_1, x_2, \dots , x_{2n} \ral}{(uv- x_0x_2 \dots x_{2n},h_1,h_2,\dots, h_{k-1})},
\end{equation*}
where each $h_i = x_{i-1}+\sum_{j=2}^{\infty}j\upkappa_{ij}x_i^{j-1}+x_{i+1}$.

As in the proof of \ref{054}, the relations $h_1,\dots,h_{k-1}$ allow us to eliminate $x_2,\dots,x_k$ recursively, expressing each $x_i$ ($2\leq i\leq k$) as a formal power series in $x_0$ and $x_1$.
Write this as $x_i=H_i(x_0,x_1)$.

Thus, when $k$ is even, 
\begin{equation*}
    \scrS_k \cong \frac{ \C \lal u,v,x_0,x_1, x_{k+1}, x_{k+2}, \dots , x_{2n} \ral}{uv- x_0H_2 H_4 \dots H_k x_{k+2} \dots x_{2n}}.
\end{equation*}
When $k$ is odd,
\begin{equation*}
    \scrS_k \cong \frac{ \C \lal u,v,x_0,x_1, x_{k+1}, x_{k+2}, \dots , x_{2n} \ral}{uv- x_0H_2 H_4 \dots H_{k-1} x_{k+1} \dots x_{2n}}.
\end{equation*}
In both cases, $\scrS_k$ is an integral domain and normal by e.g. \cite[4.1.1]{S}. 

Then we prove that $\scrS_{2n} \cong \scrR$. Recall from \ref{054} that we start with $g_t=y$ and $g_{t+1}=x$ and then construct $(g_0, g_1, \dots, g_{2n})$ where each $g_i \in \C \lal x, y \ral$ using the equation system \eqref{501}. Then, in \ref{notation: gr2}\eqref{gr2 1}, these $g_i$ were used to define $\scrR$. 

On the other hand, 
\begin{equation*}
    \scrS_{2n} \cong \frac{ \C \lal u,v,x_0,x_1, x_2, \dots , x_{2n} \ral}{(uv- x_0x_2 \dots x_{2n},h_1,h_2,\dots,h_{2n-1})}.
\end{equation*}
Similar to \ref{054}, we can express each $x_s$ as a formal power series of $x_t$ and $x_{t+1}$ using $h_1,h_2,\dots,h_{2n-1}$.
More precisely, in $\scrS_{2n}$ we have $x_s=g_s(x_{t+1},x_t)$ for all $s$, where $g_s$ is the power series in
variables $(x,y)$ obtained from \eqref{501} with $g_t=y$ and $g_{t+1}=x$. Hence

\begin{equation*}
    \scrS_{2n}  \cong \frac{ \C \lal u,v,x_t,x_{t+1} \ral}{uv- g_0(x_{t+1},x_t)g_2(x_{t+1},x_t) \dots g_{2n}(x_{t+1},x_t)}.
\end{equation*}
Define a ring homomorphism $\varphi \colon \scrS_{2n} \to \scrR$ by 
$u \mapsto u$, $v \mapsto v$, $x_{t+1} \mapsto x$, and $x_t \mapsto y$. 
It is immediate that $\varphi$ is an isomorphism, and moreover $\varphi(N_{2n}) = M$.
\end{proof}

With notation as in \textnormal{\ref{notation: gr2}}, let $\uppi_1\colon \scrX_1\to \Spec \scrS_1$ be the universal resolution with $\Lambda(\uppi_1)\cong \End_{\scrS_1}(N_1)$ \cite[\S5]{IW2}. As shown in
Appendix~\ref{thm31}, there is an isomorphism
\[
\End_{\scrS_1}(N_1)\cong \C\lbl Q\rbl/R_1,
\]
where $Q$ and $R_1$ are as in \eqref{22}.

For notational convenience, set $\Lambda \colonequals  \C  \lbl Q \rbl /R_1$. 
By \cite[6.2]{W1}, $\scrX_1$ is isomorphic to a moduli scheme of stable representations of $\Lambda $, of dimension vector $\updelta=(1,1, \dots , 1)$ and stability $\upvartheta =(-n,1,1, \dots , 1)$ where the $-n$ sits at vertex $0$ of $Q$. 
In notation, $\scrX_1 \cong \mathcal{M}_{\updelta}^\upvartheta(\Lambda)$, which is the moduli space of $\upvartheta$-stable representations of dimension vector $\updelta$.

Moreover, as in \cite[\S3]{W3}, the moduli space $\mathcal{M}_{\updelta}^{\upvartheta}(\Lambda)$ is covered by
$n+1$ affine charts $\scrU_{10},\dots,\scrU_{1n}$.
Accounting for the relations $R_1$ \eqref{22}, the first affine chart $\scrU_{10}$ is parameterised by
\[
\begin{tikzpicture}[bend angle=8, looseness=1.2]
\node (a) at (0,0)  {$\C$};
\node (b) at (2.8,0)  {$\C$};
\node (c) at (5.6,0) {$\C$};
\node (d) at (6.9,0) {$\hdots \hdots$};
\node (e) at (8.3,0) {$\C$};
\node (f) at (11.1,0)  {$\C$};
\node (g) at (5.5,-3) {$\C$};

\draw[<-,bend right,looseness=0.7] (g) to node [gap] {$\scriptstyle \mathsf{b}_{0}$} (a);
\draw[->,bend left,looseness=0.7] (g) to node [gap] {$\scriptstyle 1$} (a);
\draw[<-,bend right] (a) to node [gap] {$\scriptstyle \mathsf{b}_{2}$} (b);
\draw[->,bend left] (a) to node [gap] {$\scriptstyle 1$} (b);
\draw[<-,bend right] (b) to node[gap] {$\scriptstyle \mathsf{b}_{4}$} (c);
\draw[->,bend left] (b) to node[gap] {$\scriptstyle 1$} (c);
\draw[<-,bend right] (e) to node[gap] {$\scriptstyle \mathsf{b}_{2n-2}$} (f);
\draw[->,bend left] (e) to node[gap] {$\scriptstyle 1$} (f);

\draw[<-,bend right,looseness=0.7] (g) to node[gap] {$\scriptstyle \mathsf{a}_{2n}$} (f);
\draw[->,bend left,looseness=0.7] (g) to node[gap] {$\scriptstyle \mathsf{b}_{2n}$} (f);
\draw[<-]  (g) edge [in=-120,out=-55,loop,looseness=8] node[below] {$\scriptstyle \mathsf{b}_0, \mathsf{b}_2,\dots, \mathsf{b}_{2n-2}, \x_{2n},$} (g);
\draw[<-]  (a) edge [in=120,out=55,loop,looseness=8] node[above] {$\scriptstyle \x_1,\x_3, \dots , \x_{2n-1}$} (a);
\draw[<-]  (b) edge [in=120,out=55,loop,looseness=8] node[above] {$\scriptstyle \x_1,\x_3, \dots , \x_{2n-1}$} (b);
\draw[<-]  (c) edge [in=120,out=55,loop,looseness=8] node[above] {$\scriptstyle \x_1,\x_3, \dots , \x_{2n-1}$} (c);
\draw[<-]  (e) edge [in=120,out=55,loop,looseness=8] node[above] {$\scriptstyle \x_1,\x_3, \dots , \x_{2n-1}$} (e);
\draw[<-]  (f) edge [in=120,out=55,loop,looseness=8] node[above] {$\scriptstyle \x_1,\x_3, \dots , \x_{2n-1}$} (f);

\node (g1) at (5.5,-4.5) {$\scriptstyle \x_1,\x_3, \dots , \x_{2n-1}$};
\node (a1) at (0,1.5) {$\scriptstyle \mathsf{b}_0, \mathsf{b}_2,\dots, \mathsf{b}_{2n-2}, \x_{2n},$};
\node (b1) at (2.8,1.5) {$\scriptstyle \mathsf{b}_0, \mathsf{b}_2,\dots, \mathsf{b}_{2n-2}, \x_{2n},$};
\node (c1) at (5.6,1.5) {$\scriptstyle \mathsf{b}_0, \mathsf{b}_2,\dots, \mathsf{b}_{2n-2}, \x_{2n},$};
\node (e1) at (8.3,1.5) {$\scriptstyle \mathsf{b}_0, \mathsf{b}_2,\dots, \mathsf{b}_{2n-2}, \x_{2n},$};
\node (f1) at (11.1,1.5) {$\scriptstyle \mathsf{b}_0, \mathsf{b}_2,\dots, \mathsf{b}_{2n-2}, \x_{2n},$};
\end{tikzpicture}
\]
where $\x_{2n} = \mathsf{a}_{2n}\mathsf{b}_{2n}$ and (since we work on the completed path algebra) all cycles are nilpotent.  We claim that  $\scrU_{10} \cong \Spec \scrA_{10}$ where 
\begin{equation}\label{A10}
\scrA_{10} \colonequals \frac{\C \lal \mathsf{b}_0,\mathsf{b}_2,\dots ,\mathsf{b}_{2n-2},\mathsf{a}_{2n}, \x_1,\x_3,\dots, \x_{2n-1},\x_{2n}, \mathsf{v} \ral[\mathsf{b}_{2n}]}{(\x_{2n}-\mathsf{a}_{2n}\mathsf{b}_{2n}, \mathsf{v}-\mathsf{b}_0\mathsf{b}_2\dots \mathsf{b}_{2n}) }.   
\end{equation}
Indeed, on $\scrU_{10}$ we fix the clockwise arrows marked $1$ (except $\mathsf{a}_{2n}$), and then use all the relations in
$R_1$ to express all remaining arrow-parameters in terms of the displayed generators. 
So the question boils down to understanding nilpotent cycles.  It is clear that $\x_1,\x_3,\hdots,\x_{2n-3},\x_{2n-1},\x_{2n},\mathsf{b}_0,\mathsf{b}_2,\hdots,\mathsf{b}_{2n-2}$ are cycles, as is $\mathsf{a}_{2n}$ (once composed with all clockwise arrows marked $1$), thus they are nilpotent. As is $\mathsf{b}_{2n}\mathsf{b}_{2n-2}\hdots\mathsf{b}_{0}$.

There is no condition on $\mathsf{b}_{2n}$, so it is a polynomial variable. Introducing a new completion variable $\mathsf{v}$ to capture the nilpotency of $\mathsf{b}_{2n}\mathsf{b}_{2n-2}\hdots\mathsf{b}_{0}$, which has a mix of both polynomial and completion variables, the claim \eqref{A10} follows.

Moreover, $\uppi_1|_{\scrU_{10}} \colon \scrU_{10} \rightarrow \Spec \scrS_1$ is induced by the ring homomorphism $\varphi_{10} \colon  \scrS_1 \rightarrow \scrA_{10}$
\begin{align}\label{gather}
    & x_0 \mapsto \mathsf{b}_0, \quad x_2 \mapsto \mathsf{b}_2, \quad \dots , \quad  x_{2n-2} \mapsto \mathsf{b}_{2n-2}, \quad  x _{2n} \mapsto \x_{2n},   \nonumber \\
    & x_1 \mapsto \x_1, \quad  x_3 \mapsto \x_3, \quad \dots, \quad x_{2n-1} \mapsto \x_{2n-1}, \quad u \mapsto \mathsf{a}_{2n}, \quad  v \mapsto \mathsf{v}.
\end{align}

Similarly, the second affine chart $\scrU_{11}$ is parameterised by
\[
\begin{tikzpicture}[bend angle=8, looseness=1.2]
\node (a) at (0,0)  {$\C$};
\node (b) at (2.5,0)  {$\C$};
\node (c) at (5,0) {$\C$};
\node (d) at (6.5,0) {$\hdots \hdots$};
\node (e) at (8,0) {$\C$};
\node (f) at (10.5,0)  {$\C$};
\node (g) at (5.25,-3) {$\C$};

\draw[<-,bend right,looseness=0.7] (g) to node [gap] {$\scriptstyle \mathsf{b}_{0}$} (a);
\draw[->,bend left,looseness=0.7] (g) to node [gap] {$\scriptstyle 1$} (a);
\draw[<-,bend right] (a) to node [gap] {$\scriptstyle \mathsf{b}_{2}$} (b);
\draw[->,bend left] (a) to node [gap] {$\scriptstyle 1$} (b);
\draw[<-,bend right] (b) to node[gap] {$\scriptstyle \mathsf{b}_{4}$} (c);
\draw[->,bend left] (b) to node[gap] {$\scriptstyle 1$} (c);
\draw[<-,bend right] (e) to node[gap] {$\scriptstyle \mathsf{b}_{2n-2}$} (f);
\draw[->,bend left] (e) to node[gap] {$\scriptstyle \mathsf{a}_{2n-2}$} (f);
\draw[<-,bend right,looseness=0.7] (g) to node[gap] {$\scriptstyle \mathsf{a}_{2n}$} (f);
\draw[->,bend left,looseness=0.7] (g) to node[gap] {$\scriptstyle 1$} (f);

\draw[<-]  (g) edge [in=-120,out=-55,loop,looseness=8] node[below] {$\scriptstyle \mathsf{b}_0, \mathsf{b}_2, \dots, \mathsf{b}_{2n-4}, \x_{2n-2},\mathsf{a}_{2n},$} (g);

\draw[<-]  (a) edge [in=120,out=55,loop,looseness=8] node[above] {$\scriptstyle   \x_1,\x_3, \dots , \x_{2n-1}$} (a);
\draw[<-]  (b) edge [in=120,out=55,loop,looseness=8] node[above] {$\scriptstyle  \x_1,\x_3, \dots , \x_{2n-1}$} (b);
\draw[<-]  (c) edge [in=120,out=55,loop,looseness=8] node[above] {$\scriptstyle  \x_1,\x_3, \dots , \x_{2n-1}$} (c);
\draw[<-]  (e) edge [in=120,out=55,loop,looseness=8] node[above] {$\scriptstyle  \x_1,\x_3, \dots , \x_{2n-1}$} (e);
\draw[<-]  (f) edge [in=120,out=55,loop,looseness=8] node[above] {$\scriptstyle  \x_1,\x_3, \dots , \x_{2n-1}$} (f);

\node (a1) at (0,1.4) {$\scriptstyle \x_{2n-2},\mathsf{a}_{2n},$};
\node (a2) at (2.5,1.4) {$\scriptstyle \x_{2n-2},\mathsf{a}_{2n},$};
\node (a3) at (5,1.4) {$\scriptstyle \x_{2n-2},\mathsf{a}_{2n},$};
\node (a4) at (8,1.4) {$\scriptstyle \x_{2n-2},\mathsf{a}_{2n},$};
\node (a5) at (10.5,1.4) {$\scriptstyle \x_{2n-2},\mathsf{a}_{2n},$};

\node (g1) at (5.25,-4.5) {$\scriptstyle \x_1,\x_3, \dots , \x_{2n-1}$};
\node (a1) at (0,1.8) {$\scriptstyle  \mathsf{b}_0, \mathsf{b}_2, \dots, \mathsf{b}_{2n-4},$};
\node (b1) at (2.5,1.8) {$\scriptstyle  \mathsf{b}_0, \mathsf{b}_2, \dots, \mathsf{b}_{2n-4},$};
\node (c1) at (5,1.8) {$\scriptstyle  \mathsf{b}_0, \mathsf{b}_2, \dots, \mathsf{b}_{2n-4},$};
\node (e1) at (8,1.8) {$\scriptstyle  \mathsf{b}_0, \mathsf{b}_2, \dots, \mathsf{b}_{2n-4},$};
\node (f1) at (10.5,1.8) {$\scriptstyle \mathsf{b}_0, \mathsf{b}_2, \dots, \mathsf{b}_{2n-4},$};
\end{tikzpicture}
\]

where $\x_{2n-2} = \mathsf{a}_{2n-2} \mathsf{b}_{2n-2}$ and (since we work on the completed path algebra) all cycles are nilpotent. We claim that $\scrU_{11} \cong \Spec \scrA_{11}$ where
\begin{equation}\label{A11}
\scrA_{11} \colonequals \frac{\C \lal \mathsf{b}_0,\mathsf{b}_2,\dots ,\mathsf{b}_{2n-4},\mathsf{a}_{2n}, \x_1,\x_3 ,\dots, \x_{2n-1},\x_{2n-2}, \mathsf{u}, \mathsf{v} \ral[\mathsf{a}_{2n-2},\mathsf{b}_{2n-2}]}{(\x_{2n-2}-\mathsf{a}_{2n-2}\mathsf{b}_{2n-2}, \mathsf{u}-\mathsf{a}_{2n-2}\mathsf{a}_{2n},\mathsf{v}-\mathsf{b}_0\mathsf{b}_2\dots \mathsf{b}_{2n-2}) }.    
\end{equation}

Similarly, on $\scrU_{11}$ we fix the clockwise arrows marked $1$ (except $\mathsf{a}_{2n-2}$), and then use all the relations in
$R_1$ to express all remaining arrow-parameters in terms of the displayed generators. Clearly $\x_1,\x_3,\hdots,\x_{2n-3},\x_{2n-1},\x_{2n-2},\mathsf{b}_0,\mathsf{b}_2,\hdots,\mathsf{b}_{2n-4}, \mathsf{a}_{2n}$ are cycles, as is  $\mathsf{a}_{2n-2}\mathsf{a}_{2n}$ (once composed with all clockwise arrows marked $1$), thus they are nilpotent. As is $\mathsf{b}_{2n-2}\mathsf{b}_{2n-4}\hdots\mathsf{b}_{0}$.

There is no condition on $\mathsf{a}_{2n-2}$ and $\mathsf{b}_{2n-2}$, so they are polynomial variables. Introducing new completion variables $\mathsf{u}$ and $\mathsf{v}$ to capture the nilpotency of $\mathsf{a}_{2n-2}\mathsf{a}_{2n}$ and $\mathsf{b}_{2n-2}\mathsf{b}_{2n-4}\hdots\mathsf{b}_{0}$ respectively, which have a mix of both polynomial and completion variables, the claim \eqref{A11} follows.

Moreover, $\uppi_1|_{\scrU_{11}} \colon \scrU_{11} \rightarrow \Spec \scrS_1$ is induced by the ring homomorphism $\varphi_{11} \colon  \scrS_1 \rightarrow \scrA_{11}$
\begin{align}\label{gather2}
& x_0 \mapsto \mathsf{b}_0, \quad x_2 \mapsto \mathsf{b}_2, \quad \dots , \quad x_{2n-4} \mapsto \mathsf{b}_{2n-4}, \quad  x_{2n-2} \mapsto \x_{2n-2},   \quad x_{2n} \mapsto \mathsf{a}_{2n},    \nonumber\\
& x_1 \mapsto \x_1, \quad  x_3 \mapsto \x_3, \quad \dots, \quad x_{2n-1} \mapsto \x_{2n-1}, \quad u \mapsto \mathsf{u}, \quad  v \mapsto \mathsf{v}.
\end{align}

Each of the remaining affine charts $\scrU_{1j}$ of $\scrX_1$ admits a similar parametrisation, and the corresponding morphism $\uppi_1|_{\scrU_{1j}} \colon \scrU_{1j} \to \Spec \scrS_1$ is defined in the same way as above.

\begin{notation}\label{notation: gr3}
With the notation $\scrU_{1j}$ above and in \textnormal{\ref{notation: gr2}}, for each $2 \leq i \leq 2n$ and $0 \leq j \leq n$, we set
\begin{enumerate}
    \item $\scrU_{ij} \colonequals \scrU_{1j} \times_{\Spec \scrS_1} \Spec \scrS_i$, the base change of $\scrU_{1j}$ along $\Spec \scrS_i \to \Spec \scrS_1$;
    \item $\scrA_{ij} \colonequals \Gamma(\scrU_{ij}, \scrO_{\scrU_{ij}})$, the coordinate ring of $\scrU_{ij}$;
    \item $\varphi_{ij} \colon  \scrS_i\rightarrow \scrA_{ij}$, the ring homomorphism associated to $\uppi_i|_{\scrU_{ij}} \colon\scrU_{ij} \rightarrow \Spec \scrS_i$.
\end{enumerate}
\end{notation}
By definition \ref{notation: gr2}\eqref{gr2 7} $\scrX_i \colonequals \scrX_1 \times_{\Spec \scrS_1} \Spec \scrS_i$, hence $\scrX_i \cong \bigcup_{j=0}^{n} \scrU_{ij}$.

Since $\scrX_1$ is the universal resolution of $\Spec \scrS_1$, it is connected and smooth. We now show that, for $2\leq i\leq 2n$, the base change $\scrX_i$ is likewise connected and smooth.

The next result shows that the connectivity of $\scrX_i$ comes from the overlap of adjacent affine charts along the exceptional curves.  

\begin{prop}\label{lemma: gr4}
With notation in \textnormal{\ref{notation: gr2}}, $\scrX_i$ is connected for all $2 \leq i \leq 2n$.
\end{prop}
\begin{proof}
For $1 \leq j \leq n$, write $\Curve_j$ for the $j$-th exceptional curve of the universal resolution $\uppi_1 \colon \scrX_1 \rightarrow \Spec \scrS_1$ over the origin. 
By definition \ref{notation: gr2}\eqref{gr2 5}, for $2 \leq i \leq 2n$ 
\[
\scrS_i \colonequals \scrS_1/(h_1,h_2,\dots,h_{i-1}),
\]
where $h_k$ is a power series without a constant term for $1\leq k \leq i-1$.

Since each $h_k$ has no constant term, the closed immersion $\Spec \scrS_i\hookrightarrow \Spec \scrS_1$ meets the origin of $\Spec \scrS_1$.  Hence the base change $\scrX_i=\scrX_1\times_{\Spec\scrS_1}\Spec\scrS_i$ contains the exceptional fibre
$\bigcup_{j=1}^n \Curve_j$ over the origin.
Moreover, for each $1 \leq j \leq n$ the affine charts of $\scrX_i$ satisfy 
\[
\Curve_j \subset \scrU_{i,j-1}\cup \scrU_{ij}\qquad\Rightarrow\qquad \scrU_{i,j-1}\cap \scrU_{ij}\neq \emptyset,
\]
and so the affine charts of $\scrX_i$ pairwise overlap along the exceptional curves. Hence $\scrX_i$ is connected.
\end{proof}

We now prove that $\scrX_i$ is smooth for $2 \leq i \leq 2n$ by analysing each affine chart $\scrU_{ij}$ of $\scrX_i$.

\begin{prop}\label{lemma: gr2}
With notation in \textnormal{\ref{notation: gr2}}, for $2 \leq i \leq 2n$, $\scrX_i$ is smooth.
\end{prop}
\begin{proof}
Since by definition \ref{notation: gr2}\eqref{gr2 7} $\scrX_i \colonequals \scrX_1 \times_{\Spec\scrS_1} \Spec\scrS_i$ for $2 \leq i \leq 2n$,  we have the following pullback squares for the $j$-th affine chart $\scrU_{ij}$ of $\scrX_i$:
\[
\begin{minipage}{0.45\textwidth}
\centering
\begin{tikzpicture}[>=stealth,scale=1.2]
\node[black!70!white] (U00) at (2,0) {$\scrU_{ij}$};
\node[black!70!white] (U0) at (4,0) {$\scrU_{1j}$};

\node[black!70!white] (S0) at (2,-2) {$\Spec \scrS_i$};
\node[black!70!white] (S) at (4,-2) {$\Spec \scrS_1$};

\draw[->] (U00)--(U0);
\draw[->] (S0)--(S);
\draw[->] (U00)--node[right]{$\uppi_i|_{\scrU_{ij}}$}(S0);
\draw[->] (U0)--node[right]{$\uppi_1|_{\scrU_{1j}}$}(S);
\end{tikzpicture}
\end{minipage}
\hfill
\begin{minipage}{0.45\textwidth}
\centering
\begin{tikzpicture}[>=stealth,scale=1.2]
\node[black!70!white] (U00) at (2,0) {$\scrA_{ij}$};
\node[black!70!white] (U0) at (4,0) {$\scrA_{1j}$};

\node[black!70!white] (S0) at (2,-2) {$ \scrS_i$};
\node[black!70!white] (S) at (4,-2) {$ \scrS_1$};

\draw[<-] (U00)--(U0);
\draw[<-] (S0)--(S);
\draw[<-] (U00)--node[right]{$\varphi_{ij}$}(S0);
\draw[<-] (U0)--node[right]{$\varphi_{1j}$}(S);
\end{tikzpicture}
\end{minipage}
\]
Recall from \ref{notation: gr2}\eqref{gr2 5} that for $2 \leq i \leq 2n$
\[
\scrS_i \colonequals \scrS_1/(h_1,h_2,\dots,h_{i-1})
\quad\text{with}\quad
h_i \colonequals x_{i-1} + \sum_{j=2}^{\infty} j\upkappa_{ij}x_i^{\,j-1} + x_{i+1}.
\]

Therefore, for $2 \leq i \leq 2n$ and $0 \leq j \leq n$,
\begin{equation}\label{gather7}
    \scrA_{ij} \cong \scrA_{1j} \otimes_{\scrS_1} \scrS_i \cong  \scrA_{1j} \otimes_{\scrS_1} \scrS_1/(h_1, h_2, \dots, h_{i-1} ) \cong \scrA_{1j}/(\scrA_{1j}h_1, \scrA_{1j}h_2, \dots, \scrA_{1j}h_{i-1}).
\end{equation}

\noindent\textbf{First chart $(j=0)$.}
From \eqref{A10} we have
\begin{equation*}
\scrA_{10} \cong \frac{\C \lal \mathsf{b}_0,\mathsf{b}_2,\dots ,\mathsf{b}_{2n-2},\mathsf{a}_{2n}, \x_1,\x_3,\dots, \x_{2n-1},\x_{2n}, \mathsf{v} \ral[\mathsf{b}_{2n}]}{(\x_{2n}-\mathsf{a}_{2n}\mathsf{b}_{2n}, \mathsf{v}-\mathsf{b}_0\mathsf{b}_2\dots \mathsf{b}_{2n}) }.   
\end{equation*}
Moreover, by \eqref{gather} $\varphi_{10}\colon \scrS_1 \to \scrA_{10}$ is given by
\begin{align*}
   & x_0 \mapsto \mathsf{b}_0, \quad x_2 \mapsto \mathsf{b}_2, \quad \dots , \quad  x_{2n-2} \mapsto \mathsf{b}_{2n-2}, \quad  x _{2n} \mapsto \x_{2n},   \nonumber \\
    & x_1 \mapsto \x_1, \quad  x_3 \mapsto \x_3, \quad \dots, \quad x_{2n-1} \mapsto \x_{2n-1}, \quad u \mapsto \mathsf{a}_{2n}, \quad  v \mapsto \mathsf{v}. 
\end{align*}

Thus, for $1 \leq i \leq 2n-1$, the images $\scrA_{10}h_i$ are
\begin{equation*}
\scrA_{10}h_i =
\begin{cases}
\mathsf{b}_{i-1}+\sum_{j=2}^{\infty}j\upkappa_{ij}\x_i^{j-1}+\mathsf{b}_{i+1}, \text{ for } i =1,3, \dots , 2n-3\\
\x_{i-1}+\sum_{j=2}^{\infty}j\upkappa_{ij}\mathsf{b}_i^{j-1}+\x_{i+1}, \text{ for } i =2,4, \dots , 2n-2\\
\mathsf{b}_{2n-2}+\sum_{j=2}^{\infty}j\upkappa_{2n-1,j}\x_{2n-1}^{j-1}+\x_{2n}, \text{ for } i=2n-1
\end{cases}
\end{equation*}

Introduce the notation obtained by successive elimination:
\[
\mathsf{b}_{01} \colonequals -\sum_{j=2}^{\infty} j\upkappa_{1j}\,\x_1^{\,j-1}-\mathsf{b}_2 \in \C \lal \x_1,\mathsf{b}_2\ral,
\]
and, using $\x_1=-\sum_{j=2}^{\infty} j\upkappa_{2j}\,\mathsf{b}_2^{\,j-1}-\x_3$,
\[
\mathsf{b}_{02} \colonequals -\sum_{j=2}^{\infty} j\upkappa_{1j}\!\left(-\sum_{r=2}^{\infty} r\upkappa_{2r}\,\mathsf{b}_2^{\,r-1}-\x_3\right)^{\!j-1}-\mathsf{b}_2 \in \C\lal \mathsf{b}_2,\x_3\ral.
\]
Continuing inductively, for $0\leq t\leq n-1$ and $2t<k\leq 2n-1$, define $\mathsf{b}_{2t,k}$ with
\[
\mathsf{b}_{2t,k} \in
\begin{cases}
\C\lal \x_k,\mathsf{b}_{k+1}\ral, & k \text{ odd},\ k\neq 2n-1,\\
\C\lal \mathsf{b}_k,\x_{k+1}\ral, & k \text{ even},\\
\C\lal \x_{2n-1},\x_{2n}\ral, & k=2n-1.
\end{cases}
\]

Each $\scrA_{10}h_i$ has a linear term, hence eliminates one variable in \eqref{gather7}. Consequently,
\begin{align*}
\scrA_{20} & \cong \scrA_{10}/(\scrA_{10}h_1) \cong \scrA_{10}/(\mathsf{b}_{0}+\sum_{j=2}^{\infty}j\upkappa_{1j}\x_1^{j-1}+\mathsf{b}_{2}) \\
& \cong \frac{\C \lal \mathsf{b}_2,\mathsf{b}_4,\dots ,\mathsf{b}_{2n-2},\mathsf{a}_{2n}, \x_1,\x_3,\dots, \x_{2n-1},\x_{2n}, \mathsf{v} \ral[\mathsf{b}_{2n}]}{(\x_{2n}-\mathsf{a}_{2n}\mathsf{b}_{2n}, \mathsf{v}-\mathsf{b}_{01}\mathsf{b}_2\dots \mathsf{b}_{2n}) },\\
\scrA_{30} &  \cong \scrA_{10}/(\scrA_{10}h_1, \scrA_{10}h_2) \cong \scrA_{10}/(\mathsf{b}_{0}+\sum_{j=2}^{\infty}j\upkappa_{1j}\x_1^{j-1}+\mathsf{b}_{2},\x_{1}+\sum_{j=2}^{\infty}j\upkappa_{2j}\mathsf{b}_2^{j-1}+\x_{3}) \\
& \cong  \frac{\C \lal \mathsf{b}_2,\mathsf{b}_4,\dots ,\mathsf{b}_{2n-2},\mathsf{a}_{2n}, \x_3,\x_5,\dots, \x_{2n-1},\x_{2n}, \mathsf{v} \ral[\mathsf{b}_{2n}]}{(\x_{2n}-\mathsf{a}_{2n}\mathsf{b}_{2n}, \mathsf{v}-\mathsf{b}_{02}\mathsf{b}_2\dots \mathsf{b}_{2n}) },\\
& \qquad \qquad  \qquad \qquad \qquad \qquad \qquad \qquad  \vdots\\
\scrA_{2n-1,0} & \cong \scrA_{10}/(\scrA_{10}h_1, \scrA_{10}h_2, \dots, \scrA_{10}h_{2n-2}) \\
 & \cong \frac{\C \lal \mathsf{b}_{2n-2},\mathsf{a}_{2n},  \x_{2n-1},\x_{2n}, \mathsf{v} \ral[\mathsf{b}_{2n}]}{(\x_{2n}-\mathsf{a}_{2n}\mathsf{b}_{2n}, \mathsf{v}-\mathsf{b}_{0,2n-2}\mathsf{b}_{2,2n-2}\dots \mathsf{b}_{2n-4,2n-2}\mathsf{b}_{2n-2}\mathsf{b}_{2n}) },\\
 \scrA_{2n,0}  & \cong \scrA_{10}/(\scrA_{10}h_1, \scrA_{10}h_2, \dots, \scrA_{10}h_{2n-2},\scrA_{10}h_{2n-1})  \\
 & \cong \frac{\C \lal \mathsf{a}_{2n},  \x_{2n-1},\x_{2n}, \mathsf{v} \ral[\mathsf{b}_{2n}]}{(\x_{2n}-\mathsf{a}_{2n}\mathsf{b}_{2n}, \mathsf{v}-\mathsf{b}_{0,2n-1}\mathsf{b}_{2,2n-1}\dots \mathsf{b}_{2n-4,2n-1}\mathsf{b}_{2n-2,2n-1}\mathsf{b}_{2n}) }.
\end{align*}
Hence, for $2 \leq i \leq 2n$, the first affine chart $\scrU_{i0} \colonequals \Spec \scrA_{i0}$ is smooth. 
Since the last affine chart $\scrU_{in}$ is analogous to $\scrU_{i0}$, it is also smooth.

\noindent\textbf{Second chart $(j=1)$.}
From \eqref{A11} we have
\begin{equation*}
\scrA_{11} \cong \frac{\C \lal \mathsf{b}_0,\mathsf{b}_2,\dots ,\mathsf{b}_{2n-4},\mathsf{a}_{2n}, \x_1,\x_3 ,\dots, \x_{2n-1},\x_{2n-2}, \mathsf{u}, \mathsf{v} \ral[\mathsf{a}_{2n-2},\mathsf{b}_{2n-2}]}{(\x_{2n-2}-\mathsf{a}_{2n-2}\mathsf{b}_{2n-2}, \mathsf{u}-\mathsf{a}_{2n-2}\mathsf{a}_{2n},\mathsf{v}-\mathsf{b}_0\mathsf{b}_2\dots \mathsf{b}_{2n-2})}.
\end{equation*}
Moreover, by \eqref{gather2} $\varphi_{11}\colon \scrS_1 \to \scrA_{11}$ is given by
\begin{align*}
& x_0 \mapsto \mathsf{b}_0, \quad x_2 \mapsto \mathsf{b}_2, \quad \dots , \quad x_{2n-4} \mapsto \mathsf{b}_{2n-4}, \quad  x_{2n-2} \mapsto \x_{2n-2},   \quad x_{2n} \mapsto \mathsf{a}_{2n},    \nonumber\\
& x_1 \mapsto \x_1, \quad  x_3 \mapsto \x_3, \quad \dots, \quad x_{2n-1} \mapsto \x_{2n-1}, \quad u \mapsto \mathsf{u}, \quad  v \mapsto \mathsf{v}. 
\end{align*}

Thus, for $1 \leq i \leq 2n-1$, the images $\scrA_{11}h_i$ are
\begin{equation*}
\scrA_{11}h_i =
\begin{cases}
\mathsf{b}_{i-1}+\sum_{j=2}^{\infty}j\upkappa_{ij}\x_i^{j-1}+\mathsf{b}_{i+1}, \text{ for } i =1,3, \dots , 2n-5\\
\x_{i-1}+\sum_{j=2}^{\infty}j\upkappa_{ij}\mathsf{b}_i^{j-1}+\x_{i+1}, \text{ for } i =2,4, \dots , 2n-4\\
\mathsf{b}_{2n-4}+\sum_{j=2}^{\infty}j\upkappa_{2n-3,j}\x_{2n-3}^{j-1}+\x_{2n-2}, \text{ for } i =2n-3\\
\x_{2n-3}+\sum_{j=2}^{\infty}j\upkappa_{2n-2,j}\x_{2n-2}^{j-1}+\x_{2n-1}, \text{ for } i =2n-2 \\
\x_{2n-2}+\sum_{j=2}^{\infty}j\upkappa_{2n-1,j}\x_{2n-1}^{j-1}+\mathsf{a}_{2n}, \text{ for } i=2n-1
\end{cases}
\end{equation*}
Define $\mathsf{b}_{2t,k}$ analogously for $0\leq t \leq n-2$ and $2t < k\leq 2n-1$.
From the last equation, set
\begin{align*}
   \x_{2n-2,2n-1}\colonequals -\sum_{j=2}^{\infty}j\upkappa_{2n-1,j}\x_{2n-1}^{j-1}-\mathsf{a}_{2n} \in \C \lal \x_{2n-1},\mathsf{a}_{2n}\ral.
\end{align*}

Again, each $\scrA_{11}h_i$ has a linear term, hence eliminates one variable in \eqref{gather7}. Consequently,
\begin{align*}
\scrA_{21} &  \cong \scrA_{11}/(\scrA_{11}h_1) \cong \scrA_{11}/(\mathsf{b}_{0}+\sum_{j=2}^{\infty}j\upkappa_{1j}\x_1^{j-1}+\mathsf{b}_{2}) \\
& \cong   \frac{\C \lal \mathsf{b}_2, \mathsf{b}_4,\dots ,\mathsf{b}_{2n-4},\mathsf{a}_{2n}, \x_1,\x_3 ,\dots, \x_{2n-1},\x_{2n-2}, \mathsf{u}, \mathsf{v} \ral[\mathsf{a}_{2n-2},\mathsf{b}_{2n-2}]}{(\x_{2n-2}-\mathsf{a}_{2n-2}\mathsf{b}_{2n-2}, \mathsf{u}-\mathsf{a}_{2n-2}\mathsf{a}_{2n},\mathsf{v}-\mathsf{b}_{01}\mathsf{b}_2\dots \mathsf{b}_{2n-2})},\\
 \scrA_{31}  & \cong \scrA_{11}/(\scrA_{11}h_1,\scrA_{11}h_2) \cong \scrA_{11}/(\mathsf{b}_{0}+\sum_{j=2}^{\infty}j\upkappa_{1j}\x_1^{j-1}+\mathsf{b}_{2},\x_{1}+\sum_{j=2}^{\infty}j\upkappa_{2j}\mathsf{b}_2^{j-1}+\x_{3})  \\
 & \cong  \frac{\C \lal \mathsf{b}_2, \mathsf{b}_4,\dots ,\mathsf{b}_{2n-4},\mathsf{a}_{2n}, \x_3,\x_5 ,\dots, \x_{2n-1},\x_{2n-2}, \mathsf{u}, \mathsf{v} \ral[\mathsf{a}_{2n-2},\mathsf{b}_{2n-2}]}{(\x_{2n-2}-\mathsf{a}_{2n-2}\mathsf{b}_{2n-2}, \mathsf{u}-\mathsf{a}_{2n-2}\mathsf{a}_{2n},\mathsf{v}-\mathsf{b}_{02}\mathsf{b}_2\dots \mathsf{b}_{2n-2})},\\
& \qquad \qquad  \qquad \qquad \qquad \qquad \qquad \qquad  \vdots\\
 \scrA_{2n-1,1} & \cong \scrA_{11}/(\scrA_{11}h_1, \scrA_{11}h_2, \dots, \scrA_{11}h_{2n-2}) \\
& \cong \frac{\C \lal \mathsf{a}_{2n},  \x_{2n-1},\x_{2n-2}, \mathsf{u}, \mathsf{v} \ral[\mathsf{a}_{2n-2},\mathsf{b}_{2n-2}]}{(\x_{2n-2}-\mathsf{a}_{2n-2}\mathsf{b}_{2n-2}, \mathsf{u}-\mathsf{a}_{2n-2}\mathsf{a}_{2n},\mathsf{v}-\mathsf{b}_{0,2n-2}\mathsf{b}_{2,2n-2}\dots \mathsf{b}_{2n-4,2n-2}\mathsf{b}_{2n-2})},\\
\scrA_{2n,1} & \cong \scrA_{11}/(\scrA_{11}h_1, \scrA_{11}h_2, \dots,\scrA_{11}h_{2n-2}, \scrA_{11}h_{2n-1}) \\
& \cong  \frac{\C \lal \mathsf{a}_{2n},  \x_{2n-1}, \mathsf{u}, \mathsf{v} \ral[\mathsf{a}_{2n-2},\mathsf{b}_{2n-2}]}{(\x_{2n-2,2n-1}-\mathsf{a}_{2n-2}\mathsf{b}_{2n-2}, \mathsf{u}-\mathsf{a}_{2n-2}\mathsf{a}_{2n},\mathsf{v}-\mathsf{b}_{0,2n-1}\mathsf{b}_{2,2n-1}\dots \mathsf{b}_{2n-4,2n-1}\mathsf{b}_{2n-2})}.
\end{align*}

Since $\x_{2n-2,2n-1}$ has a linear term $\mathsf{a}_{2n}$, 
it follows that for $2 \leq i \leq 2n$ the second affine chart 
$\scrU_{i1} \colonequals \Spec \scrA_{i1}$ is smooth.
For $2 \leq j \leq n-1$, the affine charts $\scrU_{ij}$ are analogous to $\scrU_{i1}$ (for each fixed $i$), hence smooth as well. 
Therefore $\scrX_i$ is smooth for all $2 \leq i \leq 2n$.
\end{proof}

\begin{cor}\label{lemma: gr3}
With notation in \textnormal{\ref{notation: gr2}}, $\End_{\scrS_i}(N_i) \cong \C \lbl Q \rbl /  R_i $ for $1 \leq i \leq 2n$.
\end{cor}
\begin{proof}
Recall from \ref{notation: gr2} the commutative diagram
\[
\begin{tikzpicture}[bend angle=8, looseness=1.2]

\node (a2) at (-2,0) {$\Spec \scrS_1$};
\node (a3) at (-4,0) {$\Spec \scrS_2$};
\node (a4) at (-8,0) {$\Spec \scrS_{2n}$};

\node (b2) at (-2,1.5) {$\scrX_1$};
\node (b3) at (-4,1.5) {$\scrX_2$};
\node (b4) at (-8,1.5) {$\scrX_{2n}$};

\draw[->] (-2,1.2) -- node[right] {$\scriptstyle \uppi_1$} (-2,0.3);
\draw[->] (-4,1.2) -- node[right] {$\scriptstyle \uppi_2$} (-4,0.3);
\draw[->] (-8,1.2) -- node[right] {$\scriptstyle \uppi_{2n}$} (-8,0.3);

\draw[->] (-3.7,1.5) -- node[above] {} (-2.3,1.5);
\draw[->] (-5.3,1.5) -- node[above] {} (-4.3,1.5);
\node (c2) at (-6,1.5) {$\dots$};
\draw[->] (-7.6,1.5) -- node[above] {} (-6.7,1.5);

\draw[->] (-3.3,0) -- node[above] {} (-2.7,0);
\draw[->] (-5.3,0) -- node[above] {} (-4.7,0);
\node (c2) at (-6,0) {$\dots$};
\draw[->] (-7.3,0) -- node[above] {} (-6.7,0);
\end{tikzpicture}
\]
together with the $\scrS_i$-module $N_i$ for $1 \leq i \leq 2n$.

By \cite[\S 5]{IW2}, $N_1$ is the tilting bundle for $\uppi_1$, and by \eqref{502} we have
\begin{equation}\label{502}
    \End_{\scrS_1}(N_1) \cong \C \lbl Q \rbl /R_1,
\end{equation}
where $Q$ and $R_1$ are given in \eqref{22}. 

Note that $\scrS_1$ is an integral domain and normal, and $\scrX_1$ is connected and smooth. 
By \ref{lemma: gr1}, $\scrS_2$ is also an integral domain and normal. 
By \ref{lemma: gr4} and \ref{lemma: gr2}, $\scrX_2$ is connected and smooth. 
Since $N_1$ is the tilting bundle for $\uppi_1$, we can apply \cite[2.11]{V3} to deduce that $N_2 \cong N_1 \otimes_{\scrS_1} \scrS_2$ is the tilting bundle for $\uppi_2$ and
\begin{align*}
    \End_{\scrS_2}(N_2) & \cong  \End_{ \scrS_{1}/h_{1}}(N_{1}\otimes_{\scrS_{1}}\scrS_{1}/h_{1}) \tag{since $\scrS_2 \cong \scrS_{1}/h_{1}, \ N_{2} \cong N_{1}\otimes_{\scrS_{1}}\scrS_{2} $} \\
   &  \cong \End_{\scrS_{1}}(N_{1})/(h_{1}) \tag{by \cite[2.11]{V3}}\\
   & \cong \C \lbl Q \rbl / R_2. \tag{by \eqref{502}}
\end{align*}
Here $R_2$ is obtained from $R_1$ by adding the relation \eqref{21} with $i=1$, namely
\begin{equation*} 
 l_{t,0}+\sum_{j=2}^{\infty}j\upkappa_{1j}l_{t,1}^{j-1}+l_{t,2}=0, \quad  \text{ for }t \in \Z/(n+1),
\end{equation*}
which corresponds to
\[
h_1 = x_0+\sum_{j=2}^{\infty} j\upkappa_{1j}\,x_1^{\,j-1}+x_2.
\]
Iterating this argument, for any $2 \leq i \leq 2n$, we have $N_i \cong N_{i-1} \otimes_{\scrS_{i-1}} \scrS_i$ is the tilting bundle for $\uppi_i$, and 
\begin{align*}
   \End_{\scrS_i}(N_i) & \cong \End_{\scrS_{i-1}}(N_{i-1})/(h_{i-1}) \\
   & \cong  \End_{\scrS_{i-2}}(N_{i-2})/(h_{i-1},h_{i-2}) \\
   & \qquad \vdots \\
   & \cong \End_{\scrS_{1}}(N_{1})/(h_{i-1},h_{i-2}, \dots , h_1) \\
   &\cong \C \lbl Q \rbl / R_i . \qedhere
\end{align*}
\end{proof}

The following theorem shows that any monomialized Type~$A$ potential on $Q_n$, not necessarily reduced, can be realised by a crepant resolution of a $cA_n$ singularity.
\begin{theorem}\label{gr}
With the monomialized Type~$A$ potential $f$ in \eqref{f1} on $Q_n$, the $cA_n$ singularity $\scrR$, and the
$\CM$ $\scrR$-module $M$ as in \textnormal{\ref{notation: gr2}}, one has
$\underline{\End}_{\scrR}(M) \cong \Jac(Q_n,f)$.
\end{theorem}
\begin{proof}
By \ref{lemma: gr1}, $\scrR \cong \scrS_{2n}$ and $\End_{\scrR}(M) \cong \End_{\scrS_{2n}}(N_{2n})$. By \ref{lemma: gr3}, $\End_{\scrS_{2n}}(N_{2n}) \cong \C \lbl Q \rbl /  R_{2n} $. 
Thus $\End_{\scrR}(M) \cong \C \lbl Q \rbl / R_{2n}$, where $Q$ and $R_{2n}$ are as in \textnormal{\ref{notation: gr2}\eqref{gr2 9}}.

Similar to $Q_n$, we also define $\x_i$ and $\x_i'$ on $Q$ as follows: 
for any $0 \leq i \leq n$, set $\x_{2i} \colonequals a_{2i}b_{2i}$ and $\x_{2i}' \colonequals b_{2i}a_{2i}$, and for any $1 \leq i \leq n$, set $\x_{2i-1} \colonequals l_{i,2i-1} =: \x_{2i-1}'$. 

%Set the potential $\mathsf{f} \colonequals f+ \x_0'\x_1+ \x_{2n-1}'\x_{2n}$ on $Q$.
%Then we prove that the relations $R_1$ \eqref{22} and $R_2$ \eqref{21} can induce the relations generated by potential $\mathsf{f}$.
Next, we consider the following relations induced by $R_{2n}$.
For any $1 \leq t \leq n-1$, left multiplying the $i=2t$ case of $\eqref{21}$ by $b_{2t}$ gives
\begin{align*}
 & \quad b_{2t} l_{t,2t-1}+\sum_{j=2}^{\infty}j\upkappa_{2t,j}b_{2t}l_{t,2t}^{j-1}+b_{2t}l_{t,2t+1} \\
&   = b_{2t} l_{t,2t-1}+\sum_{j=2}^{\infty}j\upkappa_{2t,j}b_{2t}l_{t,2t}^{j-1}+l_{t+1,2t+1}b_{2t} \tag{since $b_{2t}l_{t,2t+1}=l_{t+1,2t+1}b_{2t}$ by \eqref{22}} \\
&  = b_{2t}\x_{2t-1}'+\sum_{j=2}^{\infty}j\upkappa_{2t,j}b_{2t}\x_{2t}^{j-1} +\x_{2t+1}b_{2t}. \tag{since $l_{t,2t}=a_{2t}b_{2t} =\x_{2t}$, $l_{t,2t-1}=\x_{2t-1}'$ and $l_{t+1,2t+1}=\x_{2t+1}$ by \eqref{22}} 
\end{align*}
Similarly, for any $1 \leq t \leq n-1$, right multiplying the $i=2t$ case of $\eqref{21}$ by $a_{2t}$ gives
\begin{align*}
&\quad l_{t,2t-1}a_{2t}+\sum_{j=2}^{\infty}j\upkappa_{2t,j}l_{t,2t}^{j-1}a_{2t}+l_{t,2t+1}a_{2t} \\
& = l_{t,2t-1}a_{2t}+\sum_{j=2}^{\infty}j\upkappa_{2t,j}l_{t,2t}^{j-1}a_{2t}+a_{2t}l_{t+1,2t+1} \tag{since $l_{t,2t+1}a_{2t}=a_{2t}l_{t+1,2t+1}$ by \eqref{22}} \\
& = \x_{2t-1}'a_{2t}+\sum_{j=2}^{\infty}j\upkappa_{2t,j}\x_{2t}^{j-1}a_{2t} +a_{2t}\x_{2t+1}.\tag{since $l_{t,2t}=a_{2t}b_{2t} =\x_{2t}$, $l_{t,2t-1}=\x_{2t-1}'$ and $l_{t+1,2t+1}=\x_{2t+1}$ by \eqref{22}}
\end{align*}
For any $1 \leq t \leq n$, the $i=2t-1$ case of \eqref{21} is
\begin{align*}
 l_{t,2t-2}+\sum_{j=2}^{\infty}j\upkappa_{ij}l_{t,2t-1}^{j-1}+l_{t,2t}=  \x_{2t-2}'+\sum_{j=2}^{\infty}j\upkappa_{2t-1,j}\x_{2t-1}^{j-1}+\x_{2t}.
 \tag{since $l_{t,2t-1}=\x_{2t-1}$, $l_{t,2t-2}=b_{2t-2}a_{2t-2}=\x'_{2t-2}$ and $l_{t,2t}=a_{2t}b_{2t}=\x_{2t}$ by notation and \eqref{22}}
\end{align*}
Combining the above three types of relations gives the following,

\begin{equation}\label{26}
T \colonequals
\begin{cases}
    b_{i}\x_{i-1}'+\sum_{j=2}^{\infty}j\upkappa_{ij}b_i\x_i^{j-1}+\x_{i+1}b_{i}=0 \text{, for } i =2,4,\dots , 2n-2.\\
    \x_{i- 1}'a_{i}+\sum_{j=2}^{\infty}j\upkappa_{ij}\x_{i}^{j-1}a_i+a_{i}\x_{i+1}=0\text{, for } i =2,4,\dots , 2n-2. \\
     \x_{i-1}'+\sum_{j=2}^{\infty}j\upkappa_{ij}\x_{i}^{j-1}+\x_{i+1}  = 0 \text{, for }i = 1,3,\dots, 2n-1.
\end{cases}    
\end{equation}

Then we define the quiver $\mathcal{Q}_n$ by deleting loops on $Q$ as follows.
For each vertex $t$ on $Q$ with $1 \leq t \leq n$, we delete all loops $l_{tj}$ except $l_{t,2t-1}$ (namely $\x_{2t-1}$). Note that $Q_n$ is $\mathcal{Q}_n$ by removing the vertex $0$ and loops on it.

In \ref{057} below we will show that $\C \lbl Q\rbl/ \langle R_{2n},e_0 \rangle \cong \C \lbl \mathcal{Q}_n \rbl/\langle T,e_0 \rangle$. 
Together with the isomorphism $\End_{\scrR}(M) \cong  \C \lbl Q \rbl /  R_{2n}$ at the start of the proof, this gives
\begin{equation*}
\underline{\End}_{\scrR}(M) \cong \C \lbl Q \rbl / \langle R_{2n},e_0 \rangle \cong \C \lbl \mathcal{Q}_n \rbl/\langle T,e_0 \rangle.
\end{equation*}

Thus $\underline{\End}_{\scrR}(M)$ is isomorphic to $\C \lbl Q_n \rbl$ factored by the relations $T$, which after deleting paths that factor through vertex $0$, become

\begin{align*} 
 &  b_{i}\x_{i-1}'+\sum_{j=2}^{\infty}j\upkappa_{ij}b_i\x_i^{j-1}+\x_{i+1}b_{i}=0 \text{, for } i =2,4,\dots , 2n-2.\\
 &  \x_{i- 1}'a_{i}+\sum_{j=2}^{\infty}j\upkappa_{ij}\x_{i}^{j-1}a_i+a_{i}\x_{i+1}=0 \text{, for } i =2,4,\dots , 2n-2. \\
& \x_{i-1}'+\sum_{j=2}^{\infty}j\upkappa_{ij}\x_{i}^{j-1}+\x_{i+1}  = 0 \text{, for }i = 3,\dots, 2n-2.\\
& \sum_{j=2}^{\infty}j\upkappa_{1j}\x_{1}^{j-1}+\x_{2}  = 0, \ \x_{2n-2}'+\sum_{j=2}^{\infty}j\upkappa_{2n-1,j}\x_{2n-1}^{j-1} = 0.
\end{align*}
These are exactly the relations generated by the derivatives of $f$.
Thus $\underline\End_{\scrR}(M)\cong \Jac(Q_{n},f)$.
\end{proof}

\begin{lemma}\label{057}
With notation in \textnormal{\ref{notation: gr2}} and \textnormal{\ref{gr}}, $\C \lbl Q\rbl/ \langle R_{2n},e_0 \rangle \cong \C \lbl \mathcal{Q}_n \rbl/\langle T,e_0 \rangle$.
\end{lemma}
\begin{proof}
We first divide the relations $R_{2n}$ in \textnormal{\ref{notation: gr2}\eqref{gr2 9}} into three parts. The following are the relations in $R_{2n}$ that factor through the vertex $0$.
\begin{equation*}
    T_0 \colonequals
     \begin{cases}
     l_{00} = a_{0}b_{0},\text{  }l_{0,2n} = b_{2n}a_{2n}.\\
     l_{0i}a_{0} = a_{0}l_{1i},\ l_{ni}a_{2n} = a_{2n}l_{0i},\ l_{0i}b_{2n} = b_{2n}l_{ni},\ l_{1i}b_{0} = b_{0}l_{0i}, \\
     l_{0i}l_{0j} = l_{0j}l_{0i} \text{, for } 0 \leq i, j \leq 2n.\\
     l_{0,i-1}+\sum_{j=2}^{\infty}j\upkappa_{ij}l_{0,i}^{j-1}+l_{0,i+1}=0
  \text{, for } 1 \leq i \leq 2n-1.
 \end{cases}
\end{equation*}
Then we divide the remaining relations of $R_{2n}$ into the following two parts.
\begin{equation*}
     T_1 \colonequals
\begin{cases}
l_{t,2t-2}=b_{2t-2}a_{2t-2}, \text{ for } 1 \leq t \leq n. \\
l_{ti}l_{tj} = l_{tj}l_{ti},  \text{ for } 1 \leq t \leq n\text { and } 0 \leq i, j \leq 2n. \\
l_{ti}a_{2t} = a_{2t}l_{t+1,i},\ l_{t+1,i}b_{2t} = b_{2t}l_{ti}, \text { for } 1 \leq t \leq n-1 \text { and } 0 \leq i\leq 2n. 
\end{cases}
\end{equation*}
\begin{equation*}
T_2  \colonequals
\begin{cases}
l_{t,2t}=a_{2t}b_{2t}, \text{ for any } 1 \leq t \leq n.\\
l_{t,i-1}+\sum_{j=2}^{\infty}j\upkappa_{ij}l_{t,i}^{j-1}+l_{t,i+1}=0
  \text{ for any } 1 \leq t \leq n \text{ and } 1 \leq i \leq 2n-1.
\end{cases}
\end{equation*}

Since $T$ in \eqref{26} is induced by $R_{2n}$, necessarily 
\begin{equation}\label{44}
    \C \lbl Q\rbl/ \langle R_{2n} \rangle \cong \C \lbl Q\rbl/ \langle R_{2n},T \rangle \cong \C \lbl Q\rbl/ \langle T_0,T_1,T_2,T \rangle.
\end{equation}
We next use $T_2$ to eliminate some loops at vertex $1,2,\dots, n$ of $Q$, as follows.

Fix some vertex $t$ with $1\leq t \leq n$ and consider the loops $l_{ti}$ on it. 
Since $l_{t,2t}=a_{2t}b_{2t}$ in $T_2$, we can eliminate $l_{t,2t}$. 
From our notation, $l_{t,2t-1} \colonequals \x_{2t-1}$ and $\x_{2t}\colonequals a_{2t}b_{2t}$. Thus we can write $l_{t,2t}= \x_{2t}$.
Since $l_{ti}l_{tj}=l_{tj}l_{ti}$ in $T_1$  for $0 \leq i,  j \leq 2n$, we can consider $\C \lbl l_{t,2t-1},l_{t,2t} \rbl$ as the polynomial ring $\C \lal l_{t,2t-1},l_{t,2t} \ral$. By the relation
\begin{equation}\label{25}
     l_{t,2t-1}+\sum_{j=2}^{\infty}j\upkappa_{ij}l_{t,2t}^{j-1}+l_{t,2t+1}=0,
\end{equation}
in $T_2$, we can express $l_{t,2t+1} \in \C \lal l_{t,2t-1},l_{t,2t} \ral = \C\lal\x_{2t-1},\x_{2t}\ral$. 
Thus we can eliminate $l_{t,2t+1}$. Similar to the argument in \ref{054}, for each $i \neq 2t-1$ we can express $l_{ti} \colonequals \bar{l}_{ti}(x_{2t-1},\x_{2t}) \in \C\lal\x_{2t-1},\x_{2t}\ral$ and eliminate it.
So we only leave one loop $l_{t,2t-1}=\x_{2t-1}$ on vertex $t$. 

Thus we can use all the relations in $T_2$ to eliminate all such loops at vertices $1,2,\dots, n$.
For $0 \leq k \leq 2$, write $\overline{T}_k$ for the the relations where we have substituted $l_{ti}$ in $T_k$ by the polynomial $\bar{l}_{ti}$ for $1 \leq t \leq n$ and $0 \leq i \leq 2n$. So we have
\begin{equation}\label{43}
    \C \lbl Q\rbl/ \langle T_0,T_1,T_2,T \rangle \cong  \C \lbl \mathcal{Q}_n \rbl/ \langle \overline{T}_0,\overline{T}_1,T  \rangle.
\end{equation}
Now during the above substitution process, the following expressions in $\overline{T}_2$
\begin{align*}
   &  \bar{l}_{t,2t}= \x_{2t}=a_{2t}b_{2t} \tag{since $\x_{2t}= a_{2t}b_{2t}$}\\
   & \bar{l}_{t,2t-1}+\sum_{j=2}^{\infty}j\upkappa_{ij}\bar{l}_{t,2t}^{j-1}+\bar{l}_{t,2t+1}=0 \tag{by \eqref{25}}
\end{align*}
hold in $\C \lbl \mathcal{Q}_n \rbl$ tautologically. Similarly, tautologically, all the other expressions in $\overline{T}_2$ also hold in $\C \lbl \mathcal{Q}_n \rbl$.

We next prove that $T$ in \eqref{26} induces $\overline{T}_1$. 

(1) Firstly, we prove that $T$ induces $\bar{l}_{t,2t-2}=b_{2t-2}a_{2t-2}$ for  $1 \leq t \leq n$. Since
\begin{align*}
 & \x_{2t-2}'+\sum_{j=2}^{\infty}j\upkappa_{2t-1,j}\x_{2t-1}^{j-1}+\x_{2t}  = 0, 
 \tag{by the $i=2t-1$ case of the third line in \eqref{26}}\\
 &\bar{l}_{t,2t-2}+\sum_{j=2}^{\infty}j\upkappa_{2t-1,j}\bar{l}_{t,2t-1}^{j-1}+\bar{l}_{t,2t}=0.\tag{since $\overline{T}_2$ holds in $\C \lbl \mathcal{Q}_n\rbl$}
\end{align*}
and by notation $\bar{l}_{t,2t-1}=\x_{2t-1}$ and $\bar{l}_{t,2t}=\x_{2t}$, then $\bar{l}_{t,2t-2}=\x_{2t-2}'=b_{2t-2}a_{2t-2}$.

(2) Secondly, we prove that $T$ induces $\bar{l}_{ti}\bar{l}_{tj} = \bar{l}_{tj}\bar{l}_{ti}$ for $1 \leq t \leq n$ and $ 0 \leq i, j \leq 2n$. 

Left multiplying the $i=2t$ case of the first line in \eqref{26} by $a_{2t}$ gives
\begin{align*}
   0= a_{2t}(b_{2t}\x_{2t-1}'+\sum_{j=2}^{\infty}j\upkappa_{2t,j}b_{2t}\x_{2t}^{j-1}+\x_{2t+1}b_{2t})= \x_{2t}\x_{2t-1}'+\sum_{j=2}^{\infty}j\upkappa_{2t,j}\x_{2t}^{j}+a_{2t}\x_{2t+1}b_{2t}. \tag{since $\x_{2t}= a_{2t}b_{2t}$}
\end{align*}
Right multiplying the $i=2t$ case of the second line in \eqref{26} by $b_{2t}$ gives
\begin{align*}
0=(\x_{2t- 1}'a_{2t}+\sum_{j=2}^{\infty}j\upkappa_{2t,j}\x_{2t}^{j-1}a_{2t}+a_{2t}\x_{2t+1})b_{2t}=\x_{2t-1}'\x_{2t} +\sum_{j=2}^{\infty}j\upkappa_{2t,j}\x_{2t}^{j}+a_{2t}\x_{2t+1}b_{2t}. \tag{since $\x_{2t}= a_{2t}b_{2t}$}
\end{align*}
Thus $\x_{2t}\x_{2t-1}'= \x_{2t-1}'\x_{2t}$. Since $\x_{2t-1}$ is the loop at vertex $t$, then by definition $\x_{2t-1}'=\x_{2t-1}$, and so $\x_{2t}\x_{2t-1}= \x_{2t-1}\x_{2t}$.
Together with the fact that each $\bar{l}_{ti} \in \C\lal\x_{2t-1},\x_{2t}\ral$ gives $\bar{l}_{ti}\bar{l}_{tj} = \bar{l}_{tj}\bar{l}_{ti}$ for $0\leq i, j \leq 2n$.

(3) Finally, we prove that $T$ induces $\bar{l}_{ti}a_{2t} = a_{2t}\bar{l}_{t+1,i},\ \bar{l}_{t+1,i}b_{2t} = b_{2t}\bar{l}_{ti}$ for $1 \leq t \leq n-1$ and $0 \leq i\leq 2n$. 
For each vertex $t$ with $1 \leq t \leq n-1$, we have
\begin{align*}
  &  \bar{l}_{t,2t}=a_{2t}b_{2t}=\x_{2t}, \tag{since $\overline{T}_2$ holds in $\C \lbl \mathcal{Q}_n\rbl$}\\ 
  & \bar{l}_{t+1,2t}=b_{2t}a_{2t}=\x_{2t}', \tag{by (1)}\\  
  & \bar{l}_{t,2t-1}=\x_{2t-1}=\x_{2t-1}',\ \bar{l}_{t+1,2t+1}=\x_{2t+1}=\x_{2t+1}'. \tag{by the definition of $\x_{2t-1}$ and $\x_{2t+1}$}
\end{align*}
Thus
\begin{align*}
 \bar{l}_{t,2t}a_{2t}
 &= a_{2t}b_{2t}a_{2t}  \tag{since $\bar{l}_{t,2t}=a_{2t}b_{2t}$}\\
 &=a_{2t}\bar{l}_{t+1,2t}, \tag{since $\bar{l}_{t+1,2t}=b_{2t}a_{2t}$}
\end{align*}
and
\begin{align*}
\bar{l}_{t,2t-1}a_{2t} 
 & =\x_{2t-1}'a_{2t}  \tag{since $\bar{l}_{t,2t-1}=\x_{2t-1}'$}\\
& =  -\sum_{j=2}^{\infty}j\upkappa_{2t,j}\x_{2t}^{j-1}a_{2t}-a_{2t}\x_{2t+1} \tag{by the $i=2t$ case of the second line in \eqref{26}}\\
& = -\sum_{j=2}^{\infty}j\upkappa_{2t,j}a_{2t}\bar{l}_{t+1,2t}^{j-1}-a_{2t}\bar{l}_{t+1,2t+1} \tag{since $\x_{2t}=a_{2t}b_{2t}$, $\bar{l}_{t+1,2t}=b_{2t}a_{2t}$ and $\x_{2t+1}=\bar{l}_{t+1,2t+1}$}\\
&= -a_{2t}(\sum_{j=2}^{\infty}j\upkappa_{2t,j}\bar{l}_{t+1,2t}^{j-1}+\bar{l}_{t+1,2t+1})  \notag \\
& = a_{2t}\bar{l}_{t+1,2t-1}.  \tag{since $\overline{T}_2$ holds in $\C \lbl \mathcal{Q}_n\rbl$}
\end{align*}

Since $\overline{T}_2$ holds in $\C \lbl \mathcal{Q}_n\rbl$, then similar to the argument in \ref{054}, each $\bar{l}_{ti} \in \C\lbl \bar{l}_{t,2t-1},\bar{l}_{t,2t}\rbl$ and $\bar{l}_{t+1,i} \in \C\lbl \bar{l}_{t+1,2t-1},\bar{l}_{t+1,2t}\rbl$. Furthermore,
\begin{equation*}
    \bar{l}_{ti}  = H_i(\bar{l}_{t,2t-1},\bar{l}_{t,2t}), \ \bar{l}_{t+1,i}=H_i(\bar{l}_{t+1,2t-1},\bar{l}_{t+1,2t}).
\end{equation*}
for the same $H_i$. 
Together with the above $\bar{l}_{t,2t}a_{2t}=a_{2t}\bar{l}_{t+1,2t}$ and $\bar{l}_{t,2t-1}a_{2t}=a_{2t}\bar{l}_{t+1,2t-1}$, this gives $\bar{l}_{ti}a_{2t}=a_{2t}\bar{l}_{t+1,i}$ for each $i$.

Similarly, $T$ \eqref{26} also induces $\bar{l}_{t+1,i}b_{2t}=b_{2t}\bar{l}_{ti}$ for each $i$.

Combining (1), (2) and (3), it follows that $T$ induces $\overline{T}_1$, and so $\C \lbl \mathcal{Q}_n \rbl/ \langle \overline{T}_0,\overline{T}_1,T  \rangle \cong \C \lbl \mathcal{Q}_n \rbl/ \langle \overline{T}_0,T  \rangle$. 
Together with \eqref{44}, this gives
\begin{equation*}
    \C \lbl Q\rbl/ \langle R_{2n}\rangle \cong \C \lbl Q\rbl/ \langle T_0,T_1,T_2,T \rangle 
    \stackrel{\scriptstyle \eqref{43}}{\cong} 
    \C \lbl \mathcal{Q}_n \rbl/ \langle \overline{T}_0,\overline{T}_1,T  \rangle \cong \C \lbl \mathcal{Q}_n \rbl/ \langle \overline{T}_0,T  \rangle,
\end{equation*}
and so $\C \lbl Q\rbl/ \langle R_{2n},e_0 \rangle \cong \C \lbl \mathcal{Q}_n \rbl/ \langle \overline{T}_0,T,e_0  \rangle \cong  \C \lbl \mathcal{Q}_n \rbl/\langle T,e_0 \rangle$.
\end{proof}

We now consider the quiver $Q_{n, I}$ for some $I \subseteq \{1,2, \dots, n\}$ and prove that any Type $A$ potential on it can be realized by a crepant resolution of a $cA_n$ singularity as follows.

\begin{definition}\label{def: Type An}
We say that $\uppi$ is \emph{Type $A_{n}$} if $\uppi$ is a crepant resolution $\scrX \rightarrow \Spec\scrR$ where $\scrR$ is $cA_n$. Moreover, we say that $\uppi$ is \emph{Type $A_{n, I}$} if the normal bundle of the exceptional curve $\Curve_i$ is $\scrO(-1) \bigoplus \scrO(-1)$ if and only if $i \in I$, else the normal bundle is $\scrO(-2) \bigoplus \scrO$. 
\end{definition}

\begin{theorem}\label{511}
For any Type $A$ potential $f$ on $Q_{n, I}$, there exists a Type $A_n$ crepant resolution $\uppi \colon \scrX \rightarrow \Spec\scrR$ such that $\Lambda_{\mathrm{con}}(\uppi) \cong \Jac(Q_{n,I}, f)$. If furthermore $f$ is reduced, then $\uppi$ is Type $A_{n, I}$.
\end{theorem}
\begin{proof}
By the Splitting Theorem (\cite[4.6]{DWZ1}) and \ref{TypeA}, there is a reduced Type A potential $f_{\red}$ on $Q_{n,I'}$ for some $I \subseteq I' \subseteq \{1,2,\dots,n\}$ such that $\Jac(Q_{n,I'},f_{\red}) \cong \Jac(Q_{n,I},f)$. Then, by \ref{047}, there exists a reduced monomialized Type $A$ potential $g$ on $Q_{n, I'}$ such that $f_{\red} \cong g$. By \ref{410}, there exists a monomialized Type $A$ potential $h$ on $Q_n$ such that $\Jac(Q_{n,I'},g) \cong \Jac(Q_n,h)$. Thus we have
\begin{equation*}
     \Jac(Q_{n,I},f) \cong \Jac(Q_{n,I'},f_{\red}) \cong \Jac(Q_{n,I'},g)  \cong \Jac(Q_n,h).
\end{equation*}

By \ref{gr}, there exists a $cA_n$ singularity $\scrR$ and a maximal $\CM$ $\scrR$-module $M$ such that $\underline\End_{\scrR}(M) \cong \Jac(Q_{n}, h)$. Denote $\uppi$ to be the crepant resolution of $\Spec \scrR$, which corresponds to $M$ in \ref{36}. Thus $\Lambda_{\mathrm{con}}(\uppi) \cong \Jac(Q_n,h)$, and so $ \Lambda_{\mathrm{con}}(\uppi) \cong \Jac(Q_{n,I},f)$.

If furthermore $f$ is reduced, then $I'=I$, $f_{\red}= f$ and $g$ is a reduced monomialized Type $A$ potential on $Q_{n, I}$. Then write 
\begin{equation*}
    h= \sum_{i=1}^{2n-2}\x_i^{\prime}\x_{i+1} + \sum_{i=1}^{2n-1}\sum_{j=2}^{\infty} \upkappa_{ij} \x_i^{j},
\end{equation*}
for some $\upkappa_{ij} \in \C$. Since $g$ is reduced on $Q_{n,I}$, then by \ref{410}
$\upkappa_{2i-1,2} \neq 0$ when $i \in I$, and $\upkappa_{2i-1,2} = 0$ when $i \notin I$. Write $\scrR$ and $M$ as follows,
\begin{equation*}
    \scrR = \frac{\mathbb{C} \lal u, v, x, y \ral}{uv-g_{0}g_{2} \dots g_{{2n}}}
\end{equation*}
and $M = \scrR \oplus (u,g_{0}) \oplus (u,g_{0}g_{2}) \oplus \dots  \oplus (u,\prod_{i=0}^{n-1} g_{2i})$. We next prove that $\uppi$ is Type $A_{n, I}$.  
\begin{enumerate}
\item For any vertex $i \in I$, since $\upkappa_{2i-1,2} \neq 0$, then $(g_{2i-2},g_{2i}) = (x,y)$ by \ref{054}, and so the normal bundle of the exceptional curve $\Curve_i$ of $\uppi$ is $\scrO(-1) \bigoplus \scrO(-1)$ by \ref{35}.
\item For any vertex $i \notin I$, since $\upkappa_{2i-1,2} = 0$, then $( g_{2i-2},g_{2i}) \subsetneq ( x,y)$ by \ref{054}, and so the normal bundle of the exceptional curve $\Curve_i$ of $\uppi$ is $\scrO(-2) \bigoplus \scrO$ by \ref{35}. \qedhere
\end{enumerate}
\end{proof}

The Brown--Wemyss Realisation Conjecture \cite{BW2} states that if $f$ is any potential which satisfies $\textnormal{Jdim}(f) \leq 1$ (see \cite[3.4]{BW2} for the definition), then $\Jac(f)$ is isomorphic to the contraction algebra of some crepant resolution $\scrX \rightarrow \Spec\scrR$ with $\scrR$ $\textnormal{cDV}$.
Therefore, \ref{511} verifies the Brown--Wemyss Realisation Conjecture for all Type~$A$ potentials
on $Q_{n,I}$, for arbitrary $n \geq 1$ and $I \subseteq \{1,2,\dots,n\}$.

\subsection{Type $A_{n,I}$ crepant resolutions and potentials}\label{Cor}
In this subsection we prove the converse to \ref{511}; namely, given any Type~$A_{n,I}$ crepant resolution $\uppi$, there exists a reduced Type~$A$ potential $f$ on $Q_{n,I}$ such that $\Lambda_{\mathrm{con}}(\uppi)\cong \Jac(f)$ (see~\ref{514}).  
Together with \ref{511}, this yields a correspondence between Type~$A_n$ crepant resolutions and monomialized Type~$A$ potentials on $Q_n$,
made precise in \ref{515} and~\ref{516}.

The following \ref{591} and~\ref{581} show that commutativity of the algebras $e_i \Lambda_{\mathrm{con}}(\uppi)e_i$ and $e_i \Jac(f) e_i$ provides the key link between Type~$A_n$ crepant resolutions and Type~$A$ potentials on $Q_n$; this will be crucial in the proof of \ref{514}.

\begin{lemma}\label{591}
If $\uppi \colon \scrX \rightarrow \Spec \scrR$ is a Type $A_n$ crepant resolution, then $e_i \Lambda_{\mathrm{con}}(\uppi)e_i$ is commutative for any $1 \leq i \leq n$.
\end{lemma}
\begin{proof}
Since $\uppi$ is a Type $A_n$ crepant resolution, we have $\Lambda(\uppi) \cong\End_{\scrR}(M)$ and $\Lambda_{\mathrm{con}}(\uppi) \cong \underline{\End}_{\scrR}(M)$ for some maximal Cohen--Macaulay ($\CM$) $\scrR$-module $M$ where $M  = \scrR \oplus M_1 \oplus \dots \oplus M_n$ and each $M_i$ is an indecomposable rank one CM $\scrR$-module. Thus $\End_{\scrR}(M_i) \cong \scrR$ for $1 \leq i \leq n$ from e.g. \cite[5.4]{IW1}. 

Let $\mathscr{C}$ denote the stable category $\underline{\mathrm{CM}}\,\scrR$ of $\CM \, \scrR$-modules.
Then $\underline{\End}_{\scrR}(M) \cong \End_{\mathscr{C}}(M)$ and $\underline{\End}_{\scrR}(M_i)\cong \End_{\mathscr{C}}(M_i)$.
Thus for any $1 \leq i \leq n$,
\begin{equation*}
e_i\Lambda_{\mathrm{con}}(\uppi)e_i \cong e_i\underline{\End}_{\scrR}(M)e_i \cong e_i\End_{\mathscr{C}} (M)e_i \cong \End_{\mathscr{C}} (M_i) \cong \underline{\operatorname{End}}_{\scrR}(M_i).
\end{equation*}
Since $\End_{\scrR}(M_i) \cong \scrR$ is commutative and $\underline{\End}_{\scrR}(M_i)$ is a quotient of $\End_{\scrR}(M_i)$, then $\underline{\End}_{\scrR}(M_i)$ is also commutative, and so $e_i \Lambda_{\mathrm{con}}(\uppi)e_i$ is commutative.
\end{proof}

\begin{lemma}\label{581}
Suppose that $f$ is a reduced potential on $Q_{n, I}$. If there is some integer $j$ where $1\leq j \leq m-1$ such that $f$ does not contain $\x_j'\x_{j+1}$, then there exists some integer $i$ (depending on $j$) where $1 \leq i \leq n$ such that $e_i \Jac(f) e_i$ is not commutative.
\end{lemma}
\begin{proof}
(1) When $\x_j$ and $\x_{j+1}$ are not loops, then there exists a vertex $i\in I$ such that $Q_{n, I}$ at vertex $i$ locally looks like the following. 
\[
\begin{array}{c}
\begin{tikzpicture}[scale=1.5]
\node (a) at (-1,0) [vertex] {};
\node (b) at (0,0) [vertex] {};
\node (c) at (1,0) [vertex] {};

\node (a1) at (-1,-0.25) {$\scriptstyle i-1$};
\node (a2) at (0,-0.25) {$\scriptstyle i$};
\node (a3) at (1,-0.25) {$\scriptstyle i+1$};

\draw[->,bend left,looseness=0.7] (a) to node[above] {$\scriptstyle a_{j}$} (b);
\draw[<-,bend right,looseness=0.7] (a) to node[below] {$\scriptstyle b_{j}$} (b);
\draw[->,bend left,looseness=0.7] (b) to node[above] {$\scriptstyle a_{j+1}$} (c);
\draw[<-,bend right,looseness=0.7] (b) to node[below] {$\scriptstyle b_{j+1}$} (c);

\node (z) at (-1.5,0) {$Q'\colonequals$};
\end{tikzpicture}
\end{array}
\]

We denote the above quiver with only the three vertices shown, as $Q'$. Then consider the noncommutative algebra $J$, defined as $\Jac( f)$ quotiented by the ideal generated by the following paths:
\begin{itemize}
\item $\mathfrak{m}^5_{Q_{n,I}}$ where $\mathfrak{m}_{Q_{n,I}}$ is the ideal generated by all the arrows of $Q_{n, I}$ (see \ref{QP}).
\item $e_k$ for all $1 \leq k < i-1$ and $i+1 <  k \leq n$.
\item possible loops $\x_{j-1}$ on vertex $i-1$ and $\x_{j+2}$ on vertex $i+1$. 
\end{itemize}
It is clear that $J \cong \Jac(Q', g)$ where $g\sim  \uplambda_1 \x_{j}^2 +  \uplambda_2\x_j'\x_{j+1}+\uplambda_3 \x_{j+1}^2$ for some $\uplambda_1, \ \uplambda_2 , \ \uplambda_3\in \C$. 
Then we suppose that $e_i \Jac(f) e_i$ is commutative and $f$ does not contain $\x_j'\x_{j+1}$, and aim for a contradiction.

Since $e_i \Jac(f) e_i$ is commutative and $e_iJe_i$ is a factor of $e_i \Jac(f) e_i$, then $e_iJe_i$ is also commutative, and furthermore $\x_{j}'\x_{j+1}=\x_{j+1}\x_j'$ in $e_iJe_i$. Since $f$ does not contain $\x_j'\x_{j+1}$, $g \sim \uplambda_1 \x_{j}^2 +\uplambda_3 \x_{j+1}^2 $. It is clear that the four relations induced by differentiating $ \uplambda_1 \x_{j}^2 +\uplambda_2 \x_{j+1}^2$ can not induce the relation $(b_ja_j)(a_{j+1}b_{j+1})= (a_{j+1}b_{j+1})(b_ja_j)$. 
Thus $\x_{j}'\x_{j+1} \neq \x_{j+1}\x_j'$ in $e_iJe_i$, a contradiction.

(2) When $\x_j$ is not a loop and $\x_{j+1}$ is a loop, then there exists a vertex $i \notin I$ such that $Q_{n, I}$ at vertex $i$ locally looks like the following.
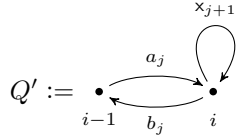
\[
\begin{array}{c}
\begin{tikzpicture}[scale=1.5]
\node (a) at (-1,0) [vertex] {};
\node (b) at (0,0) [vertex] {};

\node (a1) at (-1,-0.25) {$\scriptstyle i-1$};
\node (a2) at (0,-0.25) {$\scriptstyle i$};

\draw[->,bend left,looseness=0.7] (a) to node[above] {$\scriptstyle a_{j}$} (b);
\draw[<-,bend right,looseness=0.7] (a) to node[below] {$\scriptstyle b_{j}$} (b);
\draw[->]  (b) edge [in=55,out=120,loop,looseness=18] node[above]{$\scriptstyle \x_{j+1}$}  (b);

\node (z) at (-1.5,0) {$Q'\colonequals$};
 \end{tikzpicture}
\end{array}
\]

We again denote the above quiver with the only two vertices shown, as $Q'$. Then consider the noncommutative algebra $J$, defined as $\Jac( f)$ quotiented by the ideal generated by the following paths:
\begin{itemize}
\item $\mathfrak{m}^4_{Q_{n,I}}$ where $\mathfrak{m}_{Q_{n,I}}$ is the ideal generated by all the arrows of $Q_{n, I}$ (see \ref{QP}).
\item $e_k$ for all $1 \leq k < i-1$ and $i <  k \leq n$.
\item the possible loop $\x_{j-1}$ on vertex $i-1$. 
\item $\x_{j+1}^2$.
\end{itemize}
It is clear that $J \cong \Jac(Q', g)/(\x_{j+1}^2)$ where $g\sim  \uplambda_1 \x_{j}^2 +  \uplambda_2\x_j'\x_{j+1}+\uplambda_3\x_j'\x_{j+1}^2+\uplambda_4 \x_{j+1}^2+\uplambda_5 \x_{j+1}^3+\uplambda_6 \x_{j+1}^4$ for some $\uplambda_k \in \C$. 
We suppose that $e_i \Jac(f) e_i$ is commutative and $f$ does not contain $\x_j'\x_{j+1}$, and aim for a contradiction.

Since $e_i \Jac(f) e_i$ is commutative and $e_iJe_i$ is a factor of $e_i \Jac(f) e_i$, then $e_iJe_i$ is also commutative, and furthermore $\x_{j}'\x_{j+1}=\x_{j+1}\x_j'$ in $e_iJe_i$. Since $f$ does not contain $\x_j'\x_{j+1}$, $\uplambda_2=0$. Since $f$ is reduced, $\uplambda_4=0$. Thus 
\begin{equation*}
    J \cong \frac{\C\lbl Q'\rbl}{(\uplambda_3 (b_ja_j)\x_{j+1}+\uplambda_3 \x_{j+1}(b_ja_j),\uplambda_1b_ja_jb_j,\uplambda_1a_jb_ja_j,\x_{j+1}^2)}.
\end{equation*}
Again, it is clear that the above relations can not induce $(b_ja_j)\x_{j+1} =\x_{j+1}(b_ja_j)$. Thus $\x_{j}'\x_{j+1} \neq \x_{j+1}\x_j'$ in $e_iJe_i$, a contradiction.

(3) When $\x_j$ is a loop and $\x_{j+1}$ is not a loop, the proof is analogous to that of (2).
\end{proof}

\begin{prop}\label{514}
Given any Type $A_{n, I}$ crepant resolution $\uppi \colon \scrX \rightarrow \Spec \scrR$, there exists a reduced Type $A$ potential $f$ on $Q_{n, I}$ such that $\Lambda_{\mathrm{con}}(\uppi) \cong \Jac(f)$.   
\end{prop}
\begin{proof}
By \ref{35}, the NCCR $\Lambda(\uppi)$ can be presented as the quiver in \ref{35} with some relations. Since $\scrR$ is complete local, $\Lambda(\uppi)$ is also complete local by e.g. \cite[8.4]{BW2}. Moreover, $\Lambda(\uppi)$ is $3$-CY \cite[2.8]{IW2}. 
Since any complete local $3$--CY algebra is a Jacobi algebra, as shown in \cite[\S10]{VM}, the relations of $\Lambda(\uppi)$ are generated by some reduced potential $g$.
Since $\Lambda_{\mathrm{con}}(\uppi) \cong \Lambda(\uppi)/\langle e_0 \rangle$,
it follows that $\Lambda_{\mathrm{con}}(\uppi)$ is isomorphic to $\Jac(Q_{n,I},f)$ for some reduced potential $f$, obtained from $g$ by deleting all terms involving arrows with source or target at the vertex $0$.

It remains to show that $f$ is of Type~$A$, that is, that $f$ contains $\x_i'\x_{i+1}$ for every $1\leq i\leq m-1$.
Since $\Lambda_{\mathrm{con}}(\uppi) \cong \Jac(f)$ and $e_i \Lambda_{\mathrm{con}}(\uppi) e_i$ is commutative for each $1 \leq i \leq n$ by \ref{591}, $e_i \Jac(f) e_i$ is also commutative for each $i$. 
Thus $f$ contains $\x_i^{\prime}\x_{i+1}$ for each $i$, by \ref{581}.
\end{proof}

We are now in a position to show that our definition of Type A potential \ref{TypeA} is intrinsic.
\begin{cor}\label{cor: TypeA_intrin}
Let $f$ be a reduced potential on $Q_{n, I}$. The following are equivalent.
\begin{enumerate}
   \item $f$ is Type A.
   \item There exists a Type $A_{n, I}$ crepant resolution $\uppi$ such that $\Jac(f) \cong \Lambda_{\mathrm{con}}(\uppi)$.
   \item $e_i \Jac(f) e_i$ is commutative for $1 \leq i \leq n$.
\end{enumerate}
\end{cor}
\begin{proof}
(1) $\Rightarrow$ (2): Since $f$ is a reduced Type A potential on $Q_{n, I}$, it is immediate by \ref{511}.

(2) $\Rightarrow$ (3): Since $\uppi$ is a Type $A_n$ crepant resolution, then $e_i \Lambda_{\mathrm{con}}(\uppi)e_i$ is commutative by \ref{591}, and so $e_i \Jac(f) e_i$ is commutative for any $1 \leq i \leq n$. 

(3) $\Rightarrow$ (1): Since $f$ is a reduced potential on $Q_{n, I}$ and $e_i \Jac(f) e_i$ is commutative for any $1 \leq i \leq n$, then $f$ contains $\x_i'\x_{i+1}$ for any $1\leq i \leq m-1$ by \ref{581}, and so $f$ is Type A.
\end{proof}

\begin{comment}
\begin{cor}\label{cor: TypeA_intrin2}
Let $f$ be a potential on $Q_n$ and $\deg(f) \geq 2$. The following are equivalent.
\begin{enumerate}
\item $f$ is Type A.
\item There exists a Type $A_{n}$ crepant resolution $\uppi$ such that $\Jac(f) \cong \Lambda_{\mathrm{con}}(\uppi)$.
\item $e_i \Jac(f) e_i$ is commutative for $1 \leq i \leq n$.
\end{enumerate}
\end{cor}
\begin{proof}
(1) $\Rightarrow$ (2): Since $f$ is a Type A potential on $Q_{n}$, it is immediate by \ref{511}.    

(2) $\Rightarrow$ (3) is similar to \ref{cor: TypeA_intrin}.

(3) $\Rightarrow$ (1): Note that $f_{\textnormal{red}}$ is reduced on $Q_{n,I}$ for some $I \subseteq \{1,2,\dots,n\}$ and $\Jac(f) \cong \Jac(f_{\textnormal{red}})$. Since $e_i \Jac(f) e_i$ is commutative, then $e_i \Jac(f_{\textnormal{red}}) e_i$ is commutative for $1 \leq i \leq n$, and so $f_{\textnormal{red}}$ is Type A by \ref{cor: TypeA_intrin}. Together with $\deg(f) \geq 2$ gives that $f$ is Type A.
\end{proof}
\end{comment}

\begin{definition}
We say two crepant resolutions $\uppi_i \colon \scrX_i \rightarrow \Spec \scrR_i$ for $i=1, 2$ have the same \emph{noncommutative deformation type (NC deformation type)} if $\Lambda_\mathrm{con}(\uppi_1) \cong \Lambda_\mathrm{con}(\uppi_2)$. 
\end{definition}

The name NC deformation type comes from the fact that the contraction algebra represents the noncommutative deformation functor of the exceptional curves \cite{DW2}. 

Together with \ref{047}, the above \ref{514} induces a map $\varphi$ from Type $A_{n, I}$ crepant resolutions to the isomorphism classes of reduced monomialized Type $A$ potentials on $Q_{n, I}$. More precisely, for any Type $A_{n, I}$ crepant resolution $\uppi$, we define $\varphi(\uppi)$ to be the reduced monomialized Type $A$ potential $f$ on $Q_{n, I}$ that satisfies $\Lambda_{\mathrm{con}}(\uppi) \cong \Jac(f)$ by \ref{514} and \ref{047}. Moreover, $\varphi$ is well-defined since if there are two such $f_1$ and $f_2$, then $\Jac(f_1)\cong \Jac(f_2)$.

\begin{theorem}\label{515}
The above $\varphi$ induces a one-to-one correspondence as follows.  
\[
\begin{tikzpicture}[bend angle=0, looseness=0]
\node[align=left] (a) at (0,-0.5) {isomorphism classes of reduced monomialized Type $A$ potentials on $Q_{n, I}$};
\node[align=left] (b) at (0,0.75) {Type $A_{n, I}$ crepant resolutions up to NC deformation type};
\draw[->,bend left] (b) to node[below] {} (a);
\draw[->,bend left] (a) to node[above] {} (b);
\end{tikzpicture}
\] 
\end{theorem}
\begin{proof}
Firstly, we prove the map from top to bottom is surjective, namely that for any reduced monomialized Type $A_{n, I}$ potential $f$, there is a Type $A_{n, I}$ crepant resolution $\uppi \colon \scrX \rightarrow \Spec\scrR$ such that $\Jac(f) \cong \Lambda_{\mathrm{con}}(\uppi)$. This is immediate from \ref{511}.

Then we prove that the map from top to bottom is injective. Let $\uppi \colon \scrX_k \rightarrow \Spec \scrR_k$ be two Type $A_{n, I}$ crepant resolutions for $k=1,2$. If $\Lambda_{\mathrm{con}}(\uppi_1) \cong \Jac(f) \cong \Lambda_{\mathrm{con}}(\uppi_2)$ for some reduced monomialized Type $A$ potential $f$ on $Q_{n, I}$, then $\uppi_1$ and $\uppi_2$ have the same NC deformation type. 
\end{proof}
The following asserts that Type $A$ potentials on $Q_n$ describe the contraction algebra of all Type $A_n$ crepant resolutions.
\begin{cor}\label{516}
The set of isomorphism classes of contraction algebras associated to Type $A_n$ crepant resolutions is equal to the set of isomorphism classes of Jacobi algebras of monomialized Type $A$ potentials on $Q_n$.
\end{cor}
\begin{proof}
We first define a map $\phi$ from the isomorphism classes of contraction algebra associated with Type $A_n$ crepant resolutions to the isomorphism classes of Jacobi algebra of monomialized Type $A$ potentials on $Q_n$. 

Given any contraction algebra $\Lambda_{\mathrm{con}}(\uppi)$ where $\uppi$ is a Type $A_n$ crepant resolution, then $\uppi$ belongs to Type $A_{n, I}$ crepant resolution for some $I$. Thus $\Lambda_{\mathrm{con}}(\uppi) \cong \Jac(Q_{n,I}, f')$ for some reduced monomialized Type $A$ potential $f'$ on $Q_{n, I}$ by \ref{515}. Moreover, $f'$ is isormphic to some monomialized Type $A$ potential $f$ on $Q_n$ by \ref{410}. We define $\phi(\Lambda_{\mathrm{con}}(\uppi))\colonequals \Jac(f)$. 

Secondly, we prove that $\phi$ is well-defined. If there are two Type $A_n$ crepant resolutions $\uppi_1$ and $\uppi_2$ such that $\Lambda_{\mathrm{con}}(\uppi_1) \cong \Lambda_{\mathrm{con}}(\uppi_2)$, then $\phi(\Lambda_{\mathrm{con}}(\uppi_1))=\Jac(f_1)$ and $\phi(\Lambda_{\mathrm{con}}(\uppi_2))=\Jac(f_2)$, so $\Jac(f_1) \cong \Jac(f_2)$ from the above definition of $\phi$.

Thirdly, we prove that $\phi$ is injective. If there are two Type $A_n$ crepant resolutions $\uppi_1$ and $\uppi_2$ such that $\phi(\Lambda_{\mathrm{con}}(\uppi_1))\cong \Jac(f) \cong \phi(\Lambda_{\mathrm{con}}(\uppi_2))$ for some monomialized Type $A$ potential $f$ on $Q_n$, then $\Lambda_{\mathrm{con}}(\uppi_1) \cong \Lambda_{\mathrm{con}}(\uppi_2)$ from the above definition of $\phi$.

Finally, by \ref{511} $\phi$ is surjective.
\end{proof}

\begin{notation}\label{notation:derived}
Let $f$ and $g$ be potentials on a quiver $Q$.
We say that $f$ is \emph{derived equivalent to} $g$ (written $f \simeq g$) if the derived categories $\mathrm{D}^{\mathrm{b}}(\Jac(f))$ and $\mathrm{D}^{\mathrm{b}}(\Jac(g))$ are triangle equivalent.
\end{notation}

Given any isolated $cA_n$ singularity $\scrR$ which admits a crepant resolution, let $\uppi \colon \scrX \rightarrow \Spec\scrR$ be one of the crepant resolutions. Then, by \ref{516}, there exists some monomialized Type $A$ potential $f$ on $Q_n$ such that $\Lambda_{\mathrm{con}}(\uppi) \cong \Jac(f)$, so it induces a map $\Phi$ from isolated $cA_n$ singularities, which admit a crepant resolution to monomialized Type $A$ potentials on $Q_n$.
\begin{theorem}\label{cor}
The above $\Phi$ induces a one-to-one correspondence as follows.    
\[
\begin{tikzpicture}[bend angle=0, looseness=0]
\node[align=left] (a) at (0,-1) {derived equivalence classes of monomialized Type $A$ potentials on $Q_n$\\ with finite-dimensional Jacobi algebra};
\node[align=left] (b) at (0,0.5) {isomorphism classes of isolated $cA_n$ singularities\\ which admit a crepant resolution};
\draw[->,bend left] (b) to node[below] {} (a);
\draw[->,bend left] (a) to node[above] {} (b);
\end{tikzpicture}
\]
\end{theorem}
\begin{proof}
Firstly, we prove that the map from top to bottom is well-defined. Given any isolated $cA_n$ $\scrR$ which admits a crepant resolution, let $\uppi \colon \scrX \rightarrow \Spec \scrR$ be one of the crepant resolutions. Then there exists some monomialized Type $A$ potential $f$ on $Q_n$ such that $\Lambda_{\mathrm{con}}(\uppi) \cong \Jac(f)$ by \ref{516}. Moreover, since $\scrR$ is isolated, $\Jac(f)$ is finite-dimensional by \ref{352}. Let $\pi' \colon \scrX' \rightarrow \Spec \scrR$ be another crepant resolution such that $\Lambda_{\mathrm{con}}(\uppi') \cong \Jac(f')$ for some monomialized Type $A$ potential $f'$ on $Q_n$. Since $\uppi'$ is a flop of $\uppi$ and $\scrR$ is isolated, $f$ is derived equivalent to $f'$ by \ref{thm:august}.

Secondly, we prove that the map from top to bottom is surjective. Given any monomialized Type $A$ potential $f$ on $Q_n$ with finite-dimensional Jacobi algebra, there exists a Type $A_n$ crepant resolution $\uppi \colon \scrX \rightarrow \Spec\scrR$ such that $\Jac(f) \cong \Lambda_{\mathrm{con}}(\uppi)$ by \ref{gr}. Moreover, since $\Jac(f)$ is finite-dimensional, $\scrR$ is isolated by \ref{352}.

Finally, we prove the map from top to bottom is injective. 
This uses the proof of the Donovan-Wemyss conjecture in \ref{037}.
Let $\uppi_i \colon \scrX_i \rightarrow \Spec \scrR_i$ be two crepant resolutions of isolated $cA_n$ $\scrR_i$ with $\Lambda_{\mathrm{con}}(\uppi_i)\cong \Jac(f_i)$ for $i=1, 2$. If $f_1$ is derived equivalent to $f_2$, together with $R_1$ and $R_2$ isolated, then $R_1 \cong R_2$ by \ref{037}.
\end{proof}

\begin{remark}
Since both \ref{037} and \ref{thm:august} need the assumption of isolated cDVs, we only prove the correspondence in \ref{cor} for isolated $cA_n$ singularities.
Moreover, by testing contraction algebras of crepant resolutions of the non-isolated $cA_2$ singularity $\C \lal u,v,x,y \ral / (uv-x^2y)$, it seems that \ref{cor} can not be generalised directly to the non-isolated cases.    
\end{remark}

\section{Special cases: \texorpdfstring{$A_3$}{A3}}\label{section:A3}
This section considers the special case $Q_{3,\{1,2,3\}}$, namely  
\[
\begin{array}{cl}
\begin{array}{c}
\begin{tikzpicture}[bend angle=30, looseness=0.9]
\node (a) at (-2,0) [vertex] {};
\node (b) at (0,0) [vertex] {};
\node (c) at (2,0) [vertex] {};
\node (a1) at (-2,-0.2) {$\scriptstyle 1$};
\node (a1) at (0,-0.2) {$\scriptstyle 2$};
\node (a1) at (2,-0.2) {$\scriptstyle 3$};
\draw[<-,bend right] (a) to node[below] {$\scriptstyle b_1$} (b);
\draw[->,bend left] (a) to node[above] {$\scriptstyle a_1$} (b);
\draw[<-,bend right] (b) to node[below] {$\scriptstyle b_2 $} (c);
\draw[->,bend left] (b) to node[above] {$\scriptstyle a_2$} (c);
%\node (z) at (0,-1.2) {Quiver $Q_{3,\{1,2,3\}}$};
\end{tikzpicture}
\end{array}
&
\begin{array}{l}
\x_1=a_1b_1, \ \x_1'=b_1a_1\\
\x_2=a_2b_2, \ \x_2'=b_2a_2
\end{array}
\end{array}
\]
We classify Type~$A$ potentials on $Q_{3,\{1,2,3\}}$ up to isomorphism, and determine the derived equivalence classes among those with finite-dimensional Jacobi algebras. This generalises results in \cite{DWZ1,E1,H}.

\begin{notation}\label{no6}
In this section, for simplicity, we will adopt the following notation.
Recall the notation $f_d$, $f_{ \geq d}$ in \ref{no5}.
\begin{enumerate}
\item Write $Q$ for $Q_{3,\{1,2,3\}}$, $\x \colonequals \x_1'$ and $\y \colonequals \x_2$, whereas $\x' \colonequals \x_1$ and $\y' \colonequals \x_2'$.

\item Suppose that $f$ is a Type~$A$ potential on $Q$.
Since the aim of this section is to classify Type~$A$ potentials up to isomorphism,
we may assume that $f$ is reduced.
Define the \emph{base part} of $f$ by
\[
f_b \colonequals \upkappa_1\x^p+\x\y+\upkappa_2\y^q,
\]
where $\upkappa_1\x^p$ and $\upkappa_2\y^q$ are the lowest degree (nonzero) pure powers of $\x$ and $\y$
appearing in $f$, respectively.
Since $f$ is reduced, we necessarily have $p,q\ge 2$.
If no pure power of $\x$ (respectively, $\y$) appears in $f$, we set $\upkappa_1=0$ (respectively, $\upkappa_2=0$).
The \emph{redundant part} of $f$ is then defined by $f_r\colonequals f-f_b$.

\item Given any Type A potential $f$ on $Q$ with $f_b= \upkappa_1\x^p+\x\y+\upkappa_2\y^q$, we give a new definition of degree as follows, which differs from \ref{deg}. For any $t \geq 0$, define
\begin{equation*}
    \deg(\x^{p+t}) \colonequals t+2, \quad \deg(\y^{q+t}) \colonequals t+2.
\end{equation*}
The degree of binomials in $\x$ and $\y$ is the same as \ref{deg}.
We also write $f_d$ for the degree $d$ piece of $f$ with respect to this new grading (overwriting \ref{no5}). Similar for $f_{ij,d}$, $\scrO_d$ and $\scrO_{ij,d}$.
This new definition of degree is natural since now $f_b=f_2$ and $f_r=f_{\geq 3}$, which will unify the proof below.

\item Given any Type A potential $f$ on $Q$ with $f_b= \upkappa_1\x^p+\x\y+\upkappa_2\y^q$, consider the matrices
\begin{equation*}
    A_{11}(f) = \begin{bmatrix}  p\upkappa_1 \end{bmatrix},
    \qquad
     A_{22}(f) = \begin{bmatrix}  q\upkappa_2 \end{bmatrix},
     \qquad
     A_{12}(f) \colonequals
\begin{bmatrix}
    \upvarepsilon_{12}(f)  & 1\\
    1 & \upvarepsilon_{22}(f)
\end{bmatrix},
\end{equation*}
where $\upvarepsilon_{12}(f)= \begin{cases}
    2\upkappa_1 \text{ if } p=2\\
    0 \ \ \ \text{ if } p>2
\end{cases}$
and 
$\upvarepsilon_{22}(f)= \begin{cases}
    2\upkappa_2 \text{ if } q=2\\
    0 \ \ \ \text{ if } q>2
\end{cases}$.

\end{enumerate}
\end{notation}

\subsection{Normalization}\label{A3_norm}
The purpose of this subsection is to prove \ref{610}, which gives the isomorphism classes of Type A potentials on $Q$. 

\begin{notation}\label{no7}
In this section, we assume $f=f_b+f_r$ is a Type A potential on $Q$ with $f_b= \upkappa_1\x^p+\x\y+\upkappa_2\y^q$, and will freely use the notations $f_d$, $f_{\geq d}$, $f_{ij,d}$, $\scrO_d$ and $\scrO_{ij,d}$ in \ref{no6}(3).    
\end{notation}

The following results show that, up to right-equivalence, we can freely commute occurrences of $\y\x$ into $\x\y$ in the redundant part $f_r$.
\begin{lemma}\label{lemma:commute}
Suppose that $f_r=g+\uplambda \y\x c$ where $\uplambda \in  \C^{\times}$, $d \colonequals \deg(\y\x c)$ and $c$  is a cycle with $\deg(c) \geq 1$. Then there exists a path degree $d-1$ right-equivalence, 
\begin{equation*}
    \upvartheta \colon f \stackrel{d-1}{\rightsquigarrow} f_b+g+\uplambda\x\y c +\scrO_{d+1}.
\end{equation*}
\end{lemma}
\begin{proof}
We apply the unitriangular automorphism $\upvartheta$ of depth $d-1$ given by
$\upvartheta: a_1 \mapsto a_1 - \uplambda a_1c$, $b_1 \mapsto b_1 + \uplambda cb_1$.
Equivalently, since $\x=b_1a_1$, the induced action is
\[
\upvartheta(\x)
=
\x - \uplambda \x c + \uplambda c \x - \uplambda^2 c\x c .
\]
Applying $\upvartheta$ to $f$ gives
\begin{align*}
 \upvartheta \colon f & \mapsto \upkappa_1(\x- \uplambda \x c+ \uplambda c \x - \uplambda^2 c\x c)^p + (\x- \uplambda \x c+ \uplambda c \x - \uplambda^2 c\x c)\y +\upkappa_2\y^q +\uplambda \y\x c+g +\scrO_{d+1} \\
  & = f_b - \uplambda \x c\y +\uplambda c \x\y+\uplambda \y\x c+g +\scrO_{d+1} \tag{since $f_b= \upkappa_1\x^p+\x\y+\upkappa_2\y^q$}\\
   & \stackrel{d}{\sim} f_b +g +\uplambda  \x\y c+\scrO_{d+1}.
\end{align*}    
The above equality holds as follows.
First, as in the proof of \ref{042}, all terms produced from
$f_r=g+\uplambda \y\x c$ under $\upvartheta$ have degree at least $d+1$, and
may therefore be absorbed into $\scrO_{d+1}$ in the first line.
Second, since $\deg(c)\ge1$ by assumption and $p\ge2$ as specified in \ref{no6}(2), every term in
\[
\upkappa_1\big(\x - \uplambda \x c + \uplambda c \x - \uplambda^2 c\x c\big)^p
\]
other than $\upkappa_1\x^p$ has degree at least $d+1$; similarly,
$\deg(\uplambda^2 c\x c\y)\ge d+1$.
Hence all such terms are again absorbed in $\scrO_{d+1}$ in the second line.
\end{proof}

\begin{prop}\label{prop:commute}
Suppose that $f_r= g + \uplambda c $ where $\uplambda \in  \C^{\times}$, $d \colonequals \deg(c)$ and $c$ is a cycle with $\mathbf{T}(c)_1=i$ and $\mathbf{T}(c)_2=j$. Then there exists a path degree $d-1$ right-equivalence,
\begin{equation*}
    \uptheta \colon f \stackrel{d-1}{\rightsquigarrow}  f_b + g+\uplambda \x^i \y^j + \scrO_{d+1}.
\end{equation*}  
\end{prop}
\begin{proof}
If $j=0$, then $c \sim \x^i$, so there is nothing to prove. The case of $i=0$ is similar. Thus we assume $i$, $j >0$. Firstly, note that $c \sim \x^{i_1}\y^{j_1}\x^{i_2}\y^{j_2} \dots \x^{i_k}\y^{j_k}$ where $\sum_{t=1}^k i_t=i$ and $\sum_{t=1}^k j_t=j$. 
Since lemma \ref{lemma:commute} allows us to commute each occurrence of $\y\x$ into $\x\y$ (up to $\scrO_{d+1}$),
iterating it yields a path degree $d-1$ right-equivalence
\begin{equation*}
\uptheta \colon f \stackrel{d-1}{\rightsquigarrow}  f_b + g+\uplambda \x^i \y^j + \scrO_{d+1}. \qedhere
\end{equation*}
\end{proof}

\begin{remark}\label{412}
With notation as in \ref{prop:commute}, any term $\uplambda c$ of degree $d=\deg(c)$ can be rewritten, up to $\scrO_{d+1}$, in the form $\uplambda \x^i\y^j$.
Henceforth, since we normalise the potential degree by degree, we may assume throughout this section that $c=\x^i\y^j$.
\end{remark}

We now normalise $f$ degree by degree: the base part $f_b$ is used to eliminate higher-degree terms in $f_r$ via unitriangular automorphisms.

For any integer $s \geq 1$, we define the following depth $s+1$ unitriangular automorphisms.
\begin{align}
&\upvarphi_{11,s} \colon a_1 \mapsto a_1+\uplambda a_1\x^s, \label{norm11}\\
& \upvarphi_{22,s} \colon a_2 \mapsto a_2+\uplambda \y^s a_2, \label{norm22} \\
& \upvarphi_{12,s} \colon a_1 \mapsto a_1+\uplambda_1a_1 \x^{s-1}\y, \ a_2 \mapsto a_2+\uplambda_2 \x^{s}a_2, \label{norm12} 
\end{align}
where $\uplambda$, $\uplambda_1$, $\uplambda_2 \in \C$.

\begin{lemma}\label{lemma:norm11}
The $\upvarphi_{11,s}$ \eqref{norm11} induces a degree $s+1$ right-equivalence,
\begin{equation*}
\upvarphi_{11,s} \colon  f \stackrel{s+1}{\rightsquigarrow} f+ \uplambda p\upkappa_1 \x^{p+s} +\scrO_{12,s+2}+\scrO_{s+3}.
\end{equation*}
\end{lemma}
\begin{proof}
Applying $\upvarphi_{11,s} \colon a_1 \mapsto a_1+\uplambda a_1\x^s$ to $f$ gives
\begin{align*}
\upvarphi_{11,s} \colon  f & \mapsto \upkappa_1(b_1(a_1+\uplambda a_1\x^s))^p+ b_1(a_1+\uplambda a_1\x^s)\y +  \upkappa_2 \y^q +f_r +\scrO_{s+3}\\
& \stackrel{s+2}{\sim} f + \uplambda p\upkappa_1\x^{p+s}+ \uplambda\x^{s+1}\y+ \scrO_{s+3}\\
& = f + \uplambda p\upkappa_1\x^{p+s}+ \scrO_{12,s+2}+ \scrO_{s+3}.\qedhere
\end{align*}
\end{proof}

\begin{lemma}\label{lemma:norm22}
The $\upvarphi_{22,s}$ \eqref{norm22} induces a degree $s+1$ right-equivalence,
\begin{equation*}
\upvarphi_{22,s} \colon  f \stackrel{s+1}{\rightsquigarrow} f+ \uplambda q\upkappa_2 \y^{q+s} +\scrO_{12,s+2}+\scrO_{s+3}.
\end{equation*}    
\end{lemma}
\begin{proof}
The proof is analogous to that of \ref{lemma:norm11}.
\end{proof}
\begin{remark}
Lemma~\ref{lemma:norm22} can be deduced from \ref{lemma:norm11} by applying
the automorphism of $Q$ which interchanges $e_1$ and $e_3$ and fixes $e_2$.
Explicitly, this automorphism is given by
\[
a_1 \mapsto b_2,\quad b_1 \mapsto a_2,\quad
a_2 \mapsto b_1,\quad b_2 \mapsto a_1,
\]
or equivalently, it interchanges $\x$ and $\y$.
\end{remark}

\begin{lemma}\label{lemma:norm12}
The $\upvarphi_{12,s}$ \eqref{norm12} induces a degree $s+1$ right-equivalence,
\begin{equation*}
\upvarphi_{12,s} \colon  f \stackrel{s+1}{\rightsquigarrow} f+ \begin{bmatrix}
    \x^{s+1}\y & \x^s \y^2 \end{bmatrix}A_{12}(f)\begin{bmatrix}
        \uplambda_1 \\ \uplambda_2 \end{bmatrix} +\scrO_{s+3}.
\end{equation*}
\end{lemma}
\begin{proof}
Applying $\upvarphi_{12,s}$ to $f$ gives 
\begin{align*}
 \upvarphi_{12,s} \colon f & \mapsto  \upkappa_1(\x+\uplambda_1\x^{s}\y)^p + (\x+\uplambda_1\x^{s}\y)(\y+\uplambda_2 \x^{s}\y) +\upkappa_2(\y+\uplambda_2 \x^{s}\y)^q +f_r +\scrO_{s+3}   \\
& \stackrel{s+2}{\sim} f + \uplambda_1 p \upkappa_1 \x^{p+s-1}\y+ \uplambda_2 \x^{s+1}\y+\uplambda_1 \x^s\y^2+ \uplambda_2 q \upkappa_2 \x^s\y^q +\scrO_{s+3} \\
& = f+ \begin{bmatrix}
    \x^{s+1}\y & \x^s \y^2 \end{bmatrix}A_{12}(f)\begin{bmatrix}
        \uplambda_1 \\ \uplambda_2 \end{bmatrix} +\scrO_{s+3}.
\end{align*}
Recall the definition of $A_{12}(f)$ in \ref{no6}(4).
In the last line, the term $\uplambda_1 p \upkappa_1 \x^{p+s-1}\y$ contributes to the displayed matrix expression
only when $p=2$ (and is absorbed into $\scrO_{s+3}$ when $p>2$); similarly,
$\uplambda_2 q \upkappa_2 \x^s\y^q$ contributes only when $q=2$.
\end{proof}

\begin{prop}\label{prop:norm}
With notation in \textnormal{\ref{no7}},
for any $d \geq 3$, there exists a path degree $d-1$ right-equivalence
\begin{equation*}
\upphi_d \colon f \stackrel{d-1}{\rightsquigarrow}  f_{<d} +c_d+ \scrO_{d+1},
\end{equation*}  
where the $c_d$ is defined to be
\begin{equation*}
    c_d = \left\{
    \begin{array}{cl}
      0 & \text{if } \det(A_{12}(f)) \neq 0\\
      \upmu \x^{p+d-2}  & \text{if } \det(A_{12}(f)) = 0
    \end{array}\right.
\end{equation*}
for some $\upmu \in \C$.
\end{prop}

\begin{proof}
We first rewrite $f_r=f_d+g$ and $f_d=f_{11,d}+f_{12,d}+f_{22,d}$ where $f_{11,d}=\upalpha_1\x^{p+d-2}$ and $f_{22,d}=\upalpha_2\y^{q+d-2}$ for some $\alpha_1, \alpha_2 \in \C$. 

Recall that $f_b= \upkappa_1\x^p+\x\y+\upkappa_2\y^q$ in \ref{no7}.
If $\upkappa_2=0$, then there is no monomial of $\y$ in $f$, and so $\alpha_2=0$. Otherwise, $\upkappa_2 \neq 0$, so set $\uplambda= -\upalpha_2/(q\upkappa_2)$ and applying \ref{lemma:norm22} to obtain,
\begin{align}
  \upvarphi_{22,d-2}  \colon f & \stackrel{d-1}{\rightsquigarrow}  f+ \uplambda q\upkappa_2 \y^{q+d-2} +\scrO_{12,d}+\scrO_{d+1}\notag\\
  & = f_b + g +f_d+ \uplambda q\upkappa_2 \y^{q+d-2} +\scrO_{12,d}+\scrO_{d+1} \tag{$f= f_b+f_r, \ f_r=f_d+g$} \\
  & = f_b+g+f_{11,d}+f_{12,d}+(\upalpha_2+ \uplambda q\upkappa_2)\y^{q+d-2}  +\scrO_{12,d}+\scrO_{d+1} \tag{$f_d=f_{11,d}+f_{12,d}+f_{22,d}, \ f_{22,d}=\upalpha_2\y^{q+d-2}$}\\
  & = f_b+g+f_{11,d}+f_{12,d}+\scrO_{12,d}+\scrO_{d+1} \tag{since $\uplambda= -\upalpha_2/(q\upkappa_2)$}\\
  & =f_b+g+f_{11,d}+\scrO_{12,d}+\scrO_{d+1}.\label{30}
\end{align}

Set $\mathsf{f}_1 \colonequals  f_b+g+f_{11,d}+\scrO_{12,d}+\scrO_{d+1}$.
The proof splits into cases.

(1) $\det(A_{12}(f)) =0$.

By \ref{046}, when $\det(A_{12}(f))=0$ we may absorb the degree $d$ binomial terms $\scrO_{12,d}$ into monomial terms in $\x$ (up to $\scrO_{d+1}$). More precisely, there exists a path degree $d-1$ right-equivalence,
\begin{align*}
    \rho_d \colon \mathsf{f}_1 
    & \stackrel{d-1}{\rightsquigarrow}   f_b+g+\upmu \x^{p+d-2}+\scrO_{d+1},  \tag{$f_{11,d}=\upalpha_1\x^{p+d-2}$}\\
    & = f_{<d}+\upmu \x^{p+d-2}+\scrO_{d+1}. \tag{$f_b+g= f_{<d}+f_{>d}$}
\end{align*}
for some $\upmu \in \C$. Set $\upphi_d \colonequals \rho_d \circ \upvarphi_{22,d-2}$, we are done.

(2) $\det(A_{12}(f)) \neq 0$.

Similar to \eqref{30}, applying $ \ref{lemma:norm11}$ to $\mathsf{f}_1$ gives
\begin{equation*}
    \upvarphi_{11,d-2}  \colon \mathsf{f}_1  \stackrel{d-1}{\rightsquigarrow} \mathsf{f}_2 \colonequals f_b+g+\scrO_{12,d}+\scrO_{d+1}.
\end{equation*}

Then we continue to normalize the $\scrO_{12,d}$ in $\mathsf{f}_2$. 
By construction, $\mathsf{f}_2$ satisfies the hypotheses of \ref{042}, hence we may apply \ref{042} repeatedly to obtain
\begin{align*}
  \upvartheta \colon  \mathsf{f}_2 \stackrel{d-1}{\rightsquigarrow} \mathsf{f}_3 
 \colonequals f_b+g+\upbeta \x^{d-1}\y  +\scrO_{d+1}
\end{align*}
for some $\upbeta \in \C$. Then by \ref{lemma:norm12}, 
\begin{align*}
\upvarphi_{12,d-2} \colon \mathsf{f}_3 
& \stackrel{d-1}{\rightsquigarrow}  \mathsf{f}_3 +\begin{bmatrix}
    \x^{d-1}\y & \x^{d-2} \y^2 \end{bmatrix}A_{12}(f)\begin{bmatrix}
        \uplambda_1 \\ \uplambda_2 \end{bmatrix} +\scrO_{d+1} \\
        &  = f_b+g+\upbeta \x^{d-1}\y+\begin{bmatrix}
    \x^{d-1}\y & \x^{d-2} \y^2 \end{bmatrix}A_{12}(f)\begin{bmatrix}
        \uplambda_1 \\ \uplambda_2 \end{bmatrix} +\scrO_{d+1}.
\end{align*}

Since $\det A_{12}(f)\neq 0$, the linear system in $(\uplambda_1,\uplambda_2)$ is solvable, so we may choose $\uplambda_1,\uplambda_2$ such that
\begin{equation*}
    \upbeta \x^{d-1}\y+ \begin{bmatrix}
    \x^{d-1}\y & \x^{d-2} \y^2 \end{bmatrix}A_{12}(f)\begin{bmatrix}
        \uplambda_1 \\ \uplambda_2 \end{bmatrix}=0.
\end{equation*}        
Thus we have
\begin{align*}
    \upvarphi_{12,d-2} \colon \mathsf{f}_3 
    &\stackrel{d-1}{\rightsquigarrow} f_b+g +\scrO_{d+1}\\
    & =f_{<d}+ \scrO_{d+1}. \tag{$f_b+g=f_{<d}+f_{>d}$}
\end{align*}
Set $\upphi_d \colonequals \upvarphi_{12,d-2} \circ \upvartheta  \circ \upvarphi_{11,d-2} \circ \upvarphi_{22,d-2}$, we are done.
\end{proof}

\begin{prop}\label{prop:normal}
With notation in \textnormal{\ref{no7}}, there exists a right-equivalence,
\begin{equation*}
    \Phi \colon f \rightsquigarrow f_b + c
\end{equation*}
where $c$ is given by
\begin{equation*}
 c=\left\{
\begin{array}{cl}
 0 &\text{if } \det A_{12}(f) \neq 0 \\
\sum_{i=1}^{\infty} \upmu_i\x^{p+i} &\text{if } \det A_{12}(f) =0
\end{array}\right.
\end{equation*}
for some $\upmu_i\in \mathbb{C}$.
\end{prop}
\begin{proof}
We first apply the $\upphi_3$ in \ref{prop:norm},
\begin{equation*}
\upphi_3 \colon f \stackrel{2}{\rightsquigarrow}  \mathsf{f}_1 \colonequals f_{<3} +c_3+ \scrO_{4},
\end{equation*}
where the $c_3$ is the same as in \ref{prop:norm}. Then we continue to apply the $\upphi_4$ in \ref{prop:norm} to $\mathsf{f}_1$, 
\begin{equation*}
 \upphi_4 \colon \mathsf{f}_1 \stackrel{3}{\rightsquigarrow} \mathsf{f}_2 \colonequals (\mathsf{f}_1)_{<4}+c_4+\scrO_5= f_{<3} +\sum_{d=3}^4 c_d+  \scrO_5.
\end{equation*}
where the $c_4$ is the same as in \ref{prop:norm}. For $1 \leq s \leq 3$, repeating this process $s-2$ times gives
\begin{equation*}
    \upphi_s \circ \dots \circ \upphi_4 \circ \upphi_3 \colon f \rightsquigarrow \mathsf{f}_s \colonequals f_{<3}+\sum_{d=3}^s c_d+ \scrO_{s+1}.
\end{equation*}
Since $\upphi_d$ is a path degree $d-1$ right-equivalence for each $d \geq 3$ by \ref{prop:norm}, by \ref{31} $\Phi \colonequals \lim_{s \rightarrow \infty}\upphi_s \circ \dots \circ \upphi_4 \circ \upphi_3$ exists, and further
\begin{equation*}
\Phi  \colon f \rightsquigarrow f_{<3}+\sum_{d=3}^{\infty} c_d.
\end{equation*}
where each $c_d$ is the same as in \ref{prop:norm}.
By the grading in \ref{no6}(3), we have $f_{<3}=f_2=f_b$, hence
\[
\Phi \colon f \rightsquigarrow f_b+\sum_{d=3}^{\infty} c_d.
\]
Thus set $c \colonequals \sum_{d=3}^{\infty} c_d$, we are done.
\end{proof}

Proposition~\ref{prop:normal} shows that if $\det A_{12}(f)\neq 0$, then all terms of $f_r$ can be eliminated.
We therefore restrict to the case $\det A_{12}(f)=0$.
The following lemma is immediate from the definition of $A_{12}(f)$ in \ref{no6}(4).

\begin{lemma}\label{lemma:A12}
$\det A_{12}(f)=0$ if and only if $f_b=\upkappa_1\x^2+\x\y+\upkappa_2\y^2$ with $4\upkappa_1\upkappa_2 = 1$.
\end{lemma}

\begin{lemma}\label{6100}
With notation in \textnormal{\ref{no7}}, suppose that $f$ satisfies $\det A_{12}(f)=0$, and $f_r = \upmu \x^s+\scrO_{t}$ where $t>s\geq 3$ and $ \upmu \in \C^{\times}$. Then there exists a path degree $t-s+1$ right-equivalence $\uppsi_{t}$ such that
\begin{equation*}
    \uppsi_{t} \colon f \stackrel{t-s+1}{\rightsquigarrow} f_b + \upmu\x^s+\scrO_{t+1}.
\end{equation*}
\end{lemma}

\begin{proof}
By \ref{lemma:A12}, we have $f_b=\upkappa_1\x^2+\x\y+\upkappa_2\y^2$ with $4\upkappa_1\upkappa_2=1$.
If the degree $t$ part of $\scrO_t$ vanishes, there is nothing to prove.
Otherwise, applying \ref{042} repeatedly yields a right-equivalence
\begin{equation*}
 \upvartheta  \colon f \stackrel{t-1}{\rightsquigarrow}  \mathsf{f}_1 \colonequals f_b +\upmu \x^s +\upbeta \x^{t-1}\y+ \scrO_{t+1},
\end{equation*}
for some $\upbeta \in \C$. If $\upbeta=0$, we are done. Otherwise, we next apply $\upvarphi_{12,t-s}$ in \ref{lemma:norm12} which gives
\begin{align*}
    \upvarphi_{12,t-s}\colon \mathsf{f}_1 
 \mapsto & \upkappa_1(\x+\uplambda_1\x^{t-s}\y)^2+(\x+\uplambda_1\x^{t-s}\y)(\y+\uplambda_2\x^{t-s}\y)+\upkappa_2(\y+\uplambda_2\x^{t-s}\y)^2\\
& +\upmu(\x+\uplambda_1\x^{t-s}\y)^s+\upbeta\x^{t-1}\y+\scrO_{t+1}\\
  \stackrel{t-s+2}{\sim} & f_b+ \upmu\x^s+(2\upkappa_1\uplambda_1+\uplambda_2)\x^{t-s+1}\y+(\uplambda_1+2\upkappa_2\uplambda_2)\x^{t-s}\y^2+(s\upmu\uplambda_1+\upbeta)\x^{t-1}\y\\
& + (\upkappa_1\uplambda_1^2+\upkappa_2\uplambda_2^2+\uplambda_1\uplambda_2)\x^{2(t-s)}\y^2+\scrO_{t+1}.
\end{align*}
Note that in the above simplification we rewrite mixed monomials into the normal form $\x^i\y^j$ (up to higher degree terms) using \ref{412}, and then collect like terms.

Since $4\upkappa_1\upkappa_2=1$, set $\uplambda_2\colonequals -2\upkappa_1\uplambda_1$.
Then
\[
2\upkappa_1\uplambda_1+\uplambda_2=0,
\qquad
\uplambda_1+2\upkappa_2\uplambda_2
=\uplambda_1-4\upkappa_1\upkappa_2\uplambda_1=0,
\]
and moreover
\[
\upkappa_1\uplambda_1^2+\upkappa_2\uplambda_2^2+\uplambda_1\uplambda_2
=\upkappa_1\uplambda_1^2+4\upkappa_1^2\upkappa_2\uplambda_1^2-2\upkappa_1\uplambda_1^2
=\upkappa_1\uplambda_1^2(4\upkappa_1\upkappa_2-1)=0.
\]
Finally, choose $\uplambda_1$ so that $s\upmu\uplambda_1+\upbeta=0$.
This makes the coefficients of $\x^{t-s+1}\y$, $\x^{t-s}\y^2$, $\x^{t-1}\y$ and $\x^{2(t-s)}\y^2$ equal to zero in the above potential. Set $\uppsi_{t} \colonequals  \upvarphi_{12,t-s} \circ \upvartheta$, we are done.
\begin{comment}
Since $4\upkappa_1\upkappa_2 =1$, then $\frac{1}{2\upkappa_1}=2\upkappa_2$, and so any $\uplambda_1$ and $\uplambda_2$ with $\frac{\uplambda_1}{\uplambda_2}=-\frac{1}{2\upkappa_1}=-2\upkappa_2$ satisfies the following system of equations
\begin{equation*}
\begin{cases}
    2\upkappa_1\uplambda_1+\uplambda_2=0 \\
    \uplambda_1+2\upkappa_2\uplambda_2=0\\
    \upkappa_1\uplambda_1^2+\upkappa_2\uplambda_2^2+\uplambda_1\uplambda_2=0.
\end{cases}
\end{equation*}
We next choose $\uplambda_1$ to satisfy $s\upmu\uplambda_1+\upbeta=0$, and set $\uplambda_2 =-2\upkappa_1\uplambda_1$.
\end{comment}
\end{proof}

The following shows that when $\det A_{12}(f)=0$, the leading term of $f_r$ can eliminate all the other terms.
\begin{prop}\label{prop:normal2}
With notation in \textnormal{\ref{no7}}, suppose that $f$ satisfies $\det A_{12}(f)=0$. Then there exists a right-equivalence $\Psi$ such that
\begin{equation*}
    \Psi \colon f \rightsquigarrow f_b  \textnormal{ or } f_b+ \upmu\x^s,
\end{equation*}   
where $\upmu \in \C^{\times}$ and $s \geq 3$.
\end{prop}

\begin{proof}
Since $\det A_{12}(f)=0$, then by \ref{prop:normal}
\begin{equation*}
    \Phi \colon  f \rightsquigarrow \mathsf{f}_1 \colonequals f_b +\sum_{i=1}^{\infty} \upmu_i\x^{i+2}.
\end{equation*}
If all $\upmu_i=0$, then $ f \rightsquigarrow f_b$. 
Otherwise, let $s\ge 3$ be minimal such that the coefficient of $\x^s$ in $\mathsf{f}_1$ is nonzero,
and write this coefficient as $\upmu\in\C^\times$.
Applying \ref{6100} to $\mathsf{f}_1$, there exists a right-equivalence
\begin{equation*}
    \uppsi_{s+1} \colon \mathsf{f}_1 \stackrel{2}{\rightsquigarrow} \mathsf{f}_2 \colonequals f_b + \upmu\x^s+\scrO_{s+2}.
\end{equation*}
Thus, repeating this process $k$ times gives
\begin{equation*}
    \uppsi_{s+k} \circ \dots \circ \uppsi_{s+2} \circ \uppsi_{s+1} \colon  \mathsf{f}_1  \rightsquigarrow \mathsf{f}_{k+1} \colonequals f_b + \upmu\x^s +\scrO_{s+k+1}.
\end{equation*}
Since by \ref{6100} each $\uppsi_t$ is a degree $t-s+1$ right-equivalence for $t >s$, \ref{31} implies that
$\Psi' \colonequals \lim_{k \rightarrow \infty}\uppsi_{s+k} \circ \dots \circ \uppsi_{s+2} \circ \uppsi_{s+1}$ exists, and further
\begin{equation*}
    \Psi'  \colon  \mathsf{f}_1 \rightsquigarrow f_b + \upmu\x^s.
\end{equation*}
Set $\Psi \colonequals \Psi' \circ \Phi$, we are done.
\end{proof}

Combining \ref{prop:normal}, \ref{lemma:A12} and \ref{prop:normal2} gives the following result.
\begin{prop}\label{prop:A3norm}
Any Type A potential on $Q$ must be right-equivalent to one of the following potentials:
\begin{enumerate}
\item $\upkappa_1\x^2+\x\y+\upkappa_2\y^2$ where $\upkappa_1,\upkappa_2  \in \C^{\times}$ and $4\upkappa_1\upkappa_2 \neq 1$.\label{a1}
\item $\upkappa_1\x^2+\x\y+\upkappa_2\y^2+\upmu \x^s$ where $4\upkappa_1\upkappa_2 = 1$, $\upmu \in \C^{\times}$ and $s \geq 3$.\label{a2}
\item $\upkappa_1\x^p+\x\y+\upkappa_2\y^q$ where $(p,q)\neq (2,2)$ and $\upkappa_1,\upkappa_2 \in \C^{\times}$.\label{a3}
\item $\upkappa_1\x^2+\x\y+\upkappa_2\y^2$ where $4\upkappa_1\upkappa_2 = 1$.\label{a4}
\item $\upkappa_1\x^p+\x\y$ where $p \geq 2$ and $\upkappa_1 \in\C^{\times}$.\label{a5}
\item $\x\y+\upkappa_2\y^q$ where $q \geq 2$ and $\upkappa_2 \in \C^{\times}$.\label{a6}
\item $\x\y$.\label{a7}
\end{enumerate}
\end{prop}
\begin{proof}
Recall in \ref{no6} and \ref{no7}, any Type A potential on $Q$ has the form of $f=f_b+f_r$ where $f_b= \upkappa_1\x^p+\x\y+\upkappa_2\y^q$.

When $\det A_{12}(f)=0$, namely $p=q=2$ and $4\upkappa_1\upkappa_2 = 1$ by \ref{lemma:A12}, then by \ref{prop:normal} $f \cong f_b  \textnormal{ or } f_b+ \upmu\x^s$ where $\upmu \in \C^{\times}$ and $s \geq 3$. This gives exactly the cases (4) and (2).

When $\det A_{12}(f) \neq 0$, by \ref{prop:normal} $f\cong f_b$. Again by \ref{lemma:A12}, we have $(p,q) \neq (2,2)$ or $4\upkappa_1\upkappa_2 \neq 1$, so $f$ must belong to one of the following cases.
\begin{enumerate}[label=\alph*)]
\item $\upkappa_1 \neq 0$ and $\upkappa_2 =0 $.
\item $\upkappa_1 = 0$ and $\upkappa_2  \neq 0$.
\item $\upkappa_1  = 0$ and $\upkappa_2 = 0$.
\item $\upkappa_1  ,\upkappa_2 \neq 0$, $4\upkappa_1\upkappa_2 \neq 1$ and $p=q=2$.
\item $\upkappa_1 , \upkappa_2 \neq 0$ and $(p,q) \neq (2,2)$.
\end{enumerate}
The a), b), c), d) and e) are items (5), (6), (7), (1) and (3) in the statement.
\end{proof}

Then we continue to normalize the coefficients of the potentials in \ref{prop:A3norm}.
\begin{cor}\label{cor:A3norm}
Any Type A potential on $Q$ must be isomorphic to one of the following potentials:   
\begin{enumerate}
    \item $\x^2+\x\y+\uplambda\y^2$ where $\uplambda \in \C \setminus \{0,\tfrac14\}$. \label{b1}
    \item $\x^2+\x\y+\frac{1}{4}\y^2+ \x^s$ where $s \geq 3$.\label{b2}
     \item $\x^p+\x\y+\y^q $ where $(p,q) \neq (2,2)$.\label{b3}
     \item $\x^2+\x\y+\frac{1}{4}\y^2$.\label{b4}
     \item $\x^p+\x\y$ where $p \geq 2$.\label{b5}
     \item $\x\y+\y^q$ where $q \geq 2$.\label{b6}
     \item $\x\y$.\label{b7}
\end{enumerate}
\end{cor}

\begin{proof}
(1) Applying $a_1 \mapsto \uplambda_1a_1,\  a_2 \mapsto \uplambda_2a_2$ where $\uplambda_1, \uplambda_2 \in \C$ to \eqref{a1} gives
\begin{equation*}
    \upkappa_1\x^2+\x\y+\upkappa_2\y^2 \mapsto \uplambda_1^2\upkappa_1\x^2+\uplambda_1\uplambda_2\x\y+\uplambda_2^2\upkappa_2\y^2.
\end{equation*}
Since $\upkappa_1\neq 0$, we can solve for $(\uplambda_1,\uplambda_2)$ such that
$\uplambda_1^2\upkappa_1=\uplambda_1\uplambda_2=1$. Moreover, since $\upkappa_2 \neq 0$ and $4\upkappa_1\upkappa_2 \neq 1$, $\uplambda_2^2\upkappa_2 =\upkappa_1\upkappa_2\neq 0,\frac{1}{4}$. Set $\uplambda \colonequals \uplambda_2^2\upkappa_2$.
Thus $\upkappa_1\x^2+\x\y+\upkappa_2\y^2 \mapsto \x^2+\x\y+\uplambda\y^2$ where $  \uplambda \neq 0,\frac{1}{4}$.

(2) Applying $\varphi \colon a_1 \mapsto \uplambda_1a_1, \ a_2 \mapsto \uplambda_2a_2$ where $\uplambda_1, \uplambda_2 \in \C$ to \eqref{a2} gives
\begin{equation*}
\upkappa_1\x^2+\x\y+\upkappa_2\y^2+\upmu \x^s \mapsto \uplambda_1^2\upkappa_1\x^2+\uplambda_1\uplambda_2\x\y+\uplambda_2^2\upkappa_2\y^2 + \uplambda_1^{s}\upmu\x^{s}.
\end{equation*}
We next claim that we can find some $(\uplambda_1, \uplambda_2)$ which satisfies
\begin{equation*}
    \uplambda_1^2\upkappa_1 : \uplambda_1\uplambda_2 : \uplambda_2^2\upkappa_2: \uplambda_1^{s}\upmu =1:1: \frac{1}{4}:1.
\end{equation*}
Once the claim is certified, it follows at once that
$\upkappa_1\x^2+\x\y+\upkappa_2\y^2+\upmu \x^s \cong \x^2+\x\y+\frac{1}{4}\y^2+ \x^s$.

To prove the claim, since $s \geq 3$ and $\upmu \neq 0$, we can solve for $\uplambda_1$ such that $\uplambda_1^2\upkappa_1 : \uplambda_1^{s}\upmu = 1:1$ holds. Then we solve $\uplambda_2$ from $\uplambda_1$ and $\uplambda_1^2\upkappa_1 : \uplambda_1\uplambda_2=1:1$. Moreover, this choice of $\uplambda_1$ and $\uplambda_2$ also satisfies $  \uplambda_1\uplambda_2: \uplambda_2^2\upkappa_2 = 1: \frac{1}{4}$ since $4\upkappa_1\upkappa_2 =1$. Combining these together, $(\uplambda_1, \uplambda_2)$ satisfies the claim.

(3) Applying $\varphi \colon a_1 \mapsto \uplambda_1a_1,\  a_2 \mapsto \uplambda_2a_2$ where $\uplambda_1, \uplambda_2 \in \C$ to \eqref{a3} gives
\begin{equation*}
    \upkappa_1\x^p+\x\y+\upkappa_2\y^q \mapsto \uplambda_1^p\upkappa_1\x^p+\uplambda_1\uplambda_2\x\y+\uplambda_2^q\upkappa_2\y^q.
\end{equation*}
Similar to (2), the statement follows once we find some $(\uplambda_1, \uplambda_2)$ which satisfies
\begin{equation*}
\uplambda_1^p\upkappa_1:\uplambda_1\uplambda_2:\uplambda_2^q\upkappa_2=1:1:1.
\end{equation*}

The above equations induce $\uplambda_2=\uplambda_1^{p-1}\upkappa_1$ and $\uplambda_1=\uplambda_2^{q-1}\upkappa_2$, and so $\uplambda_1^{(p-1)(q-1)-1}\upkappa_1^{q-1}\upkappa_2=1$. Since $\upkappa_1$, $\upkappa_2 \neq 0$ and $(p,q) \neq (2,2)$, we can solve $\uplambda_1$, and then $\uplambda_2$ such that the above equations hold.

(4) Applying $a_1 \mapsto \uplambda_1a_1,\ a_2 \mapsto \uplambda_2a_2$ where $\uplambda_1, \uplambda_2 \in \C$ to \eqref{a4} gives
\begin{equation*}
    \upkappa_1\x^2+\x\y+\upkappa_2\y^2 \mapsto \uplambda_1^2\upkappa_1\x^2+\uplambda_1\uplambda_2\x\y+\uplambda_2^2\upkappa_2\y^2.
\end{equation*}
Similar to (1), we can solve for $(\uplambda_1, \uplambda_2)$ such that 
$\uplambda_1^2\upkappa_1=\uplambda_1\uplambda_2=1$ holds, and then $\uplambda_2^2\upkappa_2 =\upkappa_1\upkappa_2=\frac{1}{4}$ since $4\upkappa_1\upkappa_2 = 1$. Thus $\upkappa_1\x^2+\x\y+\upkappa_2\y^2 \mapsto \x^2+\x\y+\frac{1}{4}\y^2$.

(5) Applying $a_1 \mapsto \uplambda_1a_1, \ a_2 \mapsto \uplambda_2a_2$ where $\uplambda_1, \uplambda_2 \in \C$ to \eqref{a5} gives
\begin{equation*}
    \upkappa_1\x^p+\x\y \mapsto  \uplambda_1^p\upkappa_1\x^p+\uplambda_1\uplambda_2\x\y.
\end{equation*}
Since $\upkappa_1 \neq 0$, we can solve for $(\uplambda_1,\uplambda_2)$ such that
$\uplambda_1^p\upkappa_1=\uplambda_1\uplambda_2=1$ holds. Thus $ \upkappa_1\x^p+\x\y \mapsto  \x^p+\x\y$.

The proof of (6) and (7) is similar to (5).
\end{proof}

We now simplify the previous geometric realization in $\S \ref{GR}$ for the potentials in \ref{cor:A3norm}.
\begin{prop}\label{715}
Each Jacobi algebra of potentials in \textnormal{\ref{cor:A3norm}} is realized by a crepant resolution of a $cA_3$ singularity $\scrR \colonequals \C \lal u, v, x, y \ral/(uv-h_{0}h_{1}h_2h_3)$, which corresponds to the $\scrR$-module $M \colonequals \scrR \oplus (u,h_{0}) \oplus (u,h_{0}h_{1}) \oplus (u,h_{0}h_{1}h_2)$ in \textnormal{\ref{36}} as follows.
\[
\begin{tabular}{ p{0.5cm}p{2cm}p{0.5cm}p{0.5cm}p{1.5cm}  }
\hline
 & $h_0$ & $h_1$ &  $h_2$ &  $h_3$\\
\hline
$(1)$ & $2x+y$ & $x$ & $y$ & $x+2\uplambda y$\\
$(2)$ & $2x+y+sx^{s-1}$  & $x$ & $y$ & $x+\frac{1}{2} y$\\
$(3)$ & $px^{p-1}+y$  & $x$ & $y$ & $x+ qy^{q-1}$\\
$(4)$ & $2x+y$  & $x$ & $y$ & $x+\frac{1}{2} y$\\
$(5)$ & $px^{p-1}+y$  & $x$ & $y$ & $x$\\
$(6)$ & $y$  & $x$ & $y$ & $x+ qy^{q-1}$\\
$(7)$ & $y$  & $x$ & $y$ & $x$\\
\hline
\end{tabular}
\]
\end{prop}
\begin{proof}
In order to construct the geometric realization by \ref{gr} and \eqref{501}, we first transform the potentials in \ref{cor:A3norm} to some potentials in $Q_3$, which has a single loop at each vertex, as illustrated below (see also \ref{no1}). 
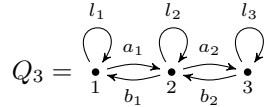
\[
\begin{tikzpicture}[bend angle=15, looseness=1.2]
\node (a) at (-1,0) [vertex] {};
\node (b) at (0,0) [vertex] {};
\node (c) at (1,0) [vertex] {};

\node (a1) at (-1,-0.2) {$\scriptstyle 1$};
\node (a2) at (0,-0.2) {$\scriptstyle 2$};
\node (a3) at (1,-0.2) {$\scriptstyle 3$};

\draw[->,bend left] (a) to node[above] {$\scriptstyle a_{1}$} (b);
\draw[<-,bend right] (a) to node[below] {$\scriptstyle b_{1}$} (b);
\draw[->,bend left] (b) to node[above] {$\scriptstyle a_{2}$} (c);
\draw[<-,bend right] (b) to node[below] {$\scriptstyle b_{2}$} (c);
\draw[<-]  (a) edge [in=120,out=55,loop,looseness=11] node[above] {$\scriptstyle l_{1}$} (a);
\draw[<-]  (b) edge [in=120,out=55,loop,looseness=11] node[above] {$\scriptstyle l_{2}$} (b);
\draw[<-]  (c) edge [in=120,out=55,loop,looseness=11] node[above] {$\scriptstyle l_{3}$} (c);
\node (z) at (-1.7,0) {$Q_{3} =$};
\end{tikzpicture}
\]

Consider a potential $f=\upkappa_1\x^p+\x\y+\upkappa_2\y^q+\upkappa_3\x^s$ on $Q$ where $\upkappa_1, \upkappa_2, \upkappa_3 \in \C$. By applying \ref{049} three times, each of which adds a loop $l_i$ at vertex $i$ of $Q$ for $1 \leq i \leq 3$, we have $\Jac(Q,f) \cong \Jac(Q_3,f')$ where
\begin{equation*}
    f' = l_1\x'+ \x l_2+l_2\y+\y'l_3-\frac{1}{2}l_1^2-\frac{1}{2}l_2^2-\frac{1}{2}l_3^2
    -\x^2-\y^2+\upkappa_1\x^p+\upkappa_2\y^q+\upkappa_3\x^s.
\end{equation*}
Then by \ref{gr}, \ref{054} and \eqref{501}, we can realize $f'$ by setting $g_2=x$, $g_3=x+y$ and then solving the following system of equations where each $g_i \in \C\lal x,y \ral$
\begin{align*}
    & g_0-g_1+g_2=0\\
    & g_1-2g_2+\upkappa_1pg_2^{p-1}+\upkappa_3sg_2^{s-1}+g_3=0\\
    & g_2-g_3+g_4=0\\
    & g_3-g_4+\upkappa_2g_4^{q-1}+g_5=0\\
    & g_4-g_5+g_6=0.
\end{align*}
Thus $(g_0,g_1,g_2,g_3,g_4,g_5,g_6)= (-\upkappa_1px^{p-1}-\upkappa_3sx^{s-1}-y, x-\upkappa_1px^{p-1}-\upkappa_3sx^{s-1}-y,x,x+y,y,-x+y-\upkappa_2qy^{q-1},-x-\upkappa_2qy^{q-1})$. Set $(h_0,h_1,h_2,h_3) \colonequals (-g_0,g_2,g_4,-g_6)$ and
consider 
\begin{equation*}
    \scrR \colonequals \frac{\mathbb{C} \lal u, v, x, y \ral}{uv-h_{0}h_{1}h_2h_3}=\frac{\mathbb{C} \lal u, v, x, y \ral}{uv-(\upkappa_1px^{p-1}+\upkappa_3sx^{s-1}+y)xy(x+\upkappa_2qy^{q-1})}
\end{equation*}
and $\scrR$-module $M = \scrR \oplus (u,h_{0}) \oplus (u,h_{0}h_{1})\oplus (u,h_{0}h_{1}h_2)$. Write $\uppi$ for the crepant resolution of $\Spec \scrR$, which corresponds to $M$ in \ref{36}. Then $\Lambda_{\mathrm{con}}(\uppi) \cong \Jac(Q_3,f')$ by \ref{gr}, and so $\Lambda_{\mathrm{con}}(\uppi) \cong \Jac(Q,f)$. 
Specialising $(\upkappa_1,\upkappa_2,\upkappa_3,p,q,s)$ to match the seven families in \ref{cor:A3norm}
yields the table, and hence the claim.
\end{proof}

We now classify Type~A potentials on $Q$ up to isomorphism; this is the main result of the section.
\begin{lemma}\label{616}
Let $\uplambda_1, \uplambda_2  \in \C$ with $\uplambda_1 \neq \uplambda_2$. Then we have $\Jac(\x^2+\x\y+\uplambda_1\y^2)\not\cong \Jac(\x^2+\x\y+\uplambda_2\y^2)$.
\end{lemma}
\begin{proof}
Write $f_i \colonequals \x^2+\x\y+\uplambda_i \y^2$ for $i=1,2$.
We argue by contradiction, assuming that there exists an algebra isomorphism 
$\phi\colon \Jac(f_1)\xrightarrow{\sim}\Jac(f_2)$.

By \ref{715}, each $f_i$ is realised by a crepant resolution $\uppi_i$ of a $cA_3$ singularity with
$\Lambda_{\mathrm{con}}(\uppi_i)\cong \Jac(f_i)$.
Hence \ref{37} implies that $\phi$ either fixes each idempotent $e_i$, or sends $e_i$ to $e_{4-i}$ for $1\leq i\leq 3$. We consider only the first case; the second is similar. 
Recall that the doubled $A_3$ quiver $Q$ on which $f_i$ is defined:
\[
\begin{array}{cl}
\begin{array}{c}
\begin{tikzpicture}[bend angle=30, looseness=0.9]
\node (a) at (-2,0) [vertex] {};
\node (b) at (0,0) [vertex] {};
\node (c) at (2,0) [vertex] {};
\node (a1) at (-2,-0.2) {$\scriptstyle 1$};
\node (a1) at (0,-0.2) {$\scriptstyle 2$};
\node (a1) at (2,-0.2) {$\scriptstyle 3$};
\draw[<-,bend right] (a) to node[below] {$\scriptstyle b_1$} (b);
\draw[->,bend left] (a) to node[above] {$\scriptstyle a_1$} (b);
\draw[<-,bend right] (b) to node[below] {$\scriptstyle b_2 $} (c);
\draw[->,bend left] (b) to node[above] {$\scriptstyle a_2$} (c);
%\node (z) at (0,-1.2) {Quiver $Q_{3,\{1,2,3\}}$};
\end{tikzpicture}
\end{array}
&
\begin{array}{l}
\x=b_1a_1\\
\y=a_2b_2
\end{array}
\end{array}
\]
Write $\m \colonequals \lcl a_1,b_1,a_2,b_2 \rcl$ for the arrow ideal.
Since $\phi(e_i)=e_i$ for $1\leq i\leq 3$, $\phi$ preserves sources and targets, so
\[
\phi(a_1)\in e_1\mathfrak{m}e_2,\quad \phi(b_1)\in e_2\mathfrak{m}e_1,\quad
\phi(a_2)\in e_2\mathfrak{m}e_3,\quad \phi(b_2)\in e_3\mathfrak{m}e_2.
\]
Moreover, the induced map on $\mathfrak{m}/\mathfrak{m}^2$ is invertible, hence
\[
\phi \colon a_1 \mapsto c_1a_1 + r_{1}, \quad b_1 \mapsto c_2b_1 + r_{2}, \quad a_2 \mapsto c_3a_2 + r_{3}, \quad b_2 \mapsto c_4b_2 + r_{4},
\]
where each $c_i\in\C^\times$ and $r_{i}\in \m^2$. 
To eliminate the higher-order terms $r_i$, we pass to the quotient modulo $\m^4$.
Since $\phi(\m)=\m$, it induces an isomorphism $\Bar{\phi}  \colon \Jac(f_1)/ \m^4\xrightarrow{\sim}\Jac(f_2)/ \m^4$. Since for each arrow $a$ the relation $\partial_a f_1$ is homogeneous of arrow-length $3$,
any term in $\phi(\partial_a f_1)$ involving at least one of the $r_i$
has arrow-length at least $4$, and hence lies in $\m^4$. It follows that $\Bar{\phi}$ depends only on the linear part of $\phi$, that is
\[
\Bar{\phi}  \colon a_1 \mapsto c_1a_1,\quad b_1 \mapsto c_2b_1,\quad a_2 \mapsto c_3a_2,\quad b_2 \mapsto c_4b_2.
\]
By a slight abuse of notation, we regard $\Bar{\phi}$ as a homomorphism on $\C \lbl Q \rbl \xrightarrow{\sim} \C \lbl Q \rbl$ and write
\[
\Bar{\phi}(J(f_1)) \colonequals \lcl \Bar{\phi}(\partial_{a_1}f_1),\, \Bar{\phi}(\partial_{b_1}f_1),\, \Bar{\phi}(\partial_{a_2}f_1),\, \Bar{\phi}(\partial_{b_2}f_1) \rcl
\]
for the ideal of $\C \lbl Q \rbl$ generated by these elements. 
Since $\Bar{\phi}$ induces an isomorphism
$\Jac(f_1)/ \m^4 \cong \Jac(f_2)/ \m^4$,
it follows that the ideals $\Bar{\phi}(J(f_1))$ and $J(f_2)$
generate the same ideal in $\C\lbl Q\rbl/\m^4$.
Therefore,
\[
\lcl \Bar{\phi}(J(f_1)),\, J(f_2),\, \m^4 \rcl
=
\lcl J(f_2),\, \m^4 \rcl .
\]

The Jacobi ideal $J(f_1)$ is generated by:
\begin{align*}
    \partial_{a_1}f_1 & = 2b_1a_1b_1+a_2b_2b_1 =2\x b_1+\y b_1,\\
   \partial_{b_1}f_1 & = 2a_1b_1a_1+a_1a_2b_2=2a_1\x+a_1\y, \\
   \partial_{a_2}f_1 & = b_2b_1a_1+2 \uplambda_1 b_2a_2b_2=b_2\x+2\uplambda_1 b_2\y,\\
  \partial_{b_2}f_1 & =  b_1a_1a_2+2\uplambda_1 a_2b_2a_2= \x a_2 +2\uplambda_1 \y a_2.
\end{align*}

Thus $\Bar{\phi}(\partial_{a_1}f_1)  =2c_1c_2^2b_1a_1b_1 +c_2c_3c_4 a_2b_2b_1 \propto 2\x b_1+ \uplambda' \y b_1$ where $\uplambda' \colonequals c_3c_4/(c_1c_2)\neq 0$, and $\propto$ denotes equality up to multiplication by a non-zero scalar.
Similarly,
\begin{align*}
  \Bar{\phi}(\partial_{b_1}f_1) \propto  2a_1\x +\uplambda' a_1\y, \quad
  \Bar{\phi}(\partial_{a_2}f_1)  \propto b_2\x +2  \uplambda'\uplambda_1 b_2\y, \quad
 \Bar{\phi}(\partial_{b_2}f_1)  \propto  \x a_2+2 \uplambda'\uplambda_1 \y a_2.
\end{align*}

Compare the generators of $\Bar{\phi}(J(f_1))$ (left column) and $J(f_2)$ (right column):

\begin{minipage}[t]{0.48\textwidth}
\begin{align}
&  2\x b_1+ \uplambda' \y b_1, \label{eq:L1}\\
&2a_1\x +\uplambda' a_1\y, \label{eq:L2}\\
& b_2\x +2  \uplambda'\uplambda_1 b_2\y, \label{eq:L3}\\
& \x a_2+2 \uplambda'\uplambda_1 \y a_2. \label{eq:L4}
\end{align}
\end{minipage}\hfill
\begin{minipage}[t]{0.48\textwidth}
\begin{align}
&2\x b_1+\y b_1, \label{eq:R1}\\
&2a_1\x+a_1\y, \label{eq:R2}\\
&b_2\x+2\uplambda_2 b_2\y, \label{eq:R3}\\
& \x a_2 +2\uplambda_2 \y a_2. \label{eq:R4}
\end{align}
\end{minipage}

Recall that $\uplambda' \neq 0$ and $\uplambda_1 \neq \uplambda_2$.
We split the proof into three cases.

(1) $\uplambda' \neq 1$.

Our strategy is to show that the four paths 
\[
\x b_1,\ \y b_1,\ a_1\x,\ a_1\y
\] 
lie in the ideal $\lcl \Bar{\phi}(J(f_1)), J(f_2), \m^4\rcl$, while none of them lies in $\lcl J(f_2), \m^4\rcl$. This yields a contradiction, since $\lcl \Bar{\phi}(J(f_1)),J(f_2) ,\m^4\rcl=\lcl J(f_2),\m^4\rcl$.

Since $\m^4$ has arrow-length at least $4$, any arrow-length $3$ element of $\lcl J(f_2),\m^4 \rcl$ must lie in the $\C$-span of \eqref{eq:R1}--\eqref{eq:R4}.
But none of $\x b_1,\ \y b_1,\ a_1\x,\ a_1\y$ lies in this span.
Hence none of them lies in $\lcl J(f_2),\m^4 \rcl$.

Since $\uplambda' \neq 1$, comparing \eqref{eq:L1} with \eqref{eq:R1} yields $\x b_1,\y b_1 \in  \lcl \Bar{\phi}(J(f_1)),J(f_2)\rcl$, and similarly, comparing \eqref{eq:L2} with \eqref{eq:R2} yields $a_1\x,  a_1\y \in  \lcl \Bar{\phi}(J(f_1)),J(f_2)\rcl$. Consequently,
\[
\x b_1,\y b_1,a_1\x,a_1\y \in  \lcl \Bar{\phi}(J(f_1)),J(f_2),\m^4\rcl,
\]
which contradicts the fact that none of these elements lies in $\lcl J(f_2),\m^4 \rcl$.

(2) $\uplambda' =1$ and $\uplambda_2 \neq 0$.

In this case the argument is analogous, but uses the relations involving $b_2$ and $a_2$.
Since $\m^4$ has arrow-length at least $4$, the arrow-length $3$ component of $\lcl J(f_2),\m^4 \rcl$ coincides with the $\C$-span of the generators
\eqref{eq:R1}--\eqref{eq:R4}. As $\uplambda_2 \neq 0$, none of the paths
\[
b_2\x,\ b_2\y,\ \x a_2,\ \y a_2
\]
lies in this span. Hence none of them lies in $\lcl J(f_2),\m^4 \rcl$.

Since $\uplambda' = 1$, the comparisons used in case~(1) are no longer available. However, since $\uplambda_1 \neq \uplambda_2$ by assumption, comparing \eqref{eq:L3} with \eqref{eq:R3}
and \eqref{eq:L4} with \eqref{eq:R4} yields 
\[
b_2\x,b_2\y,\x a_2,\y a_2 \in  \lcl \Bar{\phi}(J(f_1)),J(f_2)\rcl \subset \lcl \Bar{\phi}(J(f_1)),J(f_2),\m^4 \rcl,
\]
which contradicts the fact that none of these elements lies in $\lcl J(f_2),\m^4 \rcl$.

(3) $\uplambda_2 =0$.

This case differs from case~(2) only in that when $\uplambda_2 = 0$,
the relations \eqref{eq:R3} and \eqref{eq:R4} imply
$b_2\x,\ \x a_2 \in J(f_2)$.
However, since $\m^4$ has arrow-length at least $4$,
the arrow-length $3$ component of $\lcl J(f_2),\m^4 \rcl$
is again the $\C$-span of \eqref{eq:R1}--\eqref{eq:R4}.
But the paths $b_2\y, \y a_2$ do not lie in this span, and hence do not lie in $\lcl J(f_2),\m^4 \rcl$.

Since $\uplambda_2=0$ and $\uplambda_1\neq\uplambda_2$, we have $\uplambda_1\neq 0$.
Thus comparing \eqref{eq:L3} with \eqref{eq:R3} and \eqref{eq:L4} with \eqref{eq:R4} gives $ b_2\y,\y a_2 \in  \lcl \Bar{\phi}(J(f_1)),J(f_2)\rcl \subset \lcl \Bar{\phi}(J(f_1)),J(f_2),\m^4 \rcl$, again a contradiction.
\end{proof}

\begin{theorem}\label{610}
Any Type A potential on $Q$ must be isomorphic to one of the following isomorphism classes of potentials:
\begin{enumerate}
    \item $\x^2+\x\y+\uplambda\y^2$ for any $\uplambda \in \mathbb{C} \setminus \{0,\tfrac14\}$.\label{c1}
     \item $\x^2+\x\y+\frac{1}{4}\y^2+ \x^s$ for any $s \geq 3$.\label{c2}
    \item $\x^p+\x\y+\y^q \cong \x^q+\x\y+\y^p$ for any $(p,q) \neq (2,2)$.\label{c3}
        \item $\x^2+\x\y+\frac{1}{4}\y^2$.\label{c4}
        \item $\x^p+\x\y \cong \x\y+\y^p$ for any $p \geq 2$.\label{c5}
        \item $\x\y$.\label{c6}
\end{enumerate}
The Jacobi algebras of these potentials are all mutually non-isomorphic (except for those isomorphisms stated), and in particular, the Jacobi algebras with different parameters in the same item are non-isomorphic. 

The Jacobi algebras in \textnormal{(1)}, \textnormal{(2)}, \textnormal{(3)} are finite-dimensional and realized by crepant resolutions of isolated $cA_3$ singularities, and those in \textnormal{(4)}, \textnormal{(5)}, \textnormal{(6)} are infinite-dimensional and realized by crepant resolutions of non-isolated $cA_3$ singularities.
\end{theorem}

\begin{proof}
We first prove the isomorphisms in the statement. Applying $a_1 \mapsto b_2, b_1 \mapsto a_2, a_2 \mapsto b_1,  b_2 \mapsto a_1$ gives
\begin{align*}
    \x^p+\x\y+\y^q  \rightsquigarrow   \x^q+\x\y+\y^p, \
    \x^p+\x\y \rightsquigarrow  \x\y+\y^p.
\end{align*}

Then we prove the non-isomorphisms in the statement by using the following fact. If Type A potentials $f$ and $g$ on $Q$ are isomorphic, then $\dim_{\C}\Jac(f)=\dim_{\C}\Jac(g)$, and further by \ref{37} there is an equality of sets 
\begin{equation}\label{601}
\{\operatorname{dim}_{\C}\Jac(f)/e_1, \ \operatorname{dim}_{\C}\Jac(f)/e_3 \} = \{\operatorname{dim}_{\C}\Jac(g)/e_1, \ \operatorname{dim}_{\C}\Jac(g)/e_3 \}.
\end{equation}
The following table lists $\dim_{\C}\Jac(f)$, $\dim_{\C}\Jac(f)/e_1$ and $\dim_{\C}\Jac(f)/e_3$ for each $f$ in each item, using Toda's formula (see \cite[\S 4.4]{T2}).
\[
\begin{tabular}{ p{2cm}p{2cm}p{2.5cm}p{2.5cm}  }
\hline
 & $\dim_{\C}\Jac(f)$ & $\dim_{\C}\Jac(f)/e_1$ &  $\dim_{\C}\Jac(f)/e_3$\\
\hline
(1) & $20$ & $6$ & $6$\\
(2) & $9s+2$  & $6$ & $6$ \\
$\x^p+\x\y+\y^q$ & $4p+4q+4$ & $4q-2$ & $4p-2$\\
(4) & $\infty$ & 6 & 6\\
$\x^p+\x\y$ & $\infty$ & $\infty$ & $4p-2$\\
(6) & $\infty$ & $\infty$ & $\infty$\\
\hline
\end{tabular}
\]

Now, all Jacobi algebras in (1) have dimension $20$, but are mutually non-isomorphic by \ref{616}.
All Jacobi algebras in (2) are mutually non-isomorphic since they all have different dimensions. 

For (3), we only need to prove that $\x^p+\x\y+\y^q \not\cong \x^r+\x\y+\y^s$ for any $(p,q) \neq (r,s)$ and $(p,q) \neq (s,r)$. From the above table,
\begin{align*}
  &  \{\operatorname{dim}_{\C}\Jac(\x^p+\x\y+\y^q)/e_1, \ \operatorname{dim}_{\C}\Jac(\x^p+\x\y+\y^q)/e_3 \} = \{4q-2,4p-2\},\\
&  \{\operatorname{dim}_{\C}\Jac(\x^r+\x\y+\y^s)/e_1, \ \operatorname{dim}_{\C}\Jac(\x^r+\x\y+\y^s)/e_3 \} = \{4r-2,4s-2\}.
\end{align*}
Since $(p,q) \neq (r,s)$ and $(p,q) \neq (s,r)$, then the above two sets are not equal, and so $\x^p+\x\y+\y^q \not\cong \x^r+\x\y+\y^s$ by \eqref{601}.

For (5), since $\x^p+\x\y \cong \x\y+\y^p$, we only need to prove that $\x^p+\x\y \not\cong \x^q+\x\y$ for any $p \neq q$. From the above table,
\begin{align*}
  &  \{\operatorname{dim}_{\C}\Jac(\x^p+\x\y)/e_1, \ \operatorname{dim}_{\C}\Jac(\x^p+\x\y)/e_3 \} = \{\infty,4p-2\},\\
&  \{\operatorname{dim}_{\C}\Jac(\x^q+\x\y)/e_1, \ \operatorname{dim}_{\C}\Jac(\x^q+\x\y)/e_3 \} = \{\infty,4q-2\}.
\end{align*}
Since $p \neq q$, the above two sets are not equal, so $\x^p+\x\y \not\cong \x^q+\x\y$ by \eqref{601}.

The above shows that potentials in the same item are mutually non-isomorphic.
We finally prove that the potentials in different items are mutually non-isomorphic. 

Since Jacobi algebras in (1), (2) and (3) have finite dimension, while those in (4), (5) and (6) have infinite dimension, we only need to prove that the potentials in (1), (2) and (3) are mutually non-isomorphic, and the potentials in (4), (5) and (6) are mutually non-isomorphic, respectively.

From the above table, the Jacobi algebras in (1), (2), and (3) have dimensions $20$, $9s+2$ and $4p+4q+4$, respectively. Since $s \geq 3$ and $(p,q) \neq (2,2)$, then $9s+2 > 20$ and $4p+4q+4 > 20$, and so the potentials in (1) are not isomorphic to those in (2) and (3).
To compare the potentials in (2) and (3), since $(p,q) \neq (2,2)$, then $\{4q-2,4p-2\} \neq \{6,6\}$, then the potentials in (2) are not isomorphic to those in (3), by \eqref{601} and the table.

To compare the potentials in (4), (5) and (6), since $\{6,6\}$, $\{\infty,4p-2\}$ and $\{\infty,\infty\}$ are mutually not equal, then the potentials in (4), (5) and (6) are mutually non-isomorphic by \eqref{601} and the table.   

By the geometric realizations in \ref{715}, the Jacobi algebras in \textnormal{(1)}, \textnormal{(2)}, \textnormal{(3)} are realized by crepant resolutions of isolated $cA_3$ singularities, and those in \textnormal{(4)}, \textnormal{(5)}, \textnormal{(6)} are realized by crepant resolutions of non-isolated $cA_3$ singularities.
\end{proof}

\begin{remark}\label{6111}
In \ref{610}, the class \textnormal{(4)} can be viewed as a limit of \textnormal{(2)} as $s \to \infty$
or of \textnormal{(1)} as $\uplambda \to \frac{1}{4}$.
Similarly, the classes \textnormal{(5)} and \textnormal{(6)} arise as limits of \textnormal{(3)}
as $p \to \infty$ and $q \to \infty$, respectively.
This parallels the general phenomenon that divisor-to-curve contractions often occur as limits of flops; see also \cite{BW2}.
In the proof of \ref{610}, we establish separately the finiteness of the dimension of the Jacobi algebra and the realisation as a crepant resolution of an isolated $cA_3$ singularity.
However, by \ref{352}, either of these conditions implies the other.
\end{remark}

\begin{remark}
In this section, for a Type~A potential $f$ on the doubled $A_{3}$ quiver without loops, 
we normalise $f$ using the matrix $A_{12}(f)$ (see \ref{no6}(4)), 
which also appears in \cite[§4.2]{Z}. 
For Type~A potentials on the doubled $A_{3}$ quiver with loops, 
or more generally on the doubled $A_{n}$ quiver $Q_n$ with $n \geq 4$, 
one would instead need to employ matrices of the form $A_{ij}(f)$ with $j-i \geq 2$ in order to carry out an analogous normalisation procedure. 
At present, it is not clear how to extend the normalisation method developed here to Type~A potentials on $Q_n$ for arbitrary $n$.
\end{remark}

\subsection{Derived equivalence classes}
The purpose of this subsection is to prove \ref{611}, which describes the derived equivalence classes of Type~A potentials on $Q$ whose Jacobi algebras are finite-dimensional.
Throughout this subsection we restrict to this finite-dimensional case, since \ref{thm:august} applies only to isolated singularities.

Given a Type A potential $f$ on $Q$, by \ref{cor:A3norm} and \ref{715} we can realize $f$ by a $cA_3$ singularity
\begin{equation*}
    \scrR \cong \frac{\mathbb{C} \lal u, v, x, y \ral}{uv-h_{0}h_{1}h_2h_3}
\end{equation*}
together with the $\scrR$-module $M=  \scrR \oplus (u,h_0) \oplus (u,h_0h_1) \oplus  (u,h_0h_1h_2)$. Let $\uppi\colon \scrX\to \Spec\scrR$ be the corresponding crepant resolution. Then $\Lambda_{\mathrm{con}}(\uppi) \cong \underline{\End}_{\scrR}(M) \cong \Jac(f)$. 

\begin{notation}\label{notation:derived potential}
We adopt the following notation. We first recall $\uppi^i$, $\scrX^i$, $M^i$ and $\uppi^{\mathbf{r}}$, $\scrX^{\mathbf{r}}$, $M^{\mathbf{r}}$ in \ref{notation: cAn}. By \ref{514}, there is a Type A potential $g$ on $Q_{3,I}$ such that $\Lambda_{\mathrm{con}}(\uppi^{\mathbf{r}}) \cong \Jac(Q_{3,I},g)$ for some $I \subseteq \{1,2,3\}$, and we set $f^{\mathbf{r}}\colonequals g$, which is well defined up to the isomorphism of Jacobi algebras. For $1 \leq i \leq 3$, write $f^i$ for $f^{(i)}$.
\end{notation}
Since we are interested only in derived equivalence classes, working with $f^{\mathbf r}$ and $f^i$ up to isomorphism of Jacobi algebras causes no ambiguity.
By \ref{36}, we have
\[
\Lambda_{\mathrm{con}}(\uppi^{\mathbf{r}}) \cong \underline{\End}_{\scrR}(M^{\mathbf{r}}) \cong \Jac(f^{\mathbf{r}}).
\]
If, in addition, $\scrR$ is isolated, then \ref{thm:august} yields $f^{\mathbf{r}} \simeq f$
(in the sense of \ref{211}).
Moreover, by \ref{352}, $\scrR$ is isolated if and only if $\dim_{\C}\Lambda_{\mathrm{con}}(\uppi)<\infty$,
equivalently, if and only if $\Jac(f)$ is finite-dimensional.

Consequently, to determine derived equivalence classes of Type~A potentials on $Q$ with $\dim_\C\Jac(f)<\infty$, it suffices to study iterated flops of crepant resolutions of isolated $cA_3$ singularities.
The finiteness assumption is essential here, since \ref{thm:august} requires the base singularity to be isolated.

In order to present the NCCRs $\End_{\scrR}(M)$ and $\End_{\scrR}(M^{\mathbf{r}})$, we adopt the following.
\begin{definition}
Let $\mathcal{Q}$ be the quiver obtained from $Q$ by adding a new vertex $0$, arrows $a_0,b_0$ between $0$ and $1$, arrows $a_3,b_3$ between $0$ and $3$, and possibly a loop at $0$, as illustrated below.
\end{definition}
 \[
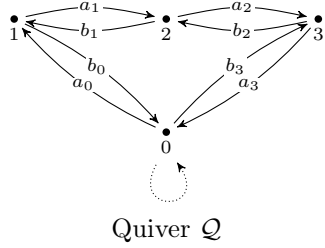
\begin{array}{c}
\begin{tikzpicture}[bend angle=12, looseness=1.2]
\node (a) at (-2,0) [vertex] {};
\node (b) at (0,0) [vertex] {};
\node (c) at (2,0) [vertex] {};
\node (d) at (0,-1.5) [vertex] {};

\node (a1) at (-2,-0.2) {$\scriptstyle 1$};
\node (a2) at (0,-0.2) {$\scriptstyle 2$};
\node (a3) at (2,-0.2) {$\scriptstyle 3$};
\node (a4) at (0,-1.7) {$\scriptstyle 0$};

\draw[->,bend left] (a) to node[gap] {$\scriptstyle a_{1}$} (b);
\draw[<-,bend right] (a) to node[gap] {$\scriptstyle b_{1}$} (b);
\draw[->,bend left] (b) to node[gap] {$\scriptstyle a_{2}$} (c);
\draw[<-,bend right] (b) to node[gap] {$\scriptstyle b_{2}$} (c);

\draw[->,bend left,looseness=0.9] (d) to node[gap] {$\scriptstyle a_{0}$} (a);
\draw[<-,bend right,looseness=0.9] (d) to node[gap] {$\scriptstyle b_{0}$} (a);
\draw[->,bend left,looseness=0.9] (d) to node[gap] {$\scriptstyle b_{3}$} (c);
\draw[<-,bend right,looseness=0.9] (d) to node[gap] {$\scriptstyle a_{3}$} (c);
\draw[<-,densely dotted]  (a4) edge [in=-120,out=-55,loop,looseness=8] node[below] {} (a4);
 
\node (z) at (0,-2.8) {Quiver $\mathcal{Q}$};
\end{tikzpicture}
\end{array}
\]

Since $f^i$ might be a potential in $Q_{3, I}$ for some $I \neq \emptyset$, and we aim to classify the derived equivalence classes of Type A potentials on $Q \colonequals Q_{3, \emptyset}$, we need the following lemma. 
\begin{lemma}\label{621}
Given a Type~A potential $f= \upkappa_1\x^p+\x\y+\upkappa_2\y^q$ in $Q$, the following statements hold.
\begin{enumerate}
\item $\upkappa_1 \neq 0$ and $p=2$ $\iff$ $f^1$ is a potential on $Q$.
\item $\upkappa_2 \neq 0$ and $q=2$ $\iff$ $f^3$ is a potential on $Q$. 
\item $\upkappa_1, \upkappa_2\neq 0$ and $p=q=2$ $\iff$ $f^2$ is a potential on $Q$.
\end{enumerate}
\end{lemma}

\begin{proof}
By \ref{715}, the potential $f$ can be realised by
\begin{equation*}
    \scrR \cong \frac{\mathbb{C} \lal u, v, x, y \ral}{uv-h_{0}h_{1}h_2h_3}
\end{equation*}
with associated $\scrR$-module $M=  \scrR \oplus (u,h_0) \oplus (u,h_0h_1) \oplus  (u,h_0h_1h_2)$ where 
\[
h_0=\upkappa_1p x^{p-1}+y,\quad h_1=x,\quad h_2=y,\quad h_3=x+\upkappa_2q y^{q-1}.
\]

Recall from notation \ref{notation:derived potential} that for each $1\leq i\leq 3$, the $f^i$ is realised by $\uppi^i  \colon \scrX^i \rightarrow \Spec \scrR$ such that $\Lambda_{\mathrm{con}}(\uppi^{i}) \cong \underline{\End}_{\scrR}(M^{i}) \cong \Jac(Q_{3,I_i}, f^{i})$ for some subset  $I_i \subseteq \{1,2,3\}$.

(1) By \ref{36}, $M^1=\scrR \oplus (u,h_1) \oplus (u,h_1h_0) \oplus  (u,h_1h_0h_2)$ and $\scrX^1$ is given pictorially by
\[
\begin{array}{ccc}
\begin{array}{c}
\scrX^1
\end{array} &
\begin{array}{c}
\begin{tikzpicture}[xscale=0.6,yscale=0.6]
\draw[black] (-0.1,-0.04,0) to [bend left=25] (2.1,-0.04,0);
\draw[black] (1.9,-0.04,0) to [bend left=25] (4.1,-0.04,0);

\draw[black] (3.9,-0.04,0) to [bend left=25] (6.1,-0.04,0);

\node at (1,0.6,0) {$\scriptstyle \Curve_{1}$};
\node at (3,0.6,0) {$\scriptstyle \Curve_{2}$};
\node at (5,0.6,0) {$\scriptstyle \Curve_{3}$};
\filldraw [black] (0,0,0) circle (1pt);
\filldraw [black] (2,0,0) circle (1pt);
\filldraw [black] (4,0,0) circle (1pt);
\filldraw [black] (6,0,0) circle (1pt);

\node at (0,-0.4,0) {$\scriptstyle h_1$};
\node at (2,-0.4,0) {$\scriptstyle h_0$};
\node at (4,-0.4,0) {$\scriptstyle h_2$};
\node at (6,-0.4,0) {$\scriptstyle h_{3}$};

\end{tikzpicture} 
\end{array}
\end{array}
\]
By \ref{35}, we have $I_1=\emptyset$ if and only if
$(h_1,h_0)=(x,y)$, $(h_0,h_2)=(x,y)$, and $(h_2,h_3)=(x,y)$,
which is equivalent to $\upkappa_1\neq 0$ and $p=2$.

(2) Similarly, by \ref{36} $M^3=\scrR \oplus (u,h_0) \oplus (u,h_0h_1) \oplus  (u,h_0h_1h_3)$ and $\scrX^3$ is depicted as
\[
\begin{array}{ccc}
\begin{array}{c}
\scrX^3
\end{array} &
\begin{array}{c}
\begin{tikzpicture}[xscale=0.6,yscale=0.6]
\draw[black] (-0.1,-0.04,0) to [bend left=25] (2.1,-0.04,0);
\draw[black] (1.9,-0.04,0) to [bend left=25] (4.1,-0.04,0);

\draw[black] (3.9,-0.04,0) to [bend left=25] (6.1,-0.04,0);

\node at (1,0.6,0) {$\scriptstyle \Curve_{1}$};
\node at (3,0.6,0) {$\scriptstyle \Curve_{2}$};
\node at (5,0.6,0) {$\scriptstyle \Curve_{3}$};
\filldraw [black] (0,0,0) circle (1pt);
\filldraw [black] (2,0,0) circle (1pt);
\filldraw [black] (4,0,0) circle (1pt);
\filldraw [black] (6,0,0) circle (1pt);

\node at (0,-0.4,0) {$\scriptstyle h_0$};
\node at (2,-0.4,0) {$\scriptstyle h_1$};
\node at (4,-0.4,0) {$\scriptstyle h_3$};
\node at (6,-0.4,0) {$\scriptstyle h_{2}$};

\end{tikzpicture} 
\end{array}
\end{array}
\]
Thus $I_3=\emptyset$ if and only if $\upkappa_2\neq 0$ and $q=2$.

(3) Similarly, by \ref{36} $M^2=\scrR \oplus (u,h_0) \oplus (u,h_0h_2) \oplus  (u,h_0h_2h_1)$ and $\scrX^2$ is depicted as
\[
\begin{array}{ccc}
\begin{array}{c}
\scrX^2
\end{array} &
\begin{array}{c}
\begin{tikzpicture}[xscale=0.6,yscale=0.6]
\draw[black] (-0.1,-0.04,0) to [bend left=25] (2.1,-0.04,0);
\draw[black] (1.9,-0.04,0) to [bend left=25] (4.1,-0.04,0);

\draw[black] (3.9,-0.04,0) to [bend left=25] (6.1,-0.04,0);

\node at (1,0.6,0) {$\scriptstyle \Curve_{1}$};
\node at (3,0.6,0) {$\scriptstyle \Curve_{2}$};
\node at (5,0.6,0) {$\scriptstyle \Curve_{3}$};
\filldraw [black] (0,0,0) circle (1pt);
\filldraw [black] (2,0,0) circle (1pt);
\filldraw [black] (4,0,0) circle (1pt);
\filldraw [black] (6,0,0) circle (1pt);

\node at (0,-0.4,0) {$\scriptstyle h_0$};
\node at (2,-0.4,0) {$\scriptstyle h_2$};
\node at (4,-0.4,0) {$\scriptstyle h_1$};
\node at (6,-0.4,0) {$\scriptstyle h_{3}$};

\end{tikzpicture} 
\end{array}
\end{array}
\]
Thus $I_2=\emptyset$ if and only if $\upkappa_1,\upkappa_2\neq 0$ and $p=q=2$.
\end{proof}

\begin{lemma}\label{622}
Suppose that $f$ is a Type A potential on $Q$. Then the following holds.
\begin{enumerate}
\item If $f=\x^2+\x\y+\uplambda\y^2$ with $\uplambda \neq 0$, then $f^1 \cong \x^2+\x\y+(\frac{1}{4}-\uplambda)\y^2 \cong f^3$ and $f^2\cong \x^2+\x\y+\frac{1}{16\uplambda}\y^2$.
\item If $f = \x^2+\x\y+\frac{1}{4}\y^2+\x^p$ with $p \geq 3$, then $f^1\cong \x^2+\x\y+\y^p$ and $f^3 \cong  \x^p+\x\y+\y^2$.
\end{enumerate}
\end{lemma}
\begin{comment}
If $f=\x^2+\x\y+\uplambda\y^q$, then 
$f_1 \cong \begin{cases}
    \x^2+\x\y+(\frac{1}{4}-\uplambda)\y^2 \text{, if } q=2\\
    \x^2+\x\y+\frac{1}{4}\y^2+\x^q \text{, if } q> 2  \text{ and } \uplambda \neq 0
\end{cases}$ on the $Q$.

If $f=\uplambda \x^p+\x\y+\y^2$, then $f_3 \cong \begin{cases}
    \x^2+\x\y+(\frac{1}{4}-\uplambda)\y^2 \text{, if } p=2\\
    \x^2+\x\y+\frac{1}{4}\y^2+\x^p \text{, if } p> 2  \text{ and } \uplambda \neq 0
\end{cases}$ on the $Q_{3,\emptyset}$.    
\end{comment}

\begin{proof}
Suppose that $f=\x^2+\x\y+\uplambda\y^p$. By \ref{715}, we can realize $f$ by 
\begin{equation*}
    \scrR \cong \frac{\mathbb{C} \lal u, v, x, y \ral}{uv-h_{0}h_{1}h_2h_3}
\end{equation*}
and $\scrR$-module $M=  \scrR \oplus (u,h_0) \oplus (u,h_0h_1) \oplus  (u,h_0h_1h_2)$ where $h_0=2\x+y$, $h_1=x$, $h_2=y$ and $h_3=x+\uplambda p\y^{p-1}$. Since $M^1=\scrR \oplus (u,h_1) \oplus (u,h_1h_0) \oplus  (u,h_1h_0h_2)$, then by \ref{35} ${\End}_{\scrR}(M^1)$ can be presented by $\mathcal{Q}$ with relations
\begin{align*}
 & \x b_1-\y b_1=2b_1b_0a_0, \quad b_2\x-b_2\y=2(a_3b_3b_2-\uplambda pb_2\y^{p-1}), \\
  & a_1\x - a_1 \y=2b_0a_0a_1,\quad  \x a_2-\y a_2= 2(a_2a_3b_3-\uplambda p\y^{p-1}a_2),
\end{align*}
together with additional relations factoring through the vertex $0$, 
which will not play a role in the following.
Hence $\underline{\End}_{\scrR}(M^1)$ can be presented by $Q$ with relations
\begin{align*}
 & \x b_1-\y b_1=0,\quad b_2\x-b_2\y=-2\uplambda pb_2\y^{p-1}, \\
  & a_1\x - a_1 \y=0, \quad \x a_2-\y a_2= -2\uplambda p\y^{p-1}a_2.
\end{align*}
Thus $\underline{\End}_{\scrR}(M^1) \cong \Jac(Q,f^1)$ where $f^1 \cong \frac{1}{2}\x^2-\x\y+\frac{1}{2}\y^2-2\uplambda\y^p$. Normalizing by applying $a_1 \mapsto -\sqrt{2}a_1$ and $a_2 \mapsto \frac{1}{\sqrt{2}}a_2$ to $f^1$ gives
\begin{equation*}
    f^1 \mapsto \x^2+\x\y+\frac{1}{4}\y^2-2^{1-\frac{p}{2}}\uplambda\y^p.
\end{equation*}
Setting $p=2$ in the above potential proves the statement about $f^1$ in (1). The proof of the statement about $f^3$ in (1) is similar. 

For $p \geq 3$ and $\uplambda \neq 0$, applying $a_1 \mapsto \frac{1}{2}b_2,b_1 \mapsto a_2,a_2 \mapsto 2b_1, b_2 \mapsto a_1$ gives
\begin{equation*}
\x^2+\x\y+\frac{1}{4}\y^2-2^{1-\frac{p}{2}}\uplambda\y^p \rightsquigarrow
\x^2+\x\y+\frac{1}{4}\y^2-2^{1+\frac{p}{2}}\uplambda\x^p.
\end{equation*}
Then since $p \geq 3$ and $\uplambda \neq 0$, by \ref{cor:A3norm}(2), 
\begin{equation*}
    \x^2+\x\y+\frac{1}{4}\y^2-2^{1+\frac{p}{2}}\uplambda\x^p \cong \x^2+\x\y+\frac{1}{4}\y^2+\x^p.
\end{equation*}

Therefore, for $p\ge 3$ we obtain $(\x^2+\x\y+\y^p)^1=\x^2+\x\y+\frac{1}{4}\y^2+\x^p$.
Since flopping is an involution, this is equivalent to the statement in~(2). The proof of the statement about $f^3$ in (2) is similar.

Then we finally prove the statement about $f^2$ in (1). In this case,  $g_0=2\x+y$, $g_1=x$, $g_2=y$ and $g_3=x+2\uplambda \y$. Since $M^2=\scrR \oplus (u,h_0) \oplus (u,h_0h_2) \oplus  (u,h_0h_2h_1)$, then by \ref{35} ${\End}_{\scrR}(M^2)$ can be presented by $\mathcal{Q}$ with relations
\begin{align*}
 & \x b_1+2\y b_1=b_1b_0a_0, \quad 2\uplambda b_2\x+b_2\y=a_3b_3b_2, \\
  & a_1\x + 2a_1 \y=b_0a_0a_1,\quad  2\uplambda\x a_2+\y a_2= a_2a_3b_3,
\end{align*}
plus some other relations that factor through the vertex $0$ (and so will not be relevant below).
Hence $\underline{\End}_{\scrR}(M^2)$ can be presented by $Q$ with relations
\begin{align*}
  & \x b_1+2\y b_1=0, \quad 2\uplambda b_2\x+b_2\y=0, \\
  & a_1\x + 2a_1 \y=0,\quad  2\uplambda\x a_2+\y a_2= 0.
\end{align*}
Thus $\underline{\End}_{\scrR}(M^2) \cong \Jac(Q,f^2)$ where $f^2 \cong \x^2+\x\y+\frac{1}{16\uplambda}\y^2$.
\end{proof}

Now let $\uppi: \scrX \to\Spec\scrR$ be a crepant resolution where $\scrR$ is cDV, with exceptional curves $\bigcup_{i=1}^n \Curve_i$. For any curve class $\upbeta \in \bigoplus_i \Z \left\langle \Curve_i \right\rangle$, there exists a integer-valued number $\GV_{\upbeta}(\uppi)$ called the Gopakumar--Vafa (GV for short) invariant of class $\upbeta$ \cite{BKL, MT}.
\begin{definition}
Define the \emph{GV set} of $\uppi$ as $\GV(\uppi)=\{\GV_{\upbeta}(\uppi)| \text{ all }  \upbeta \text{ with } \GV_{\upbeta}(\uppi) \neq 0\}$.
\end{definition}

%The following should be viewed as a generalization of \eqref{601}.
\begin{lemma}\label{lemma:gvf}
Let $\uppi_k \colon \scrX_k \rightarrow \Spec \scrR_k$ be two crepant resolutions of isolated \textnormal{cDV} singularities $\scrR_k$ for $k=1,2$. If $\Lambda_{\mathrm{con}}(\uppi_1)$ is derived equivalent to $\Lambda_{\mathrm{con}}(\uppi_2)$, then $\GV(\uppi_1) = \GV(\uppi_2)$.
\end{lemma}
\begin{proof}
Since $\Lambda_{\mathrm{con}}(\uppi_1)$ is derived equivalent to $\Lambda_{\mathrm{con}}(\uppi_2)$ and each $\scrR_k$ is isolated, then $\scrR_1 \cong \scrR_2$ by \ref{037}, so $\uppi_1$ and $\uppi_2$ are two crepant resolutions of a same cDV singularity and connected by a sequence of flops. Thus $\GV(\uppi_1) = \GV(\uppi_2)$ by \cite[5.4]{NW1}, \cite[5.10]{VG}.     
\end{proof}

For $\uplambda\in\C\setminus\{0,\tfrac14\}$, set
\[
\operatorname{Orbit}(\uplambda)\colonequals 
\left\{\uplambda,\frac{1-4\uplambda}{4},\frac{1}{4(1-4\uplambda)},\frac{\uplambda}{4\uplambda-1},\frac{4\uplambda-1}{16\uplambda},\frac{1}{16\uplambda}\right\}.
\]
This notation is purely definitional at this stage and records the parameters arising from the classification in \ref{611}.
\begin{theorem}\label{611}
The following groups the Type $A$ potentials on $Q$ with finite-dimensional Jacobi algebra into sets, where all the Jacobi algebras in a given set are derived equivalent.
\begin{enumerate}
\item $\{\x^2+\x\y+\uplambda'\y^2 \mid \uplambda'\in\operatorname{Orbit}(\uplambda)\}$ for $\uplambda \in \C \setminus \{0,\tfrac14\}$.
%\item $\{\x^2+\x\y+\uplambda' \y^2$ $\mid$ $\uplambda' = \uplambda,\frac{1-4\uplambda}{4},\frac{1}{4(1-4\uplambda)},\frac{\uplambda}{4\uplambda-1},\frac{4\uplambda-1}{16\uplambda},\frac{1}{16\uplambda}\}$ for $\uplambda \in \C \setminus \{0,\tfrac14\}$.
\item $\{\x^p+\x\y+\y^2$, $\x^2+\x\y+\y^p$, $\x^2+\x\y+\frac{1}{4}\y^2+ \x^p\}$ for $p \geq 3$.
\item $\{\x^p+\x\y+\y^q$, $\x^q+\x\y+\y^p \}$ for $p \geq 3$ and $q \geq 3$.    
\end{enumerate}
Moreover, the Jacobi algebras of the sets in \textnormal{(1)--(3)} are all mutually not derived equivalent, and in particular the Jacobi algebras of different sets in the same item are not derived equivalent. 
In \textnormal{(1)} there are no further basic algebras in the derived equivalence class, whereas in \textnormal{(2)--(3)} there are an additional finite number of basic algebras in the derived equivalence class.
\end{theorem}
\begin{proof}
By \ref{610} and \ref{352}, the potentials in the statement are precisely the Type $A$ potentials on $Q$ with finite-dimensional Jacobi algebra, thus they exhaust all possibilities.

Firstly, we prove that the Jacobi algebras in each given set are derived equivalent. By \ref{thm:august}, given a Type A potential $f$ with finite-dimensional Jacobi algebra and $\Jac(f) \cong \Lambda_{\mathrm{con}}(\uppi)$, if we want to obtain all the basic algebras that are derived equivalent to $\Jac(f)$, we only need to calculate all iterated flops from $\uppi$. So we consider $f^i$ for $1 \leq i \leq 3$ in the following.

(1) Suppose that $f=\x^2+\x\y+\uplambda\y^2$ where $\uplambda \neq 0, \frac{1}{4}$. 

By \ref{622}, $f^1 \cong \x^2+\x\y+(\frac{1}{4}-\uplambda)\y^2 \cong f^3$ and $f^2\cong \x^2+\x\y+\frac{1}{16\uplambda}\y^2$. Repeating the same argument, we have $f^{(12)} \cong \x^2+\x\y+\frac{1}{4(1-4\uplambda)}\y^2$, $f^{(21)} \cong \x^2+\x\y+\frac{4\uplambda-1}{16\uplambda}\y^2$ and $f^{(121)} \cong \x^2+\x\y+\frac{\uplambda}{4\uplambda-1}\y^2$. Repeating this process, only six numbers appear, so by \ref{thm:august} there are no further basic algebras in this derived equivalence class.

(2) Suppose that $f=\x^2+\x\y+\frac{1}{4}\y^2+ \x^p$ where $p \geq 3$. 

By \ref{622}, $f^1 \cong \x^2+\x\y+\y^p$ and $f^3 \cong \x^p+\x\y+\y^2$, and thus the three potentials in the statement are derived equivalent.
Since $p \geq 3$, then $f^{(12)}$, $f^{(13)}$, $f^{(31)}$ and $f^{(32)}$ are not potentials on $Q$ by \ref{621}, and so there are additional basic algebras in this derived equivalence class.

(3) By \ref{610}, $\x^p+\x\y+\y^q \cong \x^q+\x\y+\y^p$, and thus the two potentials in the statement are derived equivalent. Suppose that $f=\x^p+\x\y+\y^q $. Since $p \geq 3$ and $q \geq 3$, then $f^{1}$, $f^{2}$ and $f^3$ are not on $Q$ by \ref{621}, and so there are additional basic algebras in this derived equivalence class. 

The wall--chamber decomposition of the movable cone for a $cA_3$ crepant resolution is governed by the type~$A_3$ root system (see e.g. \cite[5.24, \S7]{W1}). These chambers are precisely the Weyl chambers, so their number equals the order of the Weyl group, namely $\# W(A_3) = |S_4| = 24$. 
Each chamber corresponds to a crepant resolution.

Moreover, the double $A_3$ quiver $Q$ admits a natural involution that sends 
$e_1 \mapsto e_3$, $e_2 \mapsto e_2$, and $e_3 \mapsto e_1$ 
(equivalently, exchanging $\x$ and $\y$ in the potentials). 
This symmetry identifies certain Jacobi algebras, so that there are at most 
$12$ distinct isomorphism classes. Consequently, the number of additional basic algebras appearing in the derived equivalence classes in cases~(2) and~(3) of \ref{611} is finite.

Secondly, we prove that the Jacobi algebras in different sets in \textnormal{(1)--(3)} are all mutually not derived equivalent. 

Given any potential $f$ in the statement, by \ref{715} we can find a Type 
$A_3$ crepant resolution $\uppi$ such that $\Lambda_{\mathrm{con}}(\uppi) \cong \Jac(f)$. By Toda's formula (see \cite[\S 4.4]{T2}), the GV set of each set in \textnormal{(1)--(3)} is:
\begin{dingautolist}{192}
    \item \{$1, 1, 1, 1, 1, 1$\},
    \item \{$1, 1, 1, p-1, 1, 1$\},
    \item \{$1, 1, 1, p-1, q-1, 1$\}.
\end{dingautolist}

Suppose that $f_1$ and $f_2$ are potentials in the statement with $f_1 \simeq f_2$ and each $\Jac(f_i)\cong \Lambda_{\mathrm{con}}(\uppi_i)$, where $\uppi_i \colon \scrX_i \rightarrow \Spec\scrR_i$ is a Type $A_3$ crepant resolution. Then $\Lambda_{\mathrm{con}}(\uppi_1)$ is derived equivalent to $\Lambda_{\mathrm{con}}(\uppi_2)$. Since each $\Jac(f_i)$ is finite-dimensional, then each $\scrR_i$ is isolated by \ref{352}, and so $\GV(\uppi_1) = \GV(\uppi_2)$ by \ref{lemma:gvf}. So if we want to prove that two potentials are not derived equivalent, we only need to prove that their corresponding GV sets are not equal.

Since $p \geq 3$, then any GV set in \ding{172} is different from that in \ding{173}, and so any set of potentials in (1) is not derived equivalent to that in (2). Since $q \geq 3$, then any GV set in \ding{173} is different from that in \ding{174}, and so any set of potentials in (2) is not derived equivalent to that in (3). Similar for (1) and (3).

Next, consider two sets of potentials in the same item. Given a potential $f$ in (1), we have already exhausted all $6$ potentials that are derived equivalent to $f$ in the above proof. Thus by \ref{thm:august}, different sets of potentials in (1) are not derived equivalent. Since different GV sets in \ding{173} are not equal, different sets of potentials in (2) are not derived equivalent. Similar for (3).
\end{proof}

\begin{remark}\label{625}
It is usually hard to give the derived equivalence class of an algebra $A$. 
But when $A$ is $\Jac(f)$ for a Type A potential $f$ on $Q_{n,I}$, there is a Type $A_n$ crepant resolution $\uppi\colon \scrX \rightarrow \Spec \scrR$ such that $A \cong  \Lambda_{\mathrm{con}}(\uppi)$ by \ref{511}. If further $A$ is finite-dimensional over $\C$, then $\scrR$ is isolated by \ref{352}. So we can apply \ref{thm:august} to get the \emph{full} derived equivalence class of $A$ by calculating all iterated flops from $\uppi$.
This is why we restrict this subsection to Type~A potentials on $Q$ with finite-dimensional Jacobi algebra.
\end{remark} 

\begin{remark}\label{626}
In cases \textnormal{(2)} and \textnormal{(3)} of \ref{611}, there are additional basic algebras in the derived equivalence class. These algebras are isomorphic to the Jacobi algebras of some potentials on $Q_{3, I}$ where $I \neq \emptyset$ (see the proof in \ref{611}).
\end{remark}

Recall the quaternion type quiver algebra $A_{p,q}(\upmu)$ from \cite{E1,H}. 
It is defined as the completion of the path algebra of the quiver $Q$
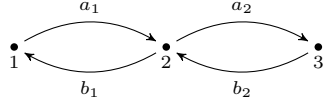
\[
\begin{tikzpicture}[bend angle=30, looseness=1]
\node (a) at (-2,0) [vertex] {};
\node (b) at (0,0) [vertex] {};
\node (c) at (2,0) [vertex] {};
\node (a1) at (-2,-0.2) {$\scriptstyle 1$};
\node (a1) at (0,-0.2) {$\scriptstyle 2$};
\node (a1) at (2,-0.2) {$\scriptstyle 3$};
\draw[<-,bend right] (a) to node[below] {$\scriptstyle b_1$} (b);
\draw[->,bend left] (a) to node[above] {$\scriptstyle a_1$} (b);
\draw[<-,bend right] (b) to node[below] {$\scriptstyle b_2 $} (c);
\draw[->,bend left] (b) to node[above] {$\scriptstyle a_2$} (c);
\end{tikzpicture}
\]
subject to the relations
\begin{equation*}
 a_1a_2b_2 -(a_1b_1)^{p-1}a_1, b_2b_1a_1-\upmu (b_2a_2)^{q-1}b_2, a_2b_2b_1-(b_1a_1)^{p-1}b_1, b_1a_1a_2-\upmu(a_2b_2)^{q-1}a_2,
\end{equation*}
where $\upmu \in \C$ and $p,q \geq 2$. 
Note that we impose fewer relations than in \cite{E1}, since we work in the completed path algebra.
Equivalently, $A_{p,q}(\upmu)\cong \Jac(Q,f)$, where
\begin{equation*}
f=\frac{1}{p}\x^p-\x\y+\frac{\upmu}{q}\y^q 
\;\cong\;
\x^p+\x\y+(-1)^q p^{-\frac{q}{p}} q^{-1}\upmu\,\y^q.
\end{equation*}
For convenience, set $B_{p,q}(\uplambda)\colonequals \Jac(Q,f)$ where $f=\x^p+\x\y+\uplambda\y^q$.
Then $A_{p,q}(\upmu) \cong B_{p,q}\!\big((-1)^q p^{-\frac{q}{p}} q^{-1}\upmu\big)$.
Denote $B_{p,q}(\uplambda) \colonequals  \Jac(Q,f)$ where $f=\x^p+\x\y+\uplambda\y^q$. Thus $A_{p,q}(\upmu) \cong B_{p,q}((-1)^qp^{-\frac{q}{p}}q^{-1}\upmu)$.

The following refines and extends the results of Erdmann and Holm \cite{E1,H}.

Similar to \ref{611}, for $\upmu \in\C\setminus\{0,1 \}$ we set a different orbit
\[
\operatorname{Orbit}(\upmu)\colonequals 
\left\{\upmu,1-\upmu, \frac{1}{1-\upmu}, \frac{\upmu}{\upmu-1}, \frac{\upmu-1}{\upmu}, \frac{1}{\upmu}\right\},
\]
which again serves as a convenient notation for the parameters occurring in the classification \ref{631} below.
\begin{cor}\label{631}
The following groups those algebras $A_{p,q}(\upmu)$ which are finite-dimensional into sets, where all the algebras in a given set are derived equivalent.
\begin{enumerate}
%\item $\{A_{2,2}(\upmu') \mid$ $\upmu' = \upmu,1-\upmu, \frac{1}{1-\upmu}, \frac{\upmu}{\upmu-1}, \frac{\upmu-1}{\upmu}, \frac{1}{\upmu}\}$ for $\upmu \in \C \setminus \{0,1\}$.
\item $\{A_{2,2}(\upmu') \mid$ $\upmu' \in \operatorname{Orbit}(\upmu) \}$ for $\upmu \in \C \setminus \{0,1\}$.
\item $\{A_{p,q}(1), A_{q,p}(1)\}$ for $(p,q)  \neq (2,2)$.
\end{enumerate}
Moreover, the algebras of the sets in \textnormal{(1)--(2)} are all mutually not derived equivalent.  
In \textnormal{(1)} there are no further basic algebras in the derived equivalence class, whereas in \textnormal{(2)} there are an additional finite number of basic algebras in the derived equivalence class.
\end{cor}

\begin{proof}
Using $A_{p,q}(\upmu) \cong B_{p,q}\!\big((-1)^q p^{-\frac{q}{p}} q^{-1}\upmu\big)$, in particular $A_{2,2}(\upmu)\cong B_{2,2}(\upmu/4)$, 
it follows from \ref{610} that the algebras listed in the statement are precisely the finite-dimensional ones, up to isomorphism.
We now show that the algebras within each set are derived equivalent.

(1) Since
\[
A_{2,2}(\upmu)\cong B_{2,2}(\frac{\upmu}{4})=\Jac\!\left(\x^2+\x\y+\frac{\upmu}{4}\y^2\right),
\]
the claim follows from \ref{611}(1). The same reference also shows that there are no further basic algebras in this derived equivalence class.

(2)
If $(p,q)\neq(2,2)$, then by the proof of \ref{cor:A3norm}(3) we have
$B_{p,q}\!\big((-1)^q p^{-\frac{q}{p}} q^{-1}\big)\cong B_{p,q}(1)$, hence $A_{p,q}(1)\cong B_{p,q}(1)$, and similarly $A_{q,p}(1)\cong B_{q,p}(1)$.
Therefore the stated pair is derived equivalent by \ref{611}(2)(3), which also implies that only finitely many further basic algebras occur in this derived equivalence class.

Finally, the algebras of the sets in \textnormal{(1)--(2)} are all mutually not derived equivalent by \ref{611}.
\end{proof}

\paragraph{Geometric interpretation of the $\mathfrak{S}_3$-symmetry.}\label{subsection:S3}

We now explain that the orbit structures appearing in \ref{611} and \ref{631},
which were introduced there in a purely ad hoc manner,
admit a uniform conceptual interpretation in terms of the $\mathfrak S_3$-action
arising from elliptic curves in Legendre form.
Conceptually, this reflects the fact that both the flopping geometry of $cA_3$ singularities and the Legendre form of elliptic curves are governed by configurations of four points on $\mathbb{P}^1$, up to projective transformation. We now make this connection precise.

We begin by recalling some standard facts about elliptic curves; see \cite[\S4.4]{H1}. Any elliptic curve $E$ over $\C$ can be written in Legendre form
\[
E \colon y^2 = x(x-1)(x-\uplambda),
\]
where $\uplambda \in \C \setminus \{0,1\}$. This realises $E$ as a double cover $E \to \mathbb{P}^1$ branched over the four points $0,\;1,\;\uplambda,\;\infty \in \mathbb{P}^1$. The $j$-invariant of $E$ is defined by
\begin{align}
j(\uplambda)= 2^8\frac{(\uplambda^2-\uplambda+1)^3}{\uplambda^2(\uplambda-1)^2}. \label{eq:j} 
\end{align}
Two elliptic curves over $\C$ are isomorphic if and only if they have the same $j$-invariant, and every element of $\C$ arises as the $j$-invariant of some elliptic curve.
%There is a one-to-one correspondence between the set of elliptic curves over $\C$, up to isomorphism, and the elements of $\C$, given by $E \rightarrow j(E)$. 

There is a natural action of the symmetric group $\mathfrak{S}_3$ on $\C \setminus \{0,1\}$, defined as follows. Given $\uplambda \in \C \setminus \{0,1\}$, we permute the three points $0,1,\uplambda$ according to an element $\sigma \in \mathfrak{S}_3$, and then apply a linear fractional transformation of $\mathbb{P}^1$ sending the first two back to $0$ and $1$.
The image of the third point defines $\sigma(\uplambda)$.
The resulting $\mathfrak{S}_3$-orbit of $\uplambda$ consists of
\[
\uplambda,\
1-\uplambda,\
\frac{1}{1-\uplambda},\
\frac{\uplambda}{\uplambda-1},\
\frac{\uplambda-1}{\uplambda},\
\frac{1}{\uplambda}.
\]
The group $\mathfrak{S}_3$ is generated by the transformations
\[
\uplambda \mapsto  1-\uplambda
\quad\text{and}\quad
\uplambda \mapsto \frac{1}{\uplambda}.
\]
Elements lying in the same orbit determine isomorphic elliptic curves, and each orbit coincides with a fibre of the $j$-invariant.

To relate the elliptic curve parameter $\uplambda$ to our setting, we consider the Type~$A$ potential
\[
f \colonequals \frac{1}{2}\x^2 - \x\y + \frac{1}{2}\uplambda \y^2
\;\cong\;
\x^2 + \x\y + \frac{1}{4}\uplambda \y^2, \quad \uplambda \in \C \setminus \{0,1\},
\]
which is precisely the potential $A_{2,2}(\uplambda)$ appearing in \ref{631}.
The restriction $\uplambda \in \C \setminus \{0,1\}$ agrees with the parameter range for elliptic curves in Legendre form.

By \ref{715} the potential $f$ is realised by the crepant resolution $\uppi \colon \scrX \rightarrow \Spec \scrR$ where 
\[
\scrR = \frac{\C \lal u, v, x, y \ral}{uv-(x-y)xy(x-\uplambda y)},
\]
with associated $\scrR$-module $M = \scrR \oplus(u,x-y)\oplus(u,(x-y)x)\oplus(u,(x-y)xy)$. Since $\uplambda \neq 0,1$, the singularity $\scrR$ is isolated by \ref{610}.
By \ref{36}, $\scrX$ is given pictorially by
\[
\begin{array}{ccc}
\begin{array}{c}
\scrX
\end{array} &
\begin{array}{c}
\begin{tikzpicture}[xscale=0.6,yscale=0.6]
\draw[black] (-0.1,-0.04,0) to [bend left=25] (2.1,-0.04,0);
\draw[black] (1.9,-0.04,0) to [bend left=25] (4.1,-0.04,0);

\draw[black] (3.9,-0.04,0) to [bend left=25] (6.1,-0.04,0);

\node at (1,0.6,0) {$\scriptstyle \Curve_{1}$};
\node at (3,0.6,0) {$\scriptstyle \Curve_{2}$};
\node at (5,0.6,0) {$\scriptstyle \Curve_{3}$};
\filldraw [black] (0,0,0) circle (1pt);
\filldraw [black] (2,0,0) circle (1pt);
\filldraw [black] (4,0,0) circle (1pt);
\filldraw [black] (6,0,0) circle (1pt);

\node at (0,-0.4,0) {$\scriptstyle x-y$};
\node at (2,-0.4,0) {$\scriptstyle x$};
\node at (4,-0.4,0) {$\scriptstyle y$};
\node at (6,-0.4,0) {$\scriptstyle x-\uplambda y$};

\end{tikzpicture} 
\end{array}
\end{array}
\]
If we evaluate the ratios $x/y$ at the defining linear factors $x-y,\ x,\ y,\ x-\uplambda y$, we obtain the values $1,\ 0,\ \infty,\ \uplambda$, which coincide exactly with the branch points of the elliptic curve in Legendre form.

We now follow the notation of \ref{notation:derived potential} and analyse how flops of $\scrX$ act on the parameter $\uplambda$.
For $1\leq i\leq 3$, we write $f^i$ for the Type~$A$ potential whose Jacobi algebra is isomorphic to the contraction algebra $\Lambda_{\mathrm{con}}(\uppi^i)$ of the crepant resolution $\uppi^i$ obtained by flopping the exceptional curve $\Curve_i$.
When $\scrR$ is isolated (equivalently, $\Jac(f)$ is finite-dimensional), by \ref{thm:august} we have the derived equivalence of potentials $f^i \simeq f$.

Consider the crepant resolution $\uppi^1 \colon \scrX^1 \rightarrow \Spec \scrR$ obtained by flopping the exceptional curve $\Curve_1$ of $\scrX$. The associated $\scrR$-module is $M^1= \scrR \oplus(u,x)\oplus(u,x(x-y))\oplus(u,x(x-y)y)$, and the corresponding contraction algebra satisfies $\Lambda_{\mathrm{con}}(\uppi^1) \cong \Jac(f^1)$ and $f \simeq f^1$.
By \ref{36}, $\scrX^1$ is given pictorially by
\[
\begin{array}{ccc}
\begin{array}{c}
\scrX^1
\end{array} &
\begin{array}{c}
\begin{tikzpicture}[xscale=0.6,yscale=0.6]
\draw[black] (-0.1,-0.04,0) to [bend left=25] (2.1,-0.04,0);
\draw[black] (1.9,-0.04,0) to [bend left=25] (4.1,-0.04,0);

\draw[black] (3.9,-0.04,0) to [bend left=25] (6.1,-0.04,0);

\node at (1,0.6,0) {$\scriptstyle \Curve_{1}$};
\node at (3,0.6,0) {$\scriptstyle \Curve_{2}$};
\node at (5,0.6,0) {$\scriptstyle \Curve_{3}$};
\filldraw [black] (0,0,0) circle (1pt);
\filldraw [black] (2,0,0) circle (1pt);
\filldraw [black] (4,0,0) circle (1pt);
\filldraw [black] (6,0,0) circle (1pt);

\node at (0,-0.4,0) {$\scriptstyle x$};
\node at (2,-0.4,0) {$\scriptstyle x-y$};
\node at (4,-0.4,0) {$\scriptstyle y$};
\node at (6,-0.4,0) {$\scriptstyle x-\uplambda y$};

\end{tikzpicture} 
\end{array}
\end{array}
\]
Applying the linear transformation of variables $\varphi_1 \colon x \mapsto x-y,\  y \mapsto -y,\  u \mapsto -u$ induces an isomorphism $\scrR \xrightarrow{\sim} \scrS_1$, where
\[
\scrS_1 \colonequals \frac{\C \lal u, v, x, y \ral}{uv-(x-y)xy(x-(1-\uplambda) y)}.
\]
Let $N_1 \colonequals \scrS_1 \otimes_{\scrR} M^1$. Then $N_1 \cong \scrS_1 \oplus(u,x-y)\oplus(u,(x-y)x)\oplus(u,(x-y)xy)$, and the corresponding crepant resolution $\uppi_1 \colon \scrY_1 \rightarrow \Spec \scrS_1$ is identified with $\uppi^1$ under the isomorphism $\scrR \cong \scrS_1$. In particular,
$\Lambda_{\mathrm{con}}(\uppi_1) \cong \Lambda_{\mathrm{con}}(\uppi^1)$. By \ref{36}, $\scrY_1$ is given pictorially by 
\[
\begin{array}{ccc}
\begin{array}{c}
\scrY_1
\end{array} &
\begin{array}{c}
\begin{tikzpicture}[xscale=0.6,yscale=0.6]
\draw[black] (-0.1,-0.04,0) to [bend left=25] (2.1,-0.04,0);
\draw[black] (1.9,-0.04,0) to [bend left=25] (4.1,-0.04,0);

\draw[black] (3.9,-0.04,0) to [bend left=25] (6.1,-0.04,0);

\node at (1,0.6,0) {$\scriptstyle \Curve_{1}$};
\node at (3,0.6,0) {$\scriptstyle \Curve_{2}$};
\node at (5,0.6,0) {$\scriptstyle \Curve_{3}$};
\filldraw [black] (0,0,0) circle (1pt);
\filldraw [black] (2,0,0) circle (1pt);
\filldraw [black] (4,0,0) circle (1pt);
\filldraw [black] (6,0,0) circle (1pt);

\node at (0,-0.4,0) {$\scriptstyle x-y$};
\node at (2,-0.4,0) {$\scriptstyle x$};
\node at (4,-0.4,0) {$\scriptstyle y$};
\node at (6,-0.4,0) {$\scriptstyle x-(1-\uplambda) y$};

\end{tikzpicture} 
\end{array}
\end{array}
\]
Thus, by \ref{715}, we have
$\Lambda_{\mathrm{con}}(\uppi_1)
\cong
\Jac\big(\x^2+\x\y+\frac{1}{4}(1-\uplambda)\y^2\big)$.
Since $\Lambda_{\mathrm{con}}(\uppi^1) \cong \Jac(f^1)$, it follows that
\[
f
=
\x^2+\x\y+\frac{1}{4}\uplambda\y^2
\;\simeq\;
\x^2+\x\y+\frac{1}{4}(1-\uplambda)\y^2
\;\cong\;
f^1.
\]
Thus, flopping the curve $\Curve_1$ induces the transformation $\uplambda \mapsto 1-\uplambda$ on the parameter.

%This parallels to the first generator of the group $\mathfrak{S}_{3}$ for the elliptic curves: $\uplambda \mapsto 1- \uplambda$.

By the symmetry of the exceptional curves $\Curve_1$ and $\Curve_3$, an entirely analogous argument applies to the flop of $\Curve_3$.
Consequently, we obtain
\[
f
=
\x^2+\x\y+\frac{1}{4}\uplambda \y^2
\;\simeq\;
\x^2+\x\y+\frac{1}{4}(1-\uplambda)\y^2
\;\cong\;
f^3.
\]
At the level of parameters, this corresponds to the transformation $\uplambda \mapsto 1-\uplambda$.

Consider the crepant resolution $\uppi^2 \colon \scrX^2 \rightarrow \Spec \scrR$ obtained by flopping the exceptional curve $\Curve_2$ of $\scrX$. The associated $\scrR$-module is $M^2 = \scrR \oplus(u,x-y)\oplus(u,(x-y)y)\oplus(u,(x-y)yx)$, and the corresponding contraction algebra satisfies $\Lambda_{\mathrm{con}}(\uppi^2) \cong \Jac(f^2)$ and $f \simeq f^2$.
By \ref{36}, $\scrX^2$ is given pictorially by
\[
\begin{array}{ccc}
\begin{array}{c}
\scrX^2
\end{array} &
\begin{array}{c}
\begin{tikzpicture}[xscale=0.6,yscale=0.6]
\draw[black] (-0.1,-0.04,0) to [bend left=25] (2.1,-0.04,0);
\draw[black] (1.9,-0.04,0) to [bend left=25] (4.1,-0.04,0);

\draw[black] (3.9,-0.04,0) to [bend left=25] (6.1,-0.04,0);

\node at (1,0.6,0) {$\scriptstyle \Curve_{1}$};
\node at (3,0.6,0) {$\scriptstyle \Curve_{2}$};
\node at (5,0.6,0) {$\scriptstyle \Curve_{3}$};
\filldraw [black] (0,0,0) circle (1pt);
\filldraw [black] (2,0,0) circle (1pt);
\filldraw [black] (4,0,0) circle (1pt);
\filldraw [black] (6,0,0) circle (1pt);

\node at (0,-0.4,0) {$\scriptstyle x-y$};
\node at (2,-0.4,0) {$\scriptstyle y$};
\node at (4,-0.4,0) {$\scriptstyle x$};
\node at (6,-0.4,0) {$\scriptstyle x-\uplambda y$};

\end{tikzpicture} 
\end{array}
\end{array}
\]
Since $\uplambda \neq 0$, applying the linear transformation of variables $\varphi_2 \colon x \mapsto y,\ y \mapsto x,\ u \mapsto \uplambda u$ induces an isomorphism $\scrR \xrightarrow{\sim} \scrS_2$, where
\[
\scrS_2 \colonequals \frac{\C \lal u, v, x, y \ral}{uv-(x-y)xy(x- \frac{1}{\uplambda} y)}.
\]
Let $N_2 \colonequals \scrS_2 \otimes_{\scrR} M^2$.
Then $N_2 \cong
\scrS_2 \oplus (u,x-y) \oplus (u,(x-y)x) \oplus (u,(x-y)xy)$,
and the corresponding crepant resolution
$\uppi_2 \colon \scrY_2 \rightarrow \Spec \scrS_2$
is identified with $\uppi^2$ under the induced isomorphism $\scrR \cong \scrS_2$.
In particular, $\Lambda_{\mathrm{con}}(\uppi_2) \cong \Lambda_{\mathrm{con}}(\uppi^2)$. By \ref{36}, $\scrY_2$ is given pictorially by 
\[
\begin{array}{ccc}
\begin{array}{c}
\scrY_2
\end{array} &
\begin{array}{c}
\begin{tikzpicture}[xscale=0.6,yscale=0.6]
\draw[black] (-0.1,-0.04,0) to [bend left=25] (2.1,-0.04,0);
\draw[black] (1.9,-0.04,0) to [bend left=25] (4.1,-0.04,0);

\draw[black] (3.9,-0.04,0) to [bend left=25] (6.1,-0.04,0);

\node at (1,0.6,0) {$\scriptstyle \Curve_{1}$};
\node at (3,0.6,0) {$\scriptstyle \Curve_{2}$};
\node at (5,0.6,0) {$\scriptstyle \Curve_{3}$};
\filldraw [black] (0,0,0) circle (1pt);
\filldraw [black] (2,0,0) circle (1pt);
\filldraw [black] (4,0,0) circle (1pt);
\filldraw [black] (6,0,0) circle (1pt);

\node at (0,-0.4,0) {$\scriptstyle x-y$};
\node at (2,-0.4,0) {$\scriptstyle x$};
\node at (4,-0.4,0) {$\scriptstyle y$};
\node at (6,-0.4,0) {$\scriptstyle x-\frac{1}{\uplambda}y$};

\end{tikzpicture} 
\end{array}
\end{array}
\]
Thus, by \ref{715}, we obtain
$\Lambda_{\mathrm{con}}(\uppi_2)
\cong
\Jac\big(\x^2+\x\y+\frac{1}{\uplambda}\y^2\big)$.
Since $\Lambda_{\mathrm{con}}(\uppi^2) \cong \Jac(f^2)$, it follows that
\[
f
=
\x^2+\x\y+\frac{1}{4}\uplambda\y^2
\;\simeq\;
\x^2+\x\y+\frac{1}{\uplambda}\y^2
\;\cong\;
f^2.
\]
At the level of parameters, this corresponds to the transformation
$\uplambda \mapsto \frac{1}{\uplambda}$.

Note that the above provides an alternative proof of \ref{622}(1).
Recall from the proof of \ref{611} that, in order to obtain all basic algebras derived equivalent to $\Jac(f)$, it suffices to consider all iterated flops of the crepant resolution $\uppi$.

As shown above, flopping the exceptional curves $\Curve_1$, $\Curve_2$, and $\Curve_3$ induces the transformations 
$\uplambda \mapsto 1-\uplambda$
and $\uplambda \mapsto \frac{1}{\uplambda}$, which generate the full $\mathfrak{S}_3$-action on $\C \setminus \{0,1\}$. Consequently, the derived equivalence class of the potential $f = \x^2+\x\y+\frac{1}{4}\uplambda \y^2$ consists precisely of the six-element orbit
\[
\left\{
\x^2+\x\y+\frac{1}{4}\uplambda' \y^2
\;\middle|\;
\uplambda' \in
\{
\uplambda,\;
1-\uplambda,\;
\frac{1}{1-\uplambda},\;
\frac{\uplambda}{\uplambda-1},\;
\frac{\uplambda-1}{\uplambda},\;
\frac{1}{\uplambda}
\}
\right\},
\]
where $\uplambda \in \C \setminus \{0,1\}$.

Therefore, in both the elliptic curve setting and the potential-theoretic setting, the parameter $\uplambda$ is acted on by a permutation of the distinguished points followed by a linear transformation of coordinates.
In each case this induces the same group action on $\uplambda$, generated by
\[
\uplambda \mapsto 1-\uplambda
\quad\text{and}\quad
\uplambda \mapsto \frac{1}{\uplambda}.
\]
Consequently, elliptic curves and Type~$A$ potentials have identical orbits in the parameter space.
Therefore we give a result parallel to the classical description for elliptic curves; see \cite[\S4.4]{H1}.
\begin{cor}
Let
\[
F \colonequals 
\left\{
\x^2+\x\y+\frac{1}{4}\uplambda \y^2
\;\middle|\;
\uplambda \in \C \setminus \{0,1\}
\right\},
\]
and for $f \in F$ define
$j(f) \colonequals j(\uplambda)$,
where $j(\uplambda)$ is given by \eqref{eq:j}.
Then:
\begin{enumerate}
\item the value $j(f)$ depends only on the derived equivalence class of $f$;
\item two potentials $f,f' \in F$ are derived equivalent if and only if $j(f)=j(f')$;
\item every element of $\C$ occurs as the $j$-invariant of some potential in $F$.
\end{enumerate}
Consequently, there is a one-to-one correspondence between derived equivalence classes of potentials in $F$ and elements of $\C$, given by $f \mapsto j(f)$.
Moreover, the fibre over $j(\uplambda)$ consists precisely of the $\mathfrak{S}_3$-orbit
\[
\{
\uplambda,\;
1-\uplambda,\;
\frac{1}{1-\uplambda},\;
\frac{\uplambda}{\uplambda-1},\;
\frac{\uplambda-1}{\uplambda},\;
\frac{1}{\uplambda}
\}.
\]
\end{cor}

It is natural to ask whether the function
\[
F \rightarrow \C,
\qquad
f \mapsto j(f),
\]
admits a natural extension to the boundary values $\uplambda \in \{0,1\}$.
Equivalently, one may ask whether this correspondence extends to a map
\[
\Bar{F}
\colonequals
\{
\x^2+\x\y+\frac{1}{4}\uplambda \y^2
\mid
\uplambda \in \C
\}
\longrightarrow
\mathbb{P}^1.
\]

From the geometric perspective developed above, this question reduces to understanding whether the potentials
\[
\x^2+\x\y \ \text{and } \x^2+\x\y+\frac{1}{4}\y^2,
\]
which arise from crepant resolutions of non-isolated $cA_3$ singularities (see \ref{610}), are derived equivalent.

\section{Appendix}\label{Appendix}
The purpose of this section is to prove \ref{thm31}, which gives a quiver presentation \eqref{Q2} of $\End_{\scrS}(N)$.
This is used to prove the geometric realization in $\S \ref{GR}$.  

We first introduce the reduction system and Diamond Lemma.
For a quiver $Q$, we denote the set of paths of degree $i$ by $Q_i$ where the degree is with respect to the path length, and write $Q_{\geq i}= \bigcup_{j \geq i} Q_j$ for the set of paths of degree $\geq i$. 
\begin{definition}\cite[\S 1]{B2}\label{def: RS}
Given a field $k$, a \emph{reduction system} $R$ for the path algebra $kQ$ is a set of pairs
\begin{equation*}
R=\left\{\left(s, \varphi_{s}\right) \mid s \in S \text { and } \varphi_{s} \in k Q\right\}
\end{equation*}
where
\begin{enumerate}
\item $S$ is a subset of $Q_{\geq 2}$ such that $s$ is not a sub-path of $s^{\prime}$ when $s \neq s^{\prime} \in S$.
\item For all $s \in S$, $s$ and $\varphi_{s}$ have the same head and tail.
\item For each pair $(s, \varphi_{s}) \in R$, $\varphi_{s}$ is \emph{irreducible}, meaning we can write $\varphi_{s}= \sum_{i} \uplambda_i p_i$ where each $\uplambda \in  k^{\times}$, and each $p_i$ does not contain elements in $S$ as a sub-path.
\end{enumerate}
\end{definition}
		
\begin{definition}
Let $(s, \varphi_{s}) \in R$ and let $q,\ r$ be two paths such that $qsr \neq 0$ in $kQ$. Following \cite[\S 2]{CS} a \emph{basic reduction} $\mathfrak{r}_{q,s,r}: kQ \rightarrow kQ$ is defined as the $k$-linear map uniquely determined by the following: for any path $p$ 
\begin{equation*}
\mathfrak{r}_{q, s, r}(p)= 
\begin{cases}q \varphi_{s} r & \text { if } p=qsr \\ 
p & \text { if } p \neq qsr\end{cases}
\end{equation*}
\end{definition}
Sometimes we write $p \rightarrow q \varphi_{s} r$ instead of $\mathfrak{r}_{q, s, r}(p)=q \varphi_{s} r$ for simplicity.

\begin{definition}
A \emph{reduction} $\mathfrak{r}$ is defined as a composition $\mathfrak{r}_{q_n, s_n, r_n} \circ \cdots \circ \mathfrak{r}_{q_2, s_2, r_2} \circ \mathfrak{r}_{q_1, s_1, r_1}$ of basic reductions for some $n \geq 1$. We say a path $p$ is \emph{reduction-finite} if for any infinite sequence of reductions $(\mathfrak{r}_i)_{i \in \mathbb{N}}$ there exists $n_0 \in \mathbb{N}$ such that for all $n \geq n_0$, we have $\mathfrak{r}_{n} \circ \cdots \circ \mathfrak{r}_{2} \circ \mathfrak{r}_{1}(p)=\mathfrak{r}_{n_{0}} \circ \cdots \circ \mathfrak{r}_{2} \circ \mathfrak{r}_{1}(p)$.    
\end{definition}

A path may contain many sub-paths in $S$, so one may obtain different elements in $kQ$ after performing different reductions.
\begin{definition}\cite[\S 1]{B2}
Let $R$ be a reduction system for $kQ$. A path $pqr \in Q_{\geq 3}$ for $p, q, r \in Q_{\geq 1}$ is an \emph{overlap ambiguity} of $R$ if $pq,\ qr \in S$.
%A path $pqt \in Q_{\geq 3}$ for $p, q, t \in Q_{\geq 1}$ is an inclusion ambiguity of $R$ if $q, pqt \in S$ (respectively, inclusion) (respectively, $r(p\varphi_{q}t) = r^{\prime}(\varphi_{pqt})$), \cite[3.10]{BW}
We say that an overlap ambiguity $pqr$ with $pq=s$ and $qr=s'$ is \emph{resolvable} if $\varphi_{s}r$ and $p\varphi_{s'}$ are reduction-finite and $\mathfrak{r}(\varphi_{s}t) = \mathfrak{r^{\prime}}(p\varphi_{s'})$ for some reductions $\mathfrak{r},\ \mathfrak{r^{\prime}}$.
\end{definition}

\begin{theorem} \textnormal{(Diamond Lemma)}\label{Diamond Lemma} \cite[1.2]{B2}
Let $R = \{(s, \varphi_s)\}_{s \in S}$ be a reduction system for $kQ$. Let $I=\left\langle s-\varphi_{s}\right\rangle_{s \in S} \subset kQ$ be the corresponding two-sided ideal and write $A = kQ/I$ for the quotient algebra. If $R$ is reduction-finite, then the following are equivalent:
\begin{enumerate}
\item All overlap ambiguities of $R$ are resolvable.
\item The image of the set of irreducible paths under the projection $kQ \rightarrow A$ forms a $k$-basis of $A$.
\end{enumerate}
\end{theorem}

Consider the following quiver $Q$ with relations $I$.
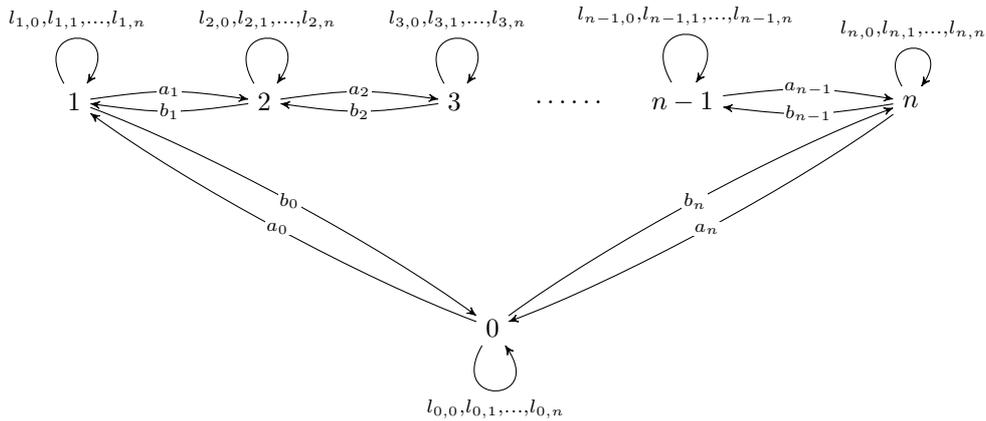
\[
\begin{tikzpicture}[bend angle=8, looseness=1.2]
\node (a) at (0,0)  {$1$};
\node (b) at (2.5,0)  {$2$};
\node (c) at (5,0) {$3$};
\node (d) at (6.5,0) {$\hdots \hdots$};
\node (e) at (8,0) {$n-1$};
\node (f) at (11,0)  {$n$};
\node (g) at (5.5,-3) {$0$};

\draw[<-,bend right,looseness=0.7] (g) to node [gap] {$\scriptstyle b_{0}$} (a);
\draw[->,bend left,looseness=0.7] (g) to node [gap] {$\scriptstyle a_{0}$} (a);
\draw[<-,bend right] (a) to node [gap] {$\scriptstyle b_{1}$} (b);
\draw[->,bend left] (a) to node [gap] {$\scriptstyle a_{1}$} (b);
\draw[<-,bend right] (b) to node[gap] {$\scriptstyle b_{2}$} (c);
\draw[->,bend left] (b) to node[gap] {$\scriptstyle a_{2}$} (c);
\draw[<-,bend right] (e) to node[gap] {$\scriptstyle b_{n-1}$} (f);
\draw[->,bend left] (e) to node[gap] {$\scriptstyle a_{n-1}$} (f);
\draw[<-,bend right,looseness=0.7] (g) to node[gap] {$\scriptstyle a_{n}$} (f);
\draw[->,bend left,looseness=0.7] (g) to node[gap] {$\scriptstyle b_{n}$} (f);

\draw[<-]  (g) edge [in=-120,out=-55,loop,looseness=8] node[below] {$\scriptstyle l_{0,0}, l_{0,1}, \dots,  l_{0,n}$} (g);
\draw[<-]  (a) edge [in=120,out=55,loop,looseness=8] node[above] {$\scriptstyle l_{1,0}, l_{1,1}, \dots,  l_{1,n}$} (a);
\draw[<-]  (b) edge [in=120,out=55,loop,looseness=8] node[above] {$\scriptstyle l_{2,0}, l_{2,1}, \dots,  l_{2,n}$} (b);
\draw[<-]  (c) edge [in=120,out=55,loop,looseness=8] node[above] {$\scriptstyle l_{3,0}, l_{3,1}, \dots,  l_{3,n}$} (c);
\draw[<-]  (e) edge [in=120,out=55,loop,looseness=8] node[above] {$\scriptstyle l_{n-1,0}, l_{n-1,1}, \dots,  l_{n-1,n}$} (e);
\draw[<-]  (f) edge [in=120,out=55,loop,looseness=8] node[above] {$\scriptstyle l_{n,0}, l_{n,1}, \dots,  l_{n,n}$} (f);
\end{tikzpicture}
\]

\begin{equation}\label{Q2}
 I \colonequals
  \begin{cases}
      l_{t,i}a_{t}=a_{t}l_{t+1,i}, \ l_{t+1,i}b_{t}=b_{t}l_{t,i}, \ l_{t,i}l_{t,j}=l_{t,j}l_{t,i}, \\
      l_{t,t}=a_{t}b_{t}, \ l_{t+1,t}=b_{t}a_{t} \text{ for any }t \in \mathbb{Z}/(n+1) \text{ and } 0 \leq i, j \leq n.
  \end{cases}  
\end{equation}

%where $i,j,k = 0 \dots n$ and we consider the indexes of $a, e, l$ in $\mathbb{Z}/(n+1)$. Hence $l_{n+1,n} = l_{0,n}$.\\
Then define the reduction system $R$ for the path algebra $kQ$ to be 
\begin{align}\label{rs}
 &  R \colonequals \{(l_{t,i}a_t, a_tl_{t+1,i}),
 (l_{t+1,i}b_t, b_tl_{t,i}),
 (a_tb_t, l_{t,t}),
 (b_ta_t, l_{t+1,t}),
 (l_{tj}l_{ti}, l_{ti}l_{tj}) \mid  \notag \\
& \textnormal{ for any } 0  \leq i \leq n,\ t \in \mathbb{Z}/(n+1) \text{ and } j >i\}.  
\end{align}
%It is clear that $R$ is a reduction system for $kQ$. 
We next prove that $R$ is reduction-finite and all overlap ambiguities of $R$ are resolvable.

\begin{lemma}\label{lemma:reduction1}
The reduction system $R$ \eqref{rs} is reduction-finite.
\end{lemma}
\begin{proof}
For any path $p$ and any infinite sequence of reductions $(\mathfrak{r}_i)_{i \in \mathbb{N}}$, if there does not exist $n_0 \in \mathbb{N}$ such that for all $n \geq n_0$ we have $\mathfrak{r}_{n} \circ \cdots \circ \mathfrak{r}_{1}(p)=\mathfrak{r}_{n_{0}} \circ \cdots \circ \mathfrak{r}_{1}(p)$, then there must exist infinite basic reductions that can be applied to $p$ consecutively. We prove that this is impossible. There are three types of path pairs in $R$:
\begin{enumerate}
    \item $(a_tb_t, l_{t,t}), (b_ta_t, l_{t+1,t})$.
    \item $(l_{t,i}a_t, a_tl_{t+1,i}), (l_{t+1,i}b_t, b_tl_{t,i})$.
    \item $(l_{t,j}l_{t,i}, l_{t,i}l_{t,j})$ for $j>i$.
\end{enumerate}
The type $(1)$ basic reduction decreases the path degree by one. 
The type $(2)$ basic reduction moves $a_t$ or $b_t$ one step left, and $l_{t,i}$ or $l_{t+1, i}$ one step right in the path. 
Similarly, the type $(3)$ basic reduction moves $l_{t, i}$ one step left, and $l_{t, j}$ one step right in the path for $j>i$.

Thus, any composition of these three types either decreases the path degree or moves $a_t$, $b_t$ to the left, $l_{t,j}$ with the larger $j$ to the right. Since the path degree of $p$ is finite, we can only apply the basic reductions of these three types to $p$ finitely many times. 
\end{proof}

\begin{lemma}\label{lemma:reduction2}
All overlap ambiguities of the reduction system $R$ \eqref{rs} are resolvable.
\end{lemma}
\begin{proof}
There are seven types of overlap ambiguities in the reduction system $R$ \eqref{rs}, namely
\[
l_{t,i}a_tb_t,\quad
l_{t+1,i}b_ta_t,\quad
l_{t,j}l_{t,i}a_t,\quad
l_{t+1,j}l_{t+1,i}b_t,\quad
a_tb_ta_t,\quad
b_ta_tb_t,\quad
l_{t,j}l_{t,i}l_{t,k},
\]
for $0 \leq i \leq n$, $t \in \mathbb{Z}/(n+1)$ and $j>i>k$.
We now verify that each of these ambiguities is resolvable.

\begin{enumerate}
\item When $t <i$,
$(l_{t,i}a_t)b_t \rightarrow a_t(l_{t+1,i}b_t) \rightarrow (a_tb_t)l_{t,i} \rightarrow l_{t,t}l_{t,i} $, and
$l_{t,i}(a_tb_t) \rightarrow l_{t,i}l_{t,t}  \rightarrow l_{t,t}l_{t,i} $.
The case of $t \geq i$ is similar.

\item When $t <i$,
$(l_{t+1,i}b_t)a_t \rightarrow b_t(l_{t,i}a_t) \rightarrow (b_ta_t)l_{t+1,i} \rightarrow l_{t+1, t}l_{t+1,i}$, and
$l_{t+1,i}(b_ta_t) \rightarrow l_{t+1,i}l_{t+1,t} \rightarrow l_{t+1,t}l_{t+1,i}.$
The case of $t \geq i$ is similar.

\item $(l_{t,j}l_{t,i})a_t \rightarrow l_{t,i}(l_{t,j}a_t) \rightarrow (l_{t,i}a_t)l_{t+1,j} \rightarrow a_tl_{t+1,i}l_{t+1,j}$,\\
$l_{tj}(l_{t,i}a_t) \rightarrow (l_{t,j}a_t)l_{t+1,i} \rightarrow a_t(l_{t+1,j}l_{t+1,i}) \rightarrow a_tl_{t+1,i}l_{t+1,j}$.

\item $(l_{t+1,j}l_{t+1,i})b_t \rightarrow l_{t+1,i}(l_{t+1,j}b_t) \rightarrow (l_{t+1,i}b_t)l_{t,j} \rightarrow b_tl_{t,i}l_{t,j}$,\\
$l_{t+1,j}(l_{t+1,i}b_t) \rightarrow (l_{t+1,j}b_t)l_{t,i} \rightarrow b_t(l_{t,j}l_{t,i}) \rightarrow b_tl_{t,i}l_{t,j}$. 

\item $(a_tb_t)a_t \rightarrow l_{t,t}a_t \rightarrow a_tl_{t+1,t}$, and $a_t(b_ta_t) \rightarrow a_tl_{t+1,t}$.
\item $(b_ta_t)b_t \rightarrow l_{t+1,t}b_t \rightarrow b_tl_{t,t}$, and $b_t(a_tb_t) \rightarrow b_tl_{t,t}$.
\item $(l_{t,j}l_{t,i})l_{t,k} \rightarrow l_{t,i}(l_{t,j}l_{t,k}) \rightarrow (l_{t,i}l_{t,k})l_{t,j} \rightarrow l_{t,k}l_{t,i}l_{t,j}$, \\
$l_{t,j}(l_{t,i}l_{t,k}) \rightarrow (l_{t,j}l_{t,k})l_{t,i} \rightarrow l_{t,k}(l_{t,j}l_{t,i}) \rightarrow l_{t,k}l_{t,i}l_{t,j}$. \qedhere
\end{enumerate}
\end{proof}

\begin{prop}\label{prop:diamond}
Consider the quiver $Q$ with relations $I$ \eqref{Q2} and its reduction system $R$ in \eqref{rs}. Then, the set of irreducible paths (with respect to $R$) of $kQ$ under the projection $kQ \rightarrow kQ/I$ forms a $k$-basis of $kQ/I$.
\end{prop}
\begin{proof}
It is clear that the two-sided ideal generated by $R$ (see \ref{Diamond Lemma}) coincides with $I$ \eqref{Q2}. Since $R$ is reduction-finite and all overlap ambiguities of $R$ are resolvable by \ref{lemma:reduction1} and \ref{lemma:reduction2}, the statement holds by \ref{Diamond Lemma}.   
\end{proof}

\begin{notation}\label{notation: ABL}
For any $t \in \Z/(n+1)$, consider the following subsets of the set of paths on $Q$ with head $t$.
\begin{enumerate}
\item $\scrA_t \colonequals \{  a_{t-i}\dots a_{t-2}a_{t-1} \mid  \text{ all } i\in \N \}$.
\item $\scrB_t \colonequals \{  b_{t+i-1}\dots b_{t+1}b_{t} \mid  \text{ all } i\in \N \}$.
\item $\scrL_t \colonequals \{l_{t,0}^{i_1}l_{t,1}^{i_2}\dots l_{t,n}^{i_n} \mid \text{ all } i_1,i_2,\dots , i_n \in \N \cup \{0\} \}$.
\item $\scrA_t \scrL_t \colonequals \{  pq \mid  \text{ all } p \in \scrA_t \text{ and } q \in \scrL_t  \}$.
\item $\scrB_t \scrL_t \colonequals \{  pq \mid  \text{ all } p \in \scrB_t \text{ and } q \in \scrL_t  \}$.
\item Then write $k \scrA_t$, $k \scrB_t$ and $k\scrL_t$ for the $k$-span of $\scrA_t$, $\scrB_t$ and $\scrL_t$ respectively.
\item For any $A \in k\scrA_t$, write $(A)_{t-1}$ for the unique element in $k\scrA_{t-1} \oplus ke_{t-1}$ such that $A= (A)_{t-1}a_{t-1}$. Here the summand $ke_{t-1}$ accounts for the case where $A=a_{t-1}$.
\item For any $B \in k\scrB_t$, write $(B)_{t+1}$ for the unique element in $k\scrB_{t+1} \oplus ke_{t+1}$ such that $B= (B)_{t+1}b_t$. Here the summand $ke_{t+1}$ accounts for the case where $B=b_t$.
\item For any $L \in k\scrL_t$ and $0 \leq  s \leq n $, write $(L)_s$ for the unique element in $k\scrL_{s}$, which is obtained by replacing $l_{t,0},l_{t,1}, \dots , l_{t,n}$ in $L$ by $l_{s,0},l_{s,1}, \dots , l_{s,n}$.
\end{enumerate}
\end{notation}

We next describe all irreducible paths in $Q$, with respect to the reduction system $R$ \eqref{rs}.

\begin{prop}\label{prop: irr}
For any path $p$ with head $t$ in $Q$, 
\begin{equation*}
    p \text{ is irreducible }  \iff p \in \scrA_t \cup \scrB_t \cup \scrL_t \cup \scrA_t \scrL_t \cup \scrB_t \scrL_t.
\end{equation*}
\end{prop}
\begin{proof}
By the reduction system $R$ \eqref{rs}, it is clear that each path in $\scrA_t, \scrB_t, \scrL_t, \scrA_t\scrL_t, \scrB_t\scrL_t$ is irreducible. We next prove the other direction.
Since the head of $p$ is $t$, $p$ either ends with $a_{t-1}$, $b_t$ or $l_{t,i}$ for some $i$. The proof splits into cases.

(1) $p$ ends with $a_{t-1}$.

Write $p = qa_{t-1}$ for some $q$ with head $t-1$. Then $q$ either ends with $a_{t-2}$, $b_{t-1}$ or $l_{t,i}$ for some $i$. However, if $q$ either ends with $b_{t-1}$ or $l_{t,i}$, then $qa_{t-1}$ is reducible by $R$ \eqref{rs}. Thus $q$ can only end with $a_{t-2}$. Repeating the same process gives $p \in \scrA_t$.

(2) $p$ ends with $b_{t}$.

Similar to (1), we can prove that $p \in \scrB_t$.

(3) $p$ ends with $l_{t,i}$.

Write $p =ql_{t,i}$ for some $q$ with head $t$. Then $q$ either ends with $a_{t-1}$, $b_t$ or $l_{t,j}$ for some $j$. If $q$ ends with $a_{t-1}$, then $q \in \scrA_t$ by (1), and so $p \in \scrA_t\scrL_t$. Similarly, if $q$ ends with $b_{t}$, then $p \in \scrB_t\scrL_t$. If $q$ ends  $l_{t,j}$, then $j \leq i$; otherwise, it will contradict the irreducibility of $ql_{t, i}$. Repeating the same process gives $p \in \scrL_t,\ \scrA_t\scrL_t$ or $\scrB_t\scrL_t$. 
\end{proof}

We next apply \ref{prop:diamond} and \ref{prop: irr} to prove the exactness of a particular complex in \ref{thm101}.
In the following, we write $P_t$ for the $k$-span of the paths with head $t$ in $kQ/I$ \eqref{Q2}.
\begin{lemma} \label{thm100}
The $k$-linear maps
\begin{align*}
& m_{l_{t,n}} \colon P_t \rightarrow P_t, \quad m_{a_t} \colon P_t \rightarrow P_{t+1} \\
& \quad \quad \quad f \mapsto fl_{t,n} \quad \qquad   f \mapsto fa_t
\end{align*}
are injective for any $ t \in \mathbb{Z}/(n+1)$.
\end{lemma}
\begin{proof}
We only prove $m_{l_{0,n}}$ and $m_{a_0}$ and are injective, the other cases are similar.
Since the reduction system $R$ \eqref{rs} is reduction-finite by \ref{lemma:reduction1}, we can assume $f \in P_0$ is irreducible. 

(1) $m_{l_{0,n}}$ is injective.

We first write $f= \sum_i \uplambda_ip_i$ as a linear combination of irreducible paths where each $\uplambda_i \in k$.
Since $p_i$ is irreducible and there are no paths in $S$ \eqref{rs} that end with $l_{0,n}$, $p_il_{0,n}$ is also irreducible. Thus if $fl_{0,n}=\sum_i \uplambda_ip_il_{0,n} = 0$, then each $\uplambda_i =0$ by \ref{prop:diamond}, and so $f=0$.

(2) $m_{a_0}$ is injective.

Since $f \in P_0$, by \ref{prop: irr} we can write $f$ as a linear combination of irreducible terms
\begin{equation*}
    f = \uplambda A+ \upmu B +\upbeta L+ \sum_{i} \uplambda_{i}A_{i}L_{i} +\sum_{j} \upmu_{j}B_{j}L_{j},
\end{equation*}
where each $\uplambda, \upmu, \upbeta , \uplambda_{i}, \upmu_{j} \in k$, and $A, A_{i}\in  k\scrA_0$, and $B, B_{j}\in  k\scrB_0$, and $L, L_{i}, L_{j}\in  k\scrL_0$. Thus
\begin{align}
fa_0&  = \uplambda Aa_0+ \upmu Ba_0 +\upbeta  La_0+ \sum_{i} \uplambda_{i}A_{i}L_{i}a_0 +\sum_{j}  \upmu_{j}B_{j}L_{j}a_0 \notag\\
&=  \uplambda Aa_0 +  \upmu(B)_1b_0a_0 +\upbeta La_0+ \sum_{i} \uplambda_{i}A_{i}L_{i}a_0 +\sum_{j}  \upmu_{j}(B_{j})_1b_0L_{j}a_0\tag{since $B=(B)_1b_0$ and $B_{j}=(B_{j})_1b_0$ } \\
& \rightarrow \uplambda Aa_0+ \upmu(B)_1l_{1,0}+\upbeta a_0(L)_1+\sum_{i}\uplambda_{i} A_{i}a_0(L_{i})_1+\sum_{j}\upmu_{j}(B_{j})_1l_{1,0}(L_{j})_1.  \label{fa} \\
& \tag{since $b_0L_ja_0 \rightarrow b_0a_0(L_j)_1 \rightarrow l_{1,0}(L_j)_1 $} 
\end{align}
By \ref{prop: irr}, each term in \eqref{fa} is irreducible. We next claim that each term in \eqref{fa} differs from the others.

Since $A_{i}L_{i}$ are different for different $i$, $A_{i}a_0(L_{i})_1$ are different for different $i$. Similarly, $(B_{j})_1l_{1,0}(L_{j})_1 $ are different for different $j$. 
Since $\deg(A_i) \geq 1$, $A_{i}a_0(L_{i})_1$ is different from $a_0(L)_1$ for each $i$. Similarly, $(B_{j})_1 l_{1,0}(L_{j})_1$ is different from $(B)_1l_{1,0}$ for each $j$. Thus we proved the claim.

So by \ref{prop:diamond} the terms in \eqref{fa} descend to give basis elements of $kQ/I$. Thus if $fa_0=0$, then each $\uplambda, \upmu, \upbeta ,\uplambda_{i}, \upmu_{j}$ is zero, and so $f=0$. Thus $m_{a_0}$  is injective.
\end{proof}

\begin{prop}\label{thm101}
\begin{equation*}
0 \rightarrow  P_0 \xrightarrow[d_4]{(a_0, b_n)}  P_1 \oplus P_n \xrightarrow[d_3]{\begin{psmallmatrix} l_{1,n} & -b_0b_n \\-a_na_0 & l_{n,0}  \end{psmallmatrix}}  P_1 \oplus P_n  \xrightarrow[d_2]{\begin{psmallmatrix}  b_0 \\a_n  \end{psmallmatrix}} P_0 \xrightarrow[d_1]{} k[l_{0,1}, l_{0,2},\dots ,l_{0,n-1}] \rightarrow 0
\end{equation*}
is an exact sequence of $k$-linear maps in $\mathrm{k}Q/I$ \eqref{Q2}.
\end{prop}
\begin{proof}
This sequence is a chain complex from the relations $I$ \eqref{Q2}. The exactness at the last three indices is from \cite[\S 6]{W2}. By \ref{thm100}, we have $d_4$ is injective, and thus this complex is exact at the first index.
So we only need to prove that $\ker d_3 \subseteq \operatorname{im} d_4$. It suffices to prove that, for any $(f,g) \in P_1 \oplus P_n$,
\begin{equation*}
fl_{1,n} = ga_na_0   \Rightarrow (f,g) = (ha_0,hb_n) \text{ for some }h \in P_0.
\end{equation*}
Since the reduction system $R$ \eqref{rs} is reduction-finite by \ref{lemma:reduction1}, we can assume that $f$ and $g$ are irreducible. Since $f$ is irreducible and there are no paths in $S$ \eqref{rs} that end with $l_{1,n}$, then $fl_{1,n}$ is also irreducible. Since 
$g \in P_n$, by \ref{prop: irr} we can write $g$ as a linear combination of irreducible terms
\begin{equation*}
g=  A + B + \sum_{i=0}^n  L_{i}l_{n,i}  + \sum_{i=0}^n\sum_{j \geq 1}  A_{ij} K_{ij}l_{n,i} + \sum_{s \geq 1} B_sJ_s,
\end{equation*}
where each $A, A_{ij}\in  k\scrA_n$, and $B,B_{s}\in  k\scrB_n$, and $L_i, K_{ij},J_{s} \in  k\scrL_n$. Since $ L_{i}l_{n,i}$ is irreducible, $L_{i} \in k \langle l_{n,0},\dots, l_{n,i} \rangle$ for each $i$. Similarly, $K_{ij} \in k \langle l_{n,0},l_{n,1},\dots, l_{n,i} \rangle$ for each $i$ and $j$. 

Write
\[
B = \uplambda b_n + B_{0,1} b_0b_n
\quad\text{and}\quad
B_s = \upmu_s b_n + B_{s,1} b_0b_n
\]
for some $\uplambda, \upmu_s \in k$ and
$B_{0,1}, B_{s,1} \in k\scrB_1 \oplus ke_1$ (see \ref{notation: ABL}(8)), where $s \ge 1$. Thus
\begin{align*}
g &=  A + \uplambda b_n + B_{0,1} b_0b_n + \sum_{i=0}^n  L_{i}l_{n,i}  + \sum_{i=0}^n\sum_{j \geq 1}  A_{ij} K_{ij}l_{n,i} + \sum_{s \geq 1} (\upmu_s b_n + B_{s,1}b_0b_n)J_s\\
& =A + \uplambda b_n + B_{0,1} b_0b_n + \sum_{i=0}^n  L_{i}l_{n,i}  + \sum_{i=0}^n\sum_{j \geq 1}  A_{ij} K_{ij}l_{n,i} +\sum_{s \geq 1}\upmu_s b_nJ_s+\sum_{s \geq 1} B_{s,1}b_0b_nJ_s.
\end{align*}

Multiplying the $g$ above on the right by $a_na_0$, $ga_na_0$ equals 
\begin{align}
& Aa_na_0 + \uplambda b_na_na_0 + B_{0,1} b_0b_na_na_0 + \sum_{i}  L_{i}l_{n,i}a_na_0  + \sum_{i,j}  A_{ij} K_{ij}l_{n,i} a_na_0 \notag\\
&+\sum_{s}\upmu_s b_nJ_sa_na_0+\sum_{s } B_{s,1}b_0b_nJ_sa_na_0 \notag\\
& \rightarrow Aa_na_0 + \uplambda a_0l_{1,n}+B_{0,1}l_{1,0}l_{1,n}+\sum_i a_na_0(L_i)_1l_{1,i}+ \sum_{i,j}A_{ij}a_na_0(K_{ij})_1l_{1,i} \notag \\
& +\sum_s \upmu_s a_0(J_s)_1l_{1,n}+\sum_s B_{s,1}l_{1,0}(J_s)_1l_{1,n}  \label{gaa}\\
& \tag{since $b_0b_nJ_s a_na_0 \rightarrow b_0b_na_na_0(J_s)_1 \rightarrow b_0 l_{0,n}a_0(J_s)_1 \rightarrow b_0a_0l_{1,n}(J_s)_1\rightarrow l_{1,0}(J_s)_1l_{1,n}$}\\
&=Aa_na_0+\uplambda a_0l_{1,n}+B_{0,1}l_{1,0}l_{1,n}+\sum_{i=0}^{n-1} a_na_0(L_i)_1l_{1,i}+a_na_0(L_n)_1l_{1,n}+\sum_j A_{nj}a_na_0(K_{nj})_1l_{1,n}\notag\\
& + \sum_{i=0}^{n-1}\sum_jA_{ij}a_na_0(K_{ij})_1l_{1,i}  +\sum_s \upmu_s a_0(J_s)_1l_{1,n}+\sum_s B_{s,1}l_{1,0}(J_s)_1l_{1,n}  \notag \\
& =Aa_na_0+\sum_{i=0}^{n-1} a_na_0(L_i)_1l_{1,i}+ \sum_{i=0}^{n-1}\sum_jA_{ij}a_na_0(K_{ij})_1l_{1,i} +\mathsf{f}_1l_{1,n},\label{gaa2}
\end{align}
where we set $\mathsf{f}_1 \colonequals \uplambda a_0+B_{0,1}l_{1,0}+a_na_0(L_n)_1+\sum_j A_{nj}a_na_0(K_{nj})_1+\sum_s \upmu_s a_0(J_s)_1+\sum_s B_{s,1}l_{1,0}(J_s)_1$.

We claim that each term in \eqref{gaa} is irreducible. To see this, we consider the terms in \eqref{gaa} separately.
\begin{enumerate}
\item By the reduction system $R$ \eqref{rs} $Aa_na_0$ is irreducible.

\item Since $l_{1,n}, l_{1,0}l_{1,n} \in \scrL_{1}$, by \ref{prop: irr} $a_0l_{1,n}$ and $B_{0,1}l_{1,0}l_{1,n}$ are irreducible.

\item Since  $L_{i} \in k \langle l_{n,0},\dots, l_{n,i} \rangle$, $(L_{i})_1 \in k \langle l_{1,0},\dots, l_{1,i} \rangle$, so $(L_{i})_1 l_{1,i} \in k\scrL_1$. Thus $a_na_0 (L_{i})_1 l_{1,i}$ is irreducible by \ref{prop: irr}.
\item Since $K_{ij} \in k \langle l_{n,0},\dots, l_{n,i} \rangle$, $(K_{ij})_1 \in k \langle l_{1,0},\dots, l_{1,i} \rangle$, so $ (K_{ij})_1 l_{1,i} \in k\scrL_1$. Thus $A_{ij} a_n a_0  (K_{ij})_1 l_{1,i} $ is irreducible by \ref{prop: irr}.

\item Since $(J_s)_1 l_{1,n}, l_{1,0} (J_s)_1 l_{1,n} \in k\scrL_{1}$, by \ref{prop: irr} $a_0(J_s)_1 l_{1,n}$ and $B_{s,1}l_{1,0} (J_s)_1 l_{1,n}$ are irreducible.
\end{enumerate}

We next claim that all terms in \eqref{gaa} are pairwise distinct.

First, each term of the form $a_na_0 (L_i)_1 l_{1,i}$ ends with $l_{1,i}$, hence these terms are distinct for different values of $i$. 
Similarly, since the paths $A_{ij}K_{ij}l_{n,i}$ are distinct for different pairs $(i,j)$, the corresponding terms
$A_{ij}a_na_0 (K_{ij})_1 l_{1,i}$ are also distinct.
Likewise, the terms $a_0(J_s)_1l_{1,n}$ and $B_{s,1}l_{1,0}(J_s)_1l_{1,n}$ are distinct for different $s$.
Therefore, the terms appearing within each individual sum are pairwise different.

Next, we compare terms coming from different sums.
Since $\deg(A_{ij}) \ge 1$, the path $A_{ij}a_na_0$ differs from $a_na_0$, and hence
$A_{ij}a_na_0 (K_{ij})_1 l_{1,i}$ cannot coincide with $a_na_0 (L_i)_1 l_{1,i}$.
By the same reasoning, terms arising from different summands in \eqref{gaa} are mutually distinct.
This proves the claim.

Since \eqref{gaa2} is obtained by combining the terms in \eqref{gaa} that end with $l_{1,n}$, each term in \eqref{gaa2} is also irreducible and differs from the others. So the terms in \eqref{gaa2} descend to give different basis elements of $kQ/I$ by \ref{prop:diamond}. 

Recall that $fl_{1,n} = ga_na_0$ and $fl_{1,n}$ is irreducible. Since only $\mathsf{f}_1l_{1,n}$ ends with $l_{1,n}$ in $ga_na_0$ \eqref{gaa2}, then all terms in $ga_na_0$ except $\mathsf{f}_1l_{1,n}$ are zero. So
\begin{align*}
g & = \uplambda b_n + B_{0,1} b_0b_n+L_{n}l_{n,n} +\sum_{j }  A_{nj} K_{nj}l_{n,n}+\sum_{s } (\upmu_s b_n + B_{s,1}b_0b_n)J_s\\
& = \uplambda b_n+ B_{0,1} b_0b_n+L_{n}a_nb_n+\sum_{j }  A_{nj} K_{nj}a_nb_n+\sum_{s } (\upmu_s  + B_{s,1}b_0  )(J_s)_0b_n \tag{since $l_{n,n}=a_nb_n$ and $b_nJ_s= (J_s)_0b_n$} \\
& = hb_n. \tag{set $h \colonequals \uplambda +B_{0,1}b_0+  L_{n}a_n +\sum_j \uplambda_{nj} A_{nj} K_{nj}a_n + \sum_s (\upmu_{s}+ B_{s,1}b_0) (J_s)_0  $}
\end{align*}

Thus $ga_na_0=hb_n a_n a_0 = ha_0l_{1,n}$. Together with $fl_{1,n} = ga_na_0$ gives $fl_{1,n} = ha_0l_{1,n}$, and so $f=ha_0$ by \ref{thm100}. Thus $ (f,g) = (ha_0,hb_n)$, proving the claim.
\end{proof}

With the exact sequence in \ref{thm101}, we can calculate the vector space dimension of each graded degree piece of $P_t$ in \eqref{106}, which will be used to prove the isomorphism in \ref{thm31}.
\begin{notation}\label{no10}
In the following, we adopt a new definition of degree of $Q$ \eqref{Q2}, which differs from path length in \ref{QP}(4).
\begin{enumerate}
\item Define $\deg(a_i)=\deg(b_i)=1$ and $\deg(l_{t,i})=2$ for each $i$ and $t$.
\item With respect to this degree, write $P_{t,d}$ for the graded piece of degree $d$ of $P_t$.   
\item Write $D_d$ for the vector space dimension of $P_{0,d}$.
\end{enumerate}
\end{notation}
By the symmetry of the quiver $Q$ and relations $I$ \eqref{Q2}, $D_d$ is also the vector space dimension of $P_{t,d}$ for $1 \leq t \leq n$.
By \ref{thm101}, for any integer $d$, there is an exact sequence
\begin{equation*}
0 \rightarrow  P_{0,d} \rightarrow  P_{1,d+1} \oplus P_{n,d+1} \rightarrow  P_{1,d+3} \oplus P_{n,d+3}  \rightarrow P_{0,d+4} \rightarrow T_{d+4} \rightarrow 0,  
\end{equation*}
where $T_{d+4}$ denotes the degree $d+4$ piece of $k[l_{0,1},l_{0,2}, \dots ,l_{0,n-1}]$. Thus
\begin{equation}
    D_d - 2D_{d+1} + 2D_{d+3} - D_{d+4} + E_{d+4} = 0 \label{106}
\end{equation}
where $E_{d+4}= \dim_{k}T_{d+4}$. 

By definition of the grading, $P_{t,d}=0$ for all $t$ and all $d<0$, hence $D_d=0$ for $d<0$.
Moreover, the exact sequence above and \eqref{106} both hold for $d<0$.

\begin{notation}\label{scrS}
We next define 
\begin{equation}
    \scrS \colonequals \frac{k[ u,v,x_0,x_1, \dots ,x_n ] }{uv - x_0x_1 \dots x_{n}},
\end{equation}
and consider the $\scrS$-module $N \colonequals \bigoplus_{i=0}^{n}N_i$ where $N_0\colonequals \scrS$ and $N_i \colonequals (u,\prod_{j=0}^{i-1} x_{j})$ for $1 \leq i \leq n$. 
\end{notation}

We will show that $kQ/I$ \eqref{Q2} presents $\End_{\scrS}(N)$.
By \cite{IW1}, every morphism in $\End_{\scrS}(N)$ can be obtained as a linear combination of compositions of the following maps.

\begin{equation}\label{Q3}
\begin{array}{c}
\begin{tikzpicture}[bend angle=7, looseness=1.2]
\node (a) at (0,0)  {$N_{1}$};
\node (b) at (2.5,0)  {$N_{2}$};
\node (c) at (5,0) {$\circ$};
\node (d) at (6,0) {$\hdots$};
\node (e) at (7,0) {$\circ$};
\node (f) at (9.5,0)  {$N_{n}$};
\node (g) at (4.75,-2.5) {$N_0$};

\draw[<-,bend right,looseness=0.8] (g) to node [gap] {$\scriptstyle inc$} (a);
\draw[->,bend left,looseness=0.8] (g) to node [gap] {$\scriptstyle x_0$} (a);
\draw[<-,bend right] (a) to node [gap] {$\scriptstyle inc$} (b);
\draw[->,bend left] (a) to node [gap] {$\scriptstyle x_1$} (b);
\draw[<-,bend right] (b) to node[gap] {$\scriptstyle inc$} (c);
\draw[->,bend left] (b) to node[gap] {$\scriptstyle x_2$} (c);
\draw[<-,bend right] (e) to node[gap] {$\scriptstyle inc$} (f);
\draw[->,bend left] (e) to node[gap] {$\scriptstyle x_{n-1}$} (f);

\draw[<-,bend right,looseness=0.8] (g) to node[gap] {$\scriptstyle \frac{x_{n}}{u}$} (f);
\draw[->,bend left,looseness=0.8] (g) to node[gap] {$\scriptstyle u$} (f);
\draw[<-]  (g) edge [in=-120,out=-55,loop,looseness=8] node[below] {$\scriptstyle x_0,x_1,\dots , x_n$} (g);
\draw[<-]  (a) edge [in=120,out=55,loop,looseness=8] node[above] {$\scriptstyle x_0,x_1,\dots , x_n$} (a);
\draw[<-]  (b) edge [in=120,out=55,loop,looseness=8] node[above] {$\scriptstyle x_0,x_1,\dots , x_n$} (b);
\draw[<-]  (f) edge [in=120,out=55,loop,looseness=7] node[above] {$\scriptstyle x_0,x_1,\dots , x_n$} (f);
\end{tikzpicture}
\end{array}
\end{equation}

Thus there exists an obvious surjective homomorphism $kQ \twoheadrightarrow  \End_{\scrS}(N)$. Since $I$ gets sent to zero by inspection, this induces a surjective homomorphism $\psi \colon kQ/I \twoheadrightarrow  \End_{\scrS}(N)$. We will show that $\psi$ is an isomorphism, by counting graded pieces.

\begin{notation}\label{no11}
Grade $\scrS$ via $\deg(u)=\deg(v)=n+1$ and $\deg(x_0)= \dots = \deg(x_n)=2$. The particular choice of the graded shift of $N$ given by $N \colonequals \bigoplus_{i=0}^n N_i(-i)$ induces a grading in $\End_{\scrS}(N)$, which explicitly grades each arrow in \eqref{Q3} as follows.
\end{notation}

\begin{equation}\label{Q4}
\begin{array}{c}
\begin{tikzpicture}[bend angle=7, looseness=1.2]
\node (a) at (0,0)  {$N_{1}(-1)$};
\node (b) at (2.5,0)  {$N_{2}(-2)$};
\node (c) at (5,0) {$\circ$};
\node (d) at (6,0) {$\hdots$};
\node (e) at (7,0) {$\circ$};
\node (f) at (9.5,0)  {$N_{n}(-n)$};
\node (g) at (4.75,-2.5) {$N_0$};

\draw[<-,bend right,looseness=0.8] (g) to node [gap] {$\scriptstyle 1$} (a);
\draw[->,bend left,looseness=0.8] (g) to node [gap] {$\scriptstyle 1$} (a);
\draw[<-,bend right] (a) to node [gap] {$\scriptstyle 1$} (b);
\draw[->,bend left] (a) to node [gap] {$\scriptstyle 1$} (b);
\draw[<-,bend right] (b) to node[gap] {$\scriptstyle 1$} (c);
\draw[->,bend left] (b) to node[gap] {$\scriptstyle 1$} (c);
\draw[<-,bend right] (e) to node[gap] {$\scriptstyle 1$} (f);
\draw[->,bend left] (e) to node[gap] {$\scriptstyle 1$} (f);

\draw[<-,bend right,looseness=0.8] (g) to node[gap] {$\scriptstyle 1$} (f);
\draw[->,bend left,looseness=0.8] (g) to node[gap] {$\scriptstyle 1$} (f);
\draw[<-]  (g) edge [in=-120,out=-55,loop,looseness=8] node[below] {$\scriptstyle 2,2,\dots , 2$} (g);
\draw[<-]  (a) edge [in=120,out=55,loop,looseness=8] node[above] {$\scriptstyle 2,2,\dots , 2$} (a);
\draw[<-]  (b) edge [in=120,out=55,loop,looseness=8] node[above] {$\scriptstyle 2,2,\dots , 2$} (b);
\draw[<-]  (f) edge [in=120,out=55,loop,looseness=7] node[above] {$\scriptstyle 2,2,\dots , 2$} (f);
\end{tikzpicture}
\end{array}
\end{equation}

\begin{notation}\label{notation: grade}
Parallel to the notation \ref{no10}, we adopt the following notation. 
\begin{enumerate}
\item Set $Q_t \colonequals \Hom_{\scrS}(N,N_t(-t))$ for $0 \leq t \leq n$.
\item With respect to \eqref{Q4}, write $Q_{t,d}$ for the degree $d$ graded piece of $Q_t$.  
\item Write $D'_d$ for the vector space dimension of $Q_{0,d}$.
\end{enumerate}
\end{notation}
By the symmetry of \eqref{Q3}, $D'_d$ is also the vector space dimension of $Q_{t,d}$ for $1 \leq t \leq n$.

By \cite{W3}, we have the following exact sequence,
\begin{equation*}
    0 \rightarrow N_0 \xrightarrow[]{(x_0, u)} N_1 \oplus N_n \xrightarrow[]{\begin{psmallmatrix} x_n & -u \\-\frac{x_0x_n}{u} & x_0  \end{psmallmatrix}} N_1 \oplus N_n \xrightarrow[]{\begin{psmallmatrix} \scriptstyle inc \\ \frac{x_{n}}{u} \end{psmallmatrix}} N_0 \rightarrow 0
\end{equation*}
Using the grading in \ref{notation: grade}, the above exact sequence becomes
\begin{equation}\label{102}
    0 \rightarrow N_0 \xrightarrow[]{d_4}  N_1(-1) \oplus N_n(-n) \xrightarrow[]{d_3} N_1(-1) \oplus N_n(-n) \xrightarrow[]{d_2} N_0 \rightarrow 0.
\end{equation}
where each $d_i$ is homogeneous, and further $\deg(d_4)=1=\deg(d_2)$ and $\deg(d_3)=2$.
Applying $\Hom_{\scrS}(N,-)$ to \eqref{102} induces the following exact sequence,
\begin{equation*}
0 \rightarrow  Q_0 \xrightarrow[]{d_4}  Q_1 \oplus Q_n \xrightarrow[]{d_3}  Q_1 \oplus Q_n  \xrightarrow[]{d_2} Q_0 \xrightarrow[]{d_1} \Lambda_{\mathrm{con}} \cong k[x_1, x_2,\dots ,x_{n-1}]  \rightarrow 0,
\end{equation*}
which is parallel to the one in \ref{thm101}. Thus for any integer $d$, there is an exact sequence
\begin{equation*}
0 \rightarrow  Q_{0,d} \rightarrow  Q_{1,d+1} \oplus Q_{n,d+1} \rightarrow  Q_{1,d+3} \oplus Q_{n,d+3}  \rightarrow Q_{0,d+4} \rightarrow T_{d+4}' \rightarrow 0,  
\end{equation*}
where $T_{d+4}'$ denotes the degree $d+4$ piece of $k[x_1, x_2,\dots ,x_{n-1}]$. Thus
\begin{equation}
    D'_d - 2D'_{d+1} + 2D'_{d+3} - D'_{d+4} + E'_{d+4} = 0 \label{107}
\end{equation}
where $E_{d+4}'= \dim_{k}T_{d+4}'$. 

Similarly to \eqref{106}, by definition of the grading we have $Q_{t,d}=0$ for all $t$ and all $d<0$, hence $D'_d=0$ for $d<0$.
Moreover, the exact sequence above and \eqref{107} both hold for $d<0$.

\begin{prop}\label{thm31}
With notation as above, the homomorphism $\psi$ induces an isomorphism $kQ/I \xrightarrow{\sim}  \End_{\scrS}(N)$.
\end{prop}
\begin{proof}
With the notation above, $\psi$ is a graded surjective homomorphism, so it suffices to show that $D_d=D_d'$ for all $d \geq 0$.
Using \eqref{106}, \eqref{107}, and the equality $E_d = E'_d$ for all $d$, together with the fact that
$D_d = 0=D'_d$ for $d<0$, we deduce by induction that $D_d = D'_d$ for all $d \geq 0$.

For completeness, we also verify the cases $0 \leq d \leq 3$ explicitly.
By \ref{prop:diamond} and \ref{prop: irr}, the vector space $P_{0,d}$ admits a $\mathbb{C}$-basis given by irreducible paths as follows:
\begin{enumerate}
    \item For $d=0$, a basis of $P_{0,0}$ is $\{e_0\}$, hence $D_0=1$.
    \item For $d=1$, a basis of $P_{0,1}$ is $\{a_n, b_0\}$, hence $D_1=2$.
    \item For $d=2$, a basis of $P_{0,2}$ is
    \[
    \{a_{n-1}a_n,\; b_1b_0,\; l_{0,0}, l_{0,1},\dots ,l_{0,n}\},
    \]
    hence $D_2=n+3$.
    \item For $d=3$, a basis of $P_{0,3}$ is
    \[
    \{a_{n-2}a_{n-1}a_n,\; b_2b_1b_0,\;
    a_nl_{0,0},\dots ,a_nl_{0,n},\;
    b_0l_{0,0},\dots ,b_0l_{0,n}\},
    \]
    hence $D_3=2n+4$.
\end{enumerate}

For each $0 \leq d \leq 3$, the elements listed above form a $\mathbb{C}$-basis of $P_{0,d}$.
Their images under the correspondence in \eqref{Q3} lie in $Q_{0,d}$ and are linearly independent.
Therefore, $D'_d \geq D_d$ for $0 \leq d \leq 3$.
Since $\psi \colon kQ/I \twoheadrightarrow \End_{\scrS}(N)$ is surjective, it follows that
$D'_d = D_d$ for $0 \leq d \leq 3$.
Combining the base cases with the inductive step completes the proof.
\end{proof}


\begin{thebibliography}{BCHM}


\bibitem[A1]{A1}
V.~I.~Arnold, \emph{Local normal forms of functions}, Invent.\ Math.\ \textbf{35} (1976), 87--109. 

\bibitem[A2]{A2}
J.~August, \emph{The tilting theory of contraction algebras}, Adv.\ Math.\ \textbf{374} (2020), 107372, 56 pp.


\bibitem[B1]{B2}
G.~M. Bergman, \emph{The diamond lemma for ring theory}, Adv. in Math. {\bf 29} (1978), no.~2, 178--218; MR0506890

\bibitem[B2]{W3}
G. Bellamy et al., {\it Noncommutative algebraic geometry}, Mathematical Sciences Research Institute Publications, 64, Cambridge University Press, New York, 2016; MR3560285


\bibitem[BCHM]{BCHM}
C. Birkar et al., \emph{Existence of minimal models for varieties of log general type}, J. Amer. Math. Soc. {\bf 23} (2010), no.~2, 405--468; MR2601039

\bibitem[BKL]{BKL}
J.~A. Bryan, S.~H. Katz and N.~C. Leung, \emph{Multiple covers and the integrality conjecture for rational curves in Calabi-Yau threefolds}, J. Algebraic Geom. {\bf 10} (2001), no.~3, 549--568; MR1832332

\bibitem[BW]{BW2}
G. Brown and M. Wemyss, \emph{Local normal forms of noncommutative functions}, Forum Math. Pi {\bf 13} (2025), Paper No. e8, 59 pp.; MR4865672


\bibitem[CS]{CS}
S. Chouhy and A.~L. Solotar, \emph{Projective resolutions of associative algebras and ambiguities}, J. Algebra {\bf 432} (2015), 22--61; MR3334140

\bibitem[DWZ]{DWZ1}
H.~Derksen, J.~Weyman, and A.~Zelevinsky, \emph{Quivers with potentials and their representations.\ I.\ Mutations.} Selecta Math.\ (N.S.) \textbf{14} (2008), no.~1, 59--119.

\bibitem[DW1]{DW2}
W.~Donovan and M.~Wemyss, \emph{Noncommutative deformations and flops}, Duke Math.\ J.\ \textbf{165} (2016), no.~8, 1397--1474. 

\bibitem[DW2]{DW1}
W.~Donovan and M.~Wemyss, \emph{Contractions and deformations}, Amer.\ J.\ Math.\ \textbf{141} (2019), no.~3, 563--592.

\bibitem[E1]{E3}
D.~Eisenbud, \emph{Commutative Algebra: with a View Toward Algebraic Geometry}, Graduate Texts in Mathematics, Vol.~150,
Springer-Verlag, New York, 1995.


\bibitem[E2]{E1}
K. Erdmann, {\it Blocks of tame representation type and related algebras}, Lecture Notes in Mathematics, 1428, Springer, Berlin, 1990; MR1064107

%\bibitem[E3]{E2}
%K.~Erdmann, \emph{Private communication}, August 2020.

\bibitem[H1]{H1}
R. Hartshorne, \emph{Algebraic geometry}, Vol. 52. Springer Science and Business Media, 2013.

\bibitem[H2]{H}
T. Holm, \emph{Derived equivalent tame blocks}, J. Algebra {\bf 194} (1997), no.~1, 178--200; MR1461486



\bibitem[HW]{HW}
Y. Hirano and M. Wemyss, \emph{Stability conditions for 3-fold flops}, Duke Math. J. {\bf 172} (2023), no.~16, 3105--3173; MR4679958

\bibitem[JKM]{JKM}
G. Jasso, B. Keller and F. Muro, \emph{The Donovan-Wemyss conjecture via the derived Auslander-Iyama correspondence}, in {\it Triangulated categories in representation theory and beyond---the Abel Symposium 2022}, 105--140, Abel Symp., 17, Springer, Cham, ; MR4786504

\bibitem[IW1]{IW2}
O.~Iyama and M.~Wemyss, \emph{Singular Derived Categories of $\Q$-factorial terminalizations and Maximal Modification Algebras},  Adv.\ Math.\ \textbf{261} (2014), 85--121.

\bibitem[IW2]{IW1}
O.~Iyama and M.~Wemyss, \emph{Reduction of triangulated categories and Maximal Modification Algebras for $cA_n$ singularities}, J.\ Reine Angew.\ Math.\ \textbf{738} (2018), 149--202.




\bibitem[MT]{MT}
D. Maulik and Y. Toda, \emph{Gopakumar-Vafa invariants via vanishing cycles}, Invent. Math. {\bf 213} (2018), no.~3, 1017--1097; MR3842061


\bibitem[MY]{MY}
J.~N. Mather and S.~S.~T. Yau, \emph{Classification of isolated hypersurface singularities by their moduli algebras}, Invent. Math. {\bf 69} (1982), no.~2, 243--251; MR0674404

\bibitem[NW]{NW1}
N. Nabijou and M. Wemyss, \emph{GV and GW invariants via the enhanced movable cone}, Moduli {\bf 1} (2024), Paper No. e8, 38 pp.; MR4892594

\bibitem[R]{R}
M. Reid, \emph{Minimal models of canonical $3$-folds}, in {\it Algebraic varieties and analytic varieties (Tokyo, 1981)}, 131--180, Adv. Stud. Pure Math., 1, North-Holland, Amsterdam, ; MR0715649

\bibitem[S]{S}
K.~ Steele, \emph{The K-theory of (compound) Du Val singularities}, \href{https://arxiv.org/abs/2009.05291}{\sf arXiv:2009.05291}.

\bibitem[T]{T2}
Y. Toda, \emph{Non-commutative deformations and Donaldson-Thomas invariants}, in {\it Algebraic geometry: Salt Lake City 2015}, 611--631, Proc. Sympos. Pure Math., 97.1, Amer. Math. Soc., Providence, RI, ; MR3821164

\bibitem[V1]{V1}
M.~Van den Bergh, \emph{Three-dimensional flops and noncommutative rings}, 
Duke Math.\ J.\ \textbf{122}  (2004), no.~3, 423--455.

\bibitem[V2]{VM}
M.~Van den Bergh, \emph{Calabi-Yau algebras and superpotentials}, Selecta Math.\ (N.S.) \textbf{21} (2015), no.~2, 555--603.

\bibitem[V3]{V3}
O.~van Garderen, \emph{Donaldson--Thomas invariants of threefold flops}, PhD thesis, University of Glasgow, 2021.

\bibitem[V4]{VG}
O.~van Garderen, \emph{Vanishing and Symmetries of BPS Invariants for CDV Singularities}, \href{https://arxiv.org/abs/2207.13540}{\sf arXiv:2207.13540}.


\bibitem[W1]{W2}
M. Wemyss, \emph{Reconstruction algebras of type $A$}, Trans. Amer. Math. Soc. {\bf 363} (2011), no.~6, 3101--3132; MR2775800


\bibitem[W2]{W1}
M. Wemyss, \emph{Flops and clusters in the homological minimal model programme}, Invent. Math. {\bf 211} (2018), no.~2, 435--521; MR3748312

\bibitem[Z]{Z}
H.~Zhang, \emph{Gopakumar--Vafa invariants associated to $cA_n$ singularities}, \href{https://arxiv.org/abs/2504.03139}{\sf arXiv:2504.03139}.

\end{thebibliography}
\end{document}